\documentclass[10pt]{article}

\usepackage{amsmath,amssymb,amscd,amsthm,makeidx,txfonts,graphicx,url,bm}
\usepackage{a4wide,pstricks,xypic}
\usepackage[backref]{hyperref}

\numberwithin{equation}{subsection}

\let\oldmarginpar\marginpar
\renewcommand\marginpar[1]{\-\oldmarginpar[\raggedleft\footnotesize #1]
{\raggedright\footnotesize #1}}
\renewcommand{\sharp}{\#}

\makeindex

\xyoption{all}
\input{xypic}

\newtheorem{theorem}{Theorem}[subsection]
\newtheorem{proposition}[theorem]{Proposition}
\newtheorem{corollary}[theorem]{Corollary}
\newtheorem{conjecture}[theorem]{Conjecture}
\newtheorem{lemma}[theorem]{Lemma}

\theoremstyle{remark}
\newtheorem{remark}[theorem]{Remark}

\theoremstyle{definition}
\newtheorem{definition}[theorem]{Definition}

\newtheorem{nothing}[theorem]{}

\newcounter{margin}
{\end{itshape}  \bigskip}

\def\beq{\begin{eqnarray}}
\def\eeq{\end{eqnarray}}
\def\bes{\begin{eqnarray*}}
\def\ees{\end{eqnarray*}}

\def\muhat{{\bm \mu}}
\def\xihat{{\bm \xi}}
\def\lambdahat{{\bm \lambda}}

\def\nuhat{{\bm \nu}}

\DeclareMathOperator{\Aut}{Aut} \DeclareMathOperator{\charx}{char}
\DeclareMathOperator{\Spec}{Spec} \DeclareMathOperator{\Hom}{Hom}
\DeclareMathOperator{\Tr}{Tr} \DeclareMathOperator{\Ind}{Ind}
\DeclareMathOperator{\Res}{Res}

\def\Lam{{\mathcal L}}

\def\A{{\bf A}}

\def\C{\mathbb{C}}
\def\M{{\mathcal{M}}}
\def\calQ{{\mathcal{Q}}}
\def\calA{{\mathcal{A}}}
\def\calI{{\mathcal{I}}}
\def\calX{{\mathcal{X}}}
\def\calU{{\mathcal{U}}}
\def\calV{{\mathcal{V}}}
\def\calF{{\mathcal{F}}}

\def\calP{\mathcal{P}}
\def\calH{\mathcal{H}}

\def\n{{\mathbf{n}}}

\def\x{\mathbf{x}}

\def\y{\mathbf{y}}

\def\v{\mathbf{v}}

\def\P{\mathcal{P}}
\def\H{\mathbb{H}}
\def\N{\mathbb{Z}_{\geq 0}}
\def\Nstar{\mathbb{Z}_{> 0}}
\def\R{\mathbb{R}}
\def\F{\mathbb{F}}
\def\Q{\mathbb{Q}}
\def\T{\mathbb{T}}

\def\calC{{\mathcal C}}
\def\calO{{\mathcal O}}

\def\Z{\mathbb{Z}}

\def\Gm{\mathbb{G}_m}
\def\K{\mathbb{K}}

\def\gl{{\mathfrak g\mathfrak l}}

\newcommand{\p}{\mathcal P}
\newcommand{\proj}{\mathbb P}

\newcommand{\nc}{\newcommand}

\def\G{{\rm G}}

\nc{\op}[1]{\mathop{\mathchoice{\mbox{\rm #1}}{\mbox{\rm #1}}
{\mbox{\rm \scriptsize #1}}{\mbox{\rm \tiny #1}}}\nolimits}
\nc{\al}{\alpha}

\nc{\ep}{\varepsilon} \nc{\ga}{\gamma} \nc{\Ga}{\Gamma}
\nc{\la}{\lambda} \nc{\La}{\Lambda} \nc{\si}{\sigma}
\nc{\Sig}{{\Gamma}} \nc{\Om}{\Omega} \nc{\om}{\omega}

\nc{\SL}{{\rm SL}} \nc{\GL}{{\rm GL}} \nc{\PGL}{{\rm PGL}}

\def\U{{\mathcal{U}}}

\nc{\cpt}{{\op{cpt}}} \nc{\Dol}{{\op{Dol}}} \nc{\DR}{{\op{DR}}}
\nc{\B}{{\op{B}}} \nc{\Triv}{\op{Triv}} \nc{\Hod}{{\op{Hod}}}
\nc{\Log}{{\op{Log}}} \nc{\Exp}{{\op{Exp}}} \nc{\Est}{E_{\op{st}}}
\nc{\Hst}{H_{\op{st}}} \nc{\Left}[1]{\hbox{$\left#1\vbox to
  10.5pt{}\right.\nulldelimiterspace=0pt \mathsurround=0pt$}}
\nc{\Right}[1]{\hbox{$\left.\vbox to
  10.5pt{}\right#1\nulldelimiterspace=0pt \mathsurround=0pt$}}
\nc{\LEFT}[1]{\hbox{$\left#1\vbox to
  15.5pt{}\right.\nulldelimiterspace=0pt \mathsurround=0pt$}}
\nc{\RIGHT}[1]{\hbox{$\left.\vbox to
  15.5pt{}\right#1\nulldelimiterspace=0pt \mathsurround=0pt$}}

\nc{\bee}{{\bf E}} \nc{\bphi}{{\bf \Phi}}

\begin{document}

\title{Arithmetic harmonic analysis on \\ character and quiver varieties}

\author{ Tam\'as Hausel
\\ {\it University of Oxford}
\\{\tt hausel@maths.ox.ac.uk} \and Emmanuel Letellier \\ {\it
  Universit\'e de Caen} \\ {\tt
  letellier.emmanuel@math.unicaen.fr}\and Fernando Rodriguez-Villegas
\\ 
{\it University of Texas at Austin} \\ {\tt
  villegas@math.utexas.edu}\\ \\
 }

\pagestyle{myheadings}

\maketitle

\begin{abstract} We propose a general conjecture for the mixed Hodge
  polynomial of the generic character varieties of representations of
  the fundamental group of a Riemann surface of genus~$g$ to
  $\GL_n(\C)$ with fixed generic semi-simple conjugacy classes at $k$
  punctures. This conjecture generalizes the Cauchy identity for
  Macdonald polynomials and is a common generalization of two formulas
  that we prove in this paper.  The first is a formula for the
  $E$-polynomial of these character varieties which we obtain using
  the character table of $\GL_n(\F_q)$. We use this formula to compute 
  the Euler characteristic of character varieties. The second formula
 gives the Poincar\'e polynomial of certain associated quiver varieties which
  we obtain using the character table of $\gl_n(\F_q)$.  In the last
  main result we prove that the Poincar\'e polynomials of the quiver
  varieties equal certain multiplicities in the tensor product of
  irreducible characters of $\GL_n(\F_q)$. As a consequence we find a
  curious connection between Kac-Moody algebras associated with
  comet-shaped, typically wild, quivers and the representation theory
  of $\GL_n(\F_q)$.
\end{abstract}

\newpage
\tableofcontents
\newpage

\section{Introduction}

\subsection{Cauchy identity for Macdonald polynomials}

Let $\x=\{x_{1},x_{2},\dots\}$ and $\y=\{y_{1},y_{2},\dots\}$ be two
infinite sets of variables and $\Lambda(\x)$ and $\Lambda(\y)$ be
the corresponding rings of symmetric functions.  For a partition
$\lambda$ let  $\tilde{H}_\lambda(\x;q,t) \in \Lambda(\x)
\otimes_\Z \Q(q,t)$ be the {\it Macdonald symmetric function} defined
in \cite[I.11]{garsia-haiman}.
These functions satisfy the Cauchy identity (in a  form equivalent to
\cite[Theorem 3.3]{garsia-haiman})
\beq
\label{Cauchy}
\Exp\left(\frac{m_{(1)}(\x)m_{(1)}(\y)}{(q-1)(1-t)} \right)=
\sum_{\lambda\in \P}
\frac{\tilde{H}_\lambda(\x;q,t)\tilde{H}_\lambda(\y;q,t)}
{\prod (q^{a+1}-t^{l})(q^{a}-t^{l+1})}
\eeq
where
$\Exp$ is the plethystic exponential  (see, for example,  \cite[\S
2.5]{hausel-villegas}; we recall the formalism of $\Exp$ and its
inverse $\Log$ in~\S\ref{generating}), $\P$ is the set of all
partitions, $m_\lambda \in \Lambda$ are the monomial symmetric functions
and the product in the denominator on the right hand side is over the
cells of $\lambda$ with $a$ and $l$ their arm and leg lengths, respectively.

In this paper we will think of \eqref{Cauchy} as the special case
$g=0,k=2$ of a formula pertaining to a genus~$g$ Riemann surface with $k$
punctures. Fix integers $g\geq 0$ and $k>0$. Let
$\x_1=\{x_{1,1},x_{1,2},\dots\}, \dots,
\x_k=\{x_{k,1},x_{k,2},\dots\}$ be $k$  sets of infinitely many independent
variables and let $\Lambda(\x_1,\ldots,\x_k)$ be the ring of functions
separately symmetric in each of the set of variables. When there is no
risk of confusion of what variables are involved we will simply write
$\Lambda$  for this ring.

Define the {\it $k$-point genus $g$ Cauchy function} (throughout the
paper $k$ will denote a positive integer)
 \beq
\label{cauchygk}
\Omega(z,w):= \sum_{{\lambda}\in {\cal P}}
{\cal H}_{\lambda}(z,w)\prod_{i=1}^k \tilde{H}_\lambda(\x_i;z^2,w^2),
\eeq
with coefficients in $\Q(z,w)\otimes_\Z    \Lambda$,
where
$$
{\cal H}_{\lambda} (z,w):=\prod
\frac{(z^{2a+1}-w^{2l+1})^{2g}} {(z^{2a+2}-w^{2l})(z^{2a}-w^{2l+2})}
$$
is a $(z,w)$-deformation of the $(2g-2)$-th power of the standard hook
polynomial. Thus in particular $\Omega(\sqrt{q},\sqrt{t})$ equals the
right hand side of \eqref{Cauchy} for $g=0,k=2$.

For $\muhat=(\mu^1,\ldots,\mu^k)\in \calP^k$ let
\beq\label{introduceH} \H_{\muhat}(z,w):= (z^2-1)(1-w^2)\left\langle
 \Log\; \Omega(z,w) ,h_\muhat\right\rangle.  \eeq Here
$h_\muhat:=h_{\mu^1}(\x_1)\cdots h_{\mu^k}(\x_k)\in \Lambda$ are the
complete symmetric functions and $\langle\cdot,\cdot\rangle$ is the
extended Hall pairing defined in \eqref{extendedhall}. Recall that
$\{m_\lambda\}$ and $\{h_\lambda\}$ are dual bases with respect to the
Hall pairing and we may hence recover $\Omega(z,w)$ from the
$\H_{\muhat}(z,w)$'s by the formula:
$$
\Omega(z,w)=\Exp\left(\sum_{\muhat \in \P^k}
\frac{\H_{\muhat}(z,w)}
{(z^2-1)(1-w^2)}m_{\muhat}\right).
$$
Note that $\H_\muhat=0$ unless $|\mu^1|=\cdots=|\mu^k|$.

With this notation \eqref{Cauchy} is equivalent to
\beq
\label{g=0k=2}\H_{\muhat}(z,w)=\left\{
\begin{array}{ll}1 & \mbox{ if $\muhat=((1),(1))$} \\ 0 &
\mbox{otherwise}\end{array}\right. .
\eeq when $g=0$ and $k=2$.

\subsection{Character varieties}
\label{intro-char-var}
Fix $\muhat=(\mu^1,\dots,\mu^k) \in {\P_n}^k$ for the rest of this
introduction where $\mu^i=(\mu^i_1,\mu^i_2,\ldots,\mu^i_{r_i})$ and
$r_i:=\ell(\mu^i)$ is the length of $\mu^i$ ($\P_n$ denotes the set of
partitions of $n$).  Let $\M_\muhat$ be a $\GL_n(\C)$ character
variety of a $k$-punctured genus $g$ Riemann surface, with generic
semi-simple conjugacy classes of type $\muhat$ at the punctures. In
other words, fix semisimple conjugacy classes
$\calC_1,\dots,\calC_k\subset \GL_n(\C)$, which are generic in the
sense of Definition~\ref{genericconjugacy} and have type
$\mu^1,\ldots, \mu^k$; i.e., $\{\mu^i_1,\mu^i_2,\ldots\}$ are the
multiplicities of the eigenvalues of any matrix in $\calC_i$. (We
prove in Lemma~\ref{exists} that there always exist generic semisimple
conjugacy classes for every $\muhat$.) The variety depends on the
actual choice of eigenvalues  but for simplicity we drop this choice from
the notation.

Concretely, \bes \M_{\muhat}:= &\{ A_1,B_1,\dots,A_g,B_g \in
\GL_n(\C), X_1\in \calC_1,\dots, X_k\in \calC_k |\\ & (A_1,B_1) \cdots
(A_g,B_g) X_1\cdots X_k=I_n\}/\!/ \GL_n(\C), \ees an affine GIT
quotient by the conjugation action of $\GL_n(\C)$, where for two
matrices $A,B\in \GL_n(\C)$, we put $(A,B)=ABA^{-1}B^{-1}$ and $I_n$
is the identity matrix.  We prove in Theorem~\ref{smoothcharacter}
that $\M_\muhat$, if non-empty, is a non-singular variety of dimension
\beq
\label{dimension}
d_\muhat:=n^2(2g-2+k)-\sum_{i,j}(\mu_j^i)^2 +2.  
\eeq 
For example, if $k=1$ and $\muhat=((n))$ then $\M_\muhat$ is just the
variety $\M_n$ of \cite{hausel-villegas} and $\H_\muhat$ is the
polynomial $\bar H_n$ (see~\S\ref{example-2} for more
details).

\subsubsection{Mixed Hodge polynomial: The
  conjectures} 

As a natural continuation of \cite{hausel-villegas} here we study the
compactly supported {\it mixed Hodge polynomials}
$$
H_c(\M_\muhat;x,y,t):=\sum h_c^{i,j;k}(\M_\muhat)
x^iy^jt^k,
$$
where $h_c^{i,j;k}(\M_\muhat)$ are the compactly supported  mixed Hodge
numbers of \cite{Del1,Del2}.  For any variety $X/\C$ the polynomial
$H_c(X;x,y,t)$ is a common deformation of its compactly supported {\it
  Poincar\'e polynomial} $P_c(X;t)=H_c(X;1,1,t)$ and its so-called
{\it $E$-polynomial} $E(X;x,y)=H_c(X;x,y,-1)$. 

We define the {\it pure part} of $H_c$ as the polynomial
$$
PH_c(X;x,y):=\sum_{i,j}h_c^{i,j;i+j}(X)x^iy^j.
$$
If $h_c^{i,j;k}(X)=0$ unless $i=j$ we will simplify the notation and
write $H_c(X;q,t):=H_c(X;\sqrt{q},\sqrt{q},t), PH_c(X;q):=PH_c(X;\sqrt
q, \sqrt q)$ and $E(X;q):=E(X;\sqrt{q},\sqrt{q})$.

\begin{conjecture}
 \label{main} 

\noindent
(i) The rational function $\H_\muhat(z,w)$ defined in
(\ref{introduceH}) is a polynomial. It has degree~$d_\muhat$ in each
variable and $\H_\muhat(-z,w)$ has non-negative integer coefficients.

\noindent
(ii) The mixed Hodge polynomial $H_c(\M_\muhat;x,y,t)$ is a polynomial
in $xy$ and $t$ and is independent of the choice of generic
eigenvalues of multiplicities $\muhat$.

\noindent
(iii) Moreover\footnote{Warning: our use of the variables
  $q,t$ in the Hodge polynomial context is different from the standard
  one in the theory of Macdonald polynomials. It should always be
  clear from the context which is in use.},
$$
H_c(\M_\muhat; q,t)=(t\sqrt q)^{d_\muhat}\;
\H_\muhat\left(-{\frac 1{\sqrt q},t\sqrt q }
\right).
$$
(iv) In particular, the pure part of $H_c(\M_\muhat;q,t)$ is
$$
PH_c(\M_\muhat;q)= q^{d_\muhat/2}\H_\muhat(0,\sqrt q).
$$
\end{conjecture}
\noindent
In this paper we will present several consistency checks and prove
several implications of this conjecture.  For example, we show
in~\S\ref{indep-eigenv} that although $\M_\muhat$ itself depends on
the choice of eigenvalues $H_c(\M_\muhat;x,y,t)$ is constant on a
dense subset (in the analytic topology) of generic eigenvalues of
multiplicities $\muhat$. This is consistent with (ii) of
Conjecture~\ref{main}.

Due to the known symmetry
$\tilde{H}_\lambda(\x_i;q,t)=\tilde{H}_{\lambda^\prime}(\x_i;t,q)$ of
Macdonald polynomials \eqref{Hduality}, the right hand side of
\eqref{introduceH} is invariant both under changing $(z,w)$ to $(w,z)$
and under changing $(z,w)$ to $(-z,-w)$. Hence the same holds for
$\H_\mu(z,w)$ and Conjecture~\ref{main} implies

\begin{conjecture}[Curious Poincar\'e
 Duality]
\label{cpd}
$$
H_c\left(\M_\mu;\frac{1}{qt^2},t\right)=(qt)^{-{d_\mu}}H_c(\M_\mu;q,t)
$$
\end{conjecture}

\subsubsection{$E$-polynomial}

\begin{theorem}
 \label{character}
 The polynomial $E(\M_\muhat;x,y)$ depends only on $xy$ and 
$$
E(\M_\muhat;q)=q^{\tfrac 12 d_\muhat}\;
\H_{\muhat}\left(\sqrt{q},\frac1 {\sqrt q} \right)
$$
\label{main1}\end{theorem}
\noindent
In other words, the conjecture (\ref{main}) is true under the
specialization $(q,t)\mapsto (q,-1)$.  We prove this
in~\S\ref{epolychar}.

The calculation of $E(\M_\muhat;q)$ follows the same route as in
\cite{hausel-villegas}.  We prove that $\M_\muhat$ is polynomial
count and hence by Katz's theorem \cite[Theorem 6.1.2.3]{hausel-villegas}
$E(\M_\muhat;q)=\#\{\M_\muhat(\F_q)\}$.  To count the
points of $\M_\muhat$ over a finite field we use the mass
formula
\begin{equation}
\label{forchar}
\#\{\M_\muhat(\F_q)\}=\sum_{\calX\in\text{Irr}\,(\GL_n(\F_q))}
\frac{|\GL_n(\F_q)|^{2g-2}(q-1)}{\calX(1)^{2g-2}}
\prod_i\frac{\calX(C_i)}{\calX(1)}|C_i|
\end{equation}
originally due to Frobenius \cite{frobenius} for $g=0$. The
evaluation of the formula is facilitated by the combinatorial
understanding of the character table of $\GL_n(\F_q)$ first obtained
in \cite{green}.

\begin{corollary}
 The $E$-polynomial is palindromic, i.e., it satisfies
$$E(\M_\muhat;q)=q^{d_\muhat}E(\M_\muhat;q^{-1}).$$
\label{coromain}
\end{corollary}
\noindent
In a forthcoming paper \cite{hausel-letellier-villegas2} we use our
formula~\eqref{character} for its $E$-polynomial to prove that
$\M_\muhat$ is connected (as announced in
\cite{hausel-letellier-villegas3}).

\subsubsection{Euler characteristic}
The $2g$-dimensional torus $(\C^{\times})^{2g}$ acts on $\M_\muhat$ by
scalar multiplication on the first $2g$-coordinates. We let
$\tilde{\M}_\muhat:=\M_\muhat/\!/(\C^{\times})^{2g}$. As a second
application of Theorem \ref{main1}, we compute the Euler
characteristic $E(\tilde{\M}_\muhat):=E(\tilde{\M}_\muhat;1)$ of
$\tilde{\M}_\muhat$ when $g>0$, using that
$E(\tilde{\M}_\muhat)=E(\M_\muhat)/(q-1)^{2g}$ (see~\S\ref{Euler}).
We obtain the following.
\begin{theorem} Assume that $g>1$, then
$$
E(\tilde{\M}_\muhat)=
\begin{cases}
\mu(n)\,n^{2g-3} & \text{ if }
\muhat=((n),\ldots ,(n))\\
0& \text{ otherwise}
\end{cases}
$$
where $\mu$ is the ordinary M\"obius function.

\end{theorem}

\begin{theorem} 
\label{euler-char-g=1}
For $g=1$, 
$$
E(\tilde{\M}_\muhat)= \frac 1n\sum_{d\mid \gcd(\mu_i^j)}
\sigma(n/d)\mu(d)\frac{\left((n/d)!\right)^k}{\prod_{i,j}(\mu_i^j/d)!}.
$$
where $\sigma(m)=\sum_{d \mid m}d$.
 \end{theorem}
\noindent
For the proofs of these theorems see~\S\ref{Euler}.

\subsection{Quiver varieties}

For $i=1,\dots,k$ let $\calO_i\subset \gl_n(\C)$ be a semisimple
adjoint orbit in the Lie algebra $\gl_n(\C)$ of type $\mu^i$; as
before, this means that $\{\mu^i_1,\mu^i_2,\ldots\}$ are the
multiplicities of the eigenvalues of any matrix in $\calO_i$. We will
call the collection $(\calO_1,\dots,\calO_k)$ {\it generic} if certain
linear equations among the eigenvalues of the conjugacy classes are
not satisfied (see Definition~\ref{genericadjoint}). There exists a
generic collection of conjugacy classes of type $\muhat$ if and only
if $\muhat$ is indivisible (i.e. $\gcd(\{\mu^i_j\})=1)$. For a generic
$(\calO_1,\dots,\calO_k)$ we define \bes \calQ_{\muhat}:= &\{
A_1,B_1,\dots,A_g,B_g \in \gl_n(\C), C_1\in \calO_1,\dots,C_k\in
\calO_k |\\ & [A_1,B_1]+ \cdots +[A_g,B_g] +C_1\cdots +C_k=0\}/\!/
\GL_n(\C), \ees an affine GIT quotient by the conjugation action of
$\GL_n(\C)$, where $[\cdot,\cdot]$ is the Lie bracket in
$\gl_n(\C)$. We prove in Theorem~\ref{smoothquiver} that
$\calQ_\muhat$ is a smooth variety of dimension $d_{\muhat}$. It is a
quiver variety in the sense of Nakajima and Crawley-Boevey associated
to the comet-shaped quiver $\Gamma$ described in \S\ref{subquiver}.

\begin{theorem}
\label{quiver}
For $\muhat$ indivisible the mixed Hodge structure on
$H^*_c(\calQ_\muhat)$ is pure, in other words,
$h^{i,j;k}(\calQ_\muhat)=0$ unless $i+j=k$, and  $E(\calQ_{\muhat};x,y)$
only depends on the product $xy$. Moreover,
\beq
\label{quiverformula}
P_c(\calQ_\muhat;\sqrt{q})= E(\calQ_{\muhat};q)=q^{\tfrac
12 d_\muhat}\;  \H_{\muhat}\left(0,\sqrt{q}\right),
\eeq
where  $P_c(\calQ_\muhat,t)$ is the compactly  supported Poincar\'e
polynomial of $\calQ_\muhat$.
\end{theorem}
As in the multiplicative case, Katz's theorem \cite[Theorem
6.1.2.3]{hausel-villegas} implies that
$E(\calQ_\muhat;q)=\#\{\calQ_\muhat(\mathbb{F}_q)\}$. The calculation
of the number of points on the right is performed using the mass
formula
\begin{equation}
\#\{\calQ_{\muhat}(\F_q)\}=
\frac{ |\gl_n(\mathbb{F}_q)|^{g-1}}{|\PGL_n(\mathbb{F}_q)|}
\sum_{x\in \gl_n(\mathbb{F}_q)}|C_{\gl_n(\mathbb{F}_q)}(x)|^g
\prod_{i=1}^k\mathcal{F}^{\mathfrak{g}}(1_{\mathcal{O}_i})(x),
\label{forquiver}
\end{equation}
where $C_{\gl_n(\mathbb{F}_q)}(x)$ denotes the centralizer of $x$ in
$\gl_n(\mathbb{F}_q)$ and
$\mathcal{F}^{\mathfrak{g}}(1_{\mathcal{O}_i})$ is the Fourier
transform \eqref{Lie-alg-Fourier-transf} of the characteristic
function of $\calO_i$. The evaluation of this sum is based on a
combinatorial understanding of the formulas in \cite{letellier} in the
case of $\gl_n(\F_q)$. The proof of~\ref{quiver} is given
in~\S\ref{poincare-quiver}. 

\begin{remark}
\label{pure-part}
The {\em purity conjecture} of \cite{hausel2} claims that
$PH_c(\M_\muhat;q)=E(\calQ_{\muhat};q)$. Combined with
Conjecture~\ref{main} it implies that the right hand side of
\eqref{quiverformula} should equal $PH_c(\M_\muhat;q)$. By extension,
we call the {\it pure part} of a function of $z,w$ be its
specialization $z=0,w=\sqrt q$. For example, the pure part of the
Macdonald polynomial is
$\tilde{H}_\lambda(\x;w):=\tilde{H}_\lambda(\x;0,w)$ a (transformed
version of) the Hall-Littlewood polynomial
(see~\S\ref{macdonald-hall-littlewood}). In particular,
Theorem~\ref{quiver} shows that the $E$-polynomials of the quiver
varieties $\calQ_\muhat$ are closely related to the generalized Cauchy
formula for Hall-Littlewood functions.
\end{remark}

\begin{remark}
\label{kac-1}
Let $A_\v(q)$ be the number of absolutely indecomposable
representations of a quiver of dimension $\v$ over the finite field
$\F_q$ (up to isomorphism). Kac~\cite{kacconj} proved that $A_\v(q)$
is a polynomial in $q$ with integer coefficients. He conjectured that
these coefficients are non-negative \cite[Conjecture
2]{kacconj}. Crawley-Boevey and Van den Bergh proved
\cite{crawley-boevey-etal} this conjecture for $\v$ indivisible by
giving a cohomological interpretation for $A_\v(q)$. In our case,
writing $A_\muhat$ for $A_\v$, their result says that for $\muhat$
indivisible
\begin{equation}
\label{A=E}
E(\calQ_{\muhat};q)=\#\{\calQ_\muhat(\mathbb{F}_q)\}=q^{\tfrac 12
  d_\muhat} A_{\muhat}(q). 
\end{equation}
In particular, for $\muhat$ indivisible, $A_\muhat$ is the pure part
of $\H_\muhat$; i.e.,
\begin{equation}
\label{A=H}
A_{\muhat}(q)=\H_{\muhat}(0,\sqrt{q}).
\end{equation}
In fact, we give in \cite{hausel-letellier-villegas2} an independent
proof of~\eqref{A=H} for {\em any} $\muhat$ using Hua's
formula~\cite{hua}.  Conjecture~\ref{main} would then give a
cohomological interpretation of $A_{\muhat}(q)$ for a comet-shaped
quiver, which in particular, would imply Kac's conjecture on the
non-negativity of the coefficients of $A_\muhat(q)$ for such quivers
for all dimension vectors. (See also remark~\ref{kac-conj-comet}).
\end{remark}

\subsection{Multiplicities}\label{intromulti}

For our third main theorem we need to introduce some
complex irreducible characters of $G:=\GL_n(\F_q)$.
Pick distinct linear characters
$\alpha_{i,1},\ldots,\alpha_{i,r_i}$ of $\F_q^\times$ for
each $i$.
Consider the subgroup
$L_i:=\prod_{j=1}^{r_i}\GL_{\mu^i_j}(\F_q)$ of
$G$ and the  linear character
$\tilde{\alpha_i}:=\prod_{j=1}^{r_i}(\alpha_{i,j}\circ
\text{det})$  of $L_i$. We get an irreducible
character of $G$ by taking the Harish-Chandra induction
$R_{L_i}^G(\tilde{\alpha}_i)$.
We assume now that the $\alpha_{i,j}$'s are chosen such that the
$k$-tuple
$\left(R_{L_1}^G(\tilde{\alpha}_1),\ldots,R_{L_k}^G(\tilde{\alpha}_k)\right)$
is  \emph{generic} in the sense of Definition \ref{genericcondi}
(such a choice is always possible for every $\muhat$ assuming that ${\rm char}\,(\F_q)$ and $q$ are large enough). To simplify the
notation we  let
$$
R_\muhat:=\bigotimes_{i=1}^kR_{L_i}^G(\tilde{\alpha}_i).
$$
Let $\Lambda:G \rightarrow \C$ be defined by $x\mapsto q^{g\dim
  C_G(x)}$, where $C_G(x)$ is the centralizer of $x$ in $G$. If $g=1$,
it is the character of the permutation representation where $G$ acts
on the finite set $\gl_n(\F_q)$ by conjugation.

\begin{theorem}
\label{thm-multiplicities}
The following identity holds
\beq
\H_{\muhat}(0,\sqrt{q})=\left\langle
\Lambda\otimes R_\muhat,
1\right\rangle
\eeq
where $\langle\,,\,\rangle$ is the
usual scalar product of characters.
\end{theorem}
\noindent
The proof of this theorem can be found in~\S\ref{multiplicities}. 

For a finite group $H$, let $\mathcal{R}(H)$ be the character ring of
$H$ (i.e., the Grothendieck ring of the category of
$\C[G]$-modules). The irreducible characters $\calX_1,\dots,\calX_k$
form a natural basis $\mathcal{B}$ of $\mathcal{R}_H$. It is an
important and difficult problem to compute the fusion rules of
$\mathcal{R}(H)$ with respect to $\mathcal{B}$, i.e., to compute
$N_{i,j}^r:=\langle \calX_i\otimes\calX_j,\calX_r\rangle$ for all
$i,j,r$.  The character ring of $\GL_n(\F_q)$ does not seem to have
been studied in the literature although the character table of
$\GL_n(\F_q)$ was computed fifty years ago \cite{green}. Our Theorem
\ref{thm-multiplicities} (with $g=0$ and $k=3$) gives a formula for
the multiplicities $N_{i,j}^r$ when $(\calX_i,\calX_j,\calX_r)$ is a
generic triple of semisimple irreducible characters.  This suggest an
interesting connection between the character ring of $\GL_n(\F_q)$,
Kac-Moody algebras and quiver representations that we discuss further
in~\S\ref{multiplicities}.

By Formulas (\ref{quiverformula}) and (\ref{thm-multiplicities}) we have: 

\begin{corollary} For $\muhat$  indivisible the following are
equivalent:

\noindent a) $\left\langle \Lambda\otimes
R_\muhat, 1\right\rangle =
0$.

\noindent b) The quiver variety $\calQ_{\muhat}$ is
empty.
\end{corollary}
In the genus $g=0$ case, the problem of deciding whether
$\calQ_{\muhat}$ is empty was solved by Kostov
\cite{kostov}\cite{kostov2}.  Later on, Crawley-Boevey
\cite{crawley-mat} reformulated Kostov's answer in terms of
roots. Namely he proved that $\calQ_{\muhat}$ is non-empty if and only
if $\v$, the dimension vector for $\Gamma$ with dimension
$n-\sum_{j=1}^l \mu^i_j$ at the $l$-th vertex on the $i$-th leg, is a
root of the Kac-Moody algebra associated to $\Gamma$.

\begin{remark}
\label{kac-conj-comet}
Combining~\eqref{A=E} with Theorems~\ref{quiver}
and~\ref{thm-multiplicities} we find that
 \beq
\label{purea-intro}
A_{\muhat}(q)=\H_{\muhat}(0,\sqrt{q})=\left\langle \Lambda\otimes
  R_\muhat, 1\right\rangle \eeq when $\muhat$ is indivisible. In
\cite{hausel-letellier-villegas2} we prove the equality
\eqref{purea-intro} for {\em any} $\muhat$. Assuming
Conjecture~\ref{main}, this gives a cohomological interpretation of
$\left\langle \Lambda\otimes R_\muhat, 1\right\rangle$. (See also
remark~\ref{kac-1}.)
\end{remark}

\subsection{Examples}
\label{examples}
When the associated comet-shaped quiver (see~\S\ref{subquiver} for a
description) is finite or tame our main conjecture
(Conjecture~\ref{main}) reduces to purely combinatorial formulas, some
of which are known. We illustrate this in a few examples.

\subsubsection{Cases related to Garsia-Haiman's formulae}

 For $g=0$ and $k=1$ (resp. $k=2$) we have
$$
\M_\muhat :=\left\{ \begin{array}{ll}
\mbox{point} &\qquad  \mbox{if $\muhat=(1)$}\,\,
\mbox{(resp. $\muhat=((1),(1))$)}\\ 
\emptyset & \qquad \mbox{otherwise} .\end{array} \right., $$ 

Hence for $g=0$ and $k=1$ the formula (cf. \cite[Corollary
3.3]{garsia-haiman}) 
$$
\Exp\left(\frac{m_{(1)}(\x)}{(q-1)(1-t)}\right) =
\sum_\lambda\frac{\tilde{H}_\lambda(\x;q,t)}
{\prod(q^{a+1}-t^l)(q^a-t^{l+1})}
$$
implies Conjecture \ref{main}, and for $g=0$ and $k=2$, the conjecture
follows from the Cauchy formula \eqref{Cauchy}, or equivalently
\eqref{g=0k=2}.

\subsubsection{Comet-shaped quivers with $k=1$ and $l(\mu^1)=1$}
\label{example-2}
As mentioned at the end of \S\ref{intro-char-var}, in this case we
have $\M_\muhat=\M_n$ and $\H_\muhat=\bar H_n$ in the notation
of~\cite{hausel-villegas}.  The point is that if we are only
interested in partitions $\mu^i$ of length at most $l$ we can, without
loss of generality, set all variables $x^i_j$ with $j>l$ to zero
(see~\S\ref{cauchy-fctns}). If $l=1$ this means we may specialize to
$\x^i=(T,0,\ldots)$ for some variable $T$. Since $\tilde
H_\lambda(T,0,\ldots)=T^{|\lambda|}$ we see that $\Omega(z,w)$
specializes to the corresponding series (left hand side of (3.5.8))
in~\cite{hausel-villegas}.

If in addition $g=1$ therefore,
Conjecture~\ref{main} reduces to the following purely combinatorial
identity of generating functions \cite[Conjecture
4.3.2]{hausel-villegas}
\begin{conjecture}
 \begin{equation}
   \label{g=1-conj}
\sum_\lambda
\prod
\frac{\left(z^{2a+1}-w^{2l+1}\right)^2}
{(z^{2a+2}-w^{2l})(z^{2a}-w^{2l+2})}\,T^{|\lambda|}=
\prod_{n\geq 1}\prod_{r>0}\prod_{s\geq 0}
\frac{(1-z^{2s+1}w^{-2r+1}\,T^n)^2}
{(1-z^{2s}w^{-2r+2}\,T^n)(1-z^{2s+2}w^{-2r}\,T^n)}.
 \end{equation}
\end{conjecture}
\noindent
The associated quiver is the Jordan quiver (one loop, one node), which
is tame. We know that the Euler specialization $z=\sqrt q,w=1/\sqrt q$
of~\eqref{g=1-conj} is true; after taking $\Log$'s it amounts to the
following easy facts
$$
\sum_{\lambda \in \P} T^{|\lambda|}=\prod_{n\geq 1}(1-T^n)^{-1},
\qquad \qquad
\sum_{r>0}\sum_{s\geq 0}q^{r+s}=(q+q^{-1}-2)^{-1}.
$$

\subsubsection{Star-shaped quiver with $k$ legs and  $l(\mu^i)\leq 2$}

 Consider the quiver consisting of one central node with no loops
($g=0$) and $k$ legs of length $1$.  It is enough
(see~\S\ref{cauchy-fctns}) to consider partitions $\mu^i$ of length at
most $2$ and restrict to these by specializing the variables
$\x_i=(x_{i,1},x_{i,2},\ldots)$ to, say, $u_0^{1/k}(1,u_i,0,0,\cdots)$
for $i=1,\ldots k$ for some new independent variables~$u_i$. The
variable $u_0$, corresponding to the central node, keeps track of the
degree of the symmetric functions.  Multipartitions
$\muhat=(\mu^1,\ldots,\mu^k)\in \P_n^k$ with $l(\mu^i)\leq 2$ are of
the form $\mu^i=(n-n_i,n_i)$ for some $0\leq n_i\leq n/2$ for
$i=1,\ldots,k$. To simplify somewhat the notation, we extend by
symmetry the definition of $\H_\muhat$ to arbitrary pairs
$\mu^i=(n-n_i,n_i)$ with $0\leq n_i\leq n$ for $i=1,2,\ldots,k$. For
$\v=(n,n_1,\ldots,n_k)$ we let $\H_\v=\H_\muhat$, where $\muhat$ is
the corresponding multipartition obtained by appropriate permutation
of the entries of each pair $(n-n_i,n_i)$.  It then follows easily
from the definition that
\begin{equation}
\label{main-examples}
\sum_\v \H_\v(z,w)\, u^\v
=
(z^2-1)(1-w^2)\,\Log\left( 
  \sum_{{\lambda}\in {\cal P}}
  \frac{\prod_{i=1}^k\tilde{H}_\lambda(1,u_i,0,\ldots;z^2,w^2)} {\prod
    (z^{2a+2}-w^{2l})(z^{2a}-w^{2l+2})}\,u_0^{|\lambda|}\right),
\end{equation}
where the sum is over all non-zero $\v=(n,n_1,\ldots,n_k)$ with $0\leq
n_i\leq n$ and $u^\v:=u_0^nu_1^{n_1}\cdots u_k^{n_k}$.
\begin{remark}
\label{g=0-remark-1}
Note that since $g=0$ (see Remark~\ref{g=0-remark}) the right hand
side, and hence $\H_\v(z,w)$, are actually functions of $z^2$ and
$w^2$; we will exploit this below without further comment.
\end{remark}

By work of Craweley-Boevey~\cite{crawley-par} the character variety
$\M_\muhat$ is empty unless $\v$ is a root of the associated
Kac--Moody root system. For $1\leq k \leq 3$ the corresponding quiver is
finite. In particular, the main conjecture implies that the sum on the
left hand side of~\eqref{main-examples} is finite in this case. More
precisely, we have the following (for convenience we reverted to the
combinatorial variables $q,t$).
\begin{conjecture}
\begin{equation}
\label{conj-example-1}
(q-1)(1-t)\,\Log\left( 
  \sum_{{\lambda}\in {\cal P}}
  \frac{\prod_{i=1}^3\tilde{H}_\lambda(1,u_i,0,\ldots;q,t)} {\prod
    (q^{a+1}-t^l)(q^a-t^{l+1})}\,u_0^{|\lambda|}\right)= 
(1+u_1)(1+u_2)(1+u_3) u_0+ u_1u_2u_3 u_0^2.
\end{equation}
\end{conjecture}
Indeed, for $k=3$ our system is $D_4$ and hence all roots are real and
given by those vectors $\v$ satisfying $\v^tC\v=2$, where $C$ is the
Cartan matrix
$$
C:=\left(
\begin{array}{rrrr}
2&-1&-1&-1\\
-1&2&0&0\\
-1&0&2&0\\
-1&0&0&2\\
\end{array}
\right).
$$
The roots with positive first coordinate are precisely:
$(1,n_1,n_2,n_3)$ with $n_i=0,1$ and $(2,1,1,1)$. For $\muhat$
corresponding to a real root with positive $n$ we actually know that
$\M_\muhat$ is a point~\cite{crawley-par} and hence its mixed Hodge
polynomial is just~$1$.  The cases $k=1,2$ can be obtained
from~\eqref{conj-example-1} by setting $u_3=u_2=0$ and $u_3=0$
respectively; a proof for these cases follows by specializing the
Cauchy formula~\eqref{Cauchy}.

For $k=4$ the quiver is tame; it corresponds to the affine system
$\tilde D_4$. 
Its Cartan matrix  is
$$
C:=\left(
\begin{array}{rrrrr}
2&-1&-1&-1&-1\\
-1&2&0&0&0\\
-1&0&2&0&0\\
-1&0&0&2&0\\
-1&0&0&0&2
\end{array}
\right),
$$
where the first coordinate corresponds to the central vertex.  The
positive real roots are the vectors $\v=(v_0,v_1,\ldots,v_4)$ with
$v_i\geq 0$ for $i=0,1,\ldots 4$ such that $\v^tC\v=2$. The positive
imaginary roots are the vectors $r\v^*$, where $\v^*:=(2,1,1,1,1)$,
with $r$ a positive integer.

The main conjecture now specializes to the following (again expressed
in the combinatorial variables~$q,t$).
\begin{conjecture}
For $u_0,\ldots, u_4$ independent variables we have
$$
(q-1)(1-t)\,\Log\left( \sum_{{\lambda}\in {\cal P}}
\frac{\prod_{i=1}^4\tilde{H}_\lambda(1,u_i,0,\ldots;q,t)}
{\prod (q^{a+1}-t^l)(q^a-t^{l+1})}\,u_0^{|\lambda|}\right)= \sum_{\v} u^\v +
(q+4+t)\sum_{r\geq 1}u^{r\v^*},
$$
where the first sum is over all positive real roots
$\v=(v_0,\ldots,v_4)$ with $v_0>0$ and  $u^\v:=\prod_{j=0}^4u_j^{v_j}$.
\end{conjecture}
To see this note that for $\muhat=((r,r),(r,r),(r,r),(r,r))$,
corresponding to the imaginary root $rv^*$, the variety $\M_\muhat$ is
a smooth affine surface (by~\eqref{dimension} the dimension $d_\muhat$
equals $2$ for all $r$). By a result
of~\cite[Theorem 6.14]{etingof-et-al} $\M_\muhat$ is isomorphic to
$S^0=S\setminus \Delta\subseteq\proj^3$, where $S$ is a smooth cubic
surface and $\Delta$ is the union of the coordinate axes. A
calculation shows that the mixed Hodge polynomials of $S^0$ are
$H(S^0;q,t)=(qt)^2+4qt^2+1$ and $H_c(S^0;q,t)=t^2+4t^2q+t^4q^2$. We
should then have $\H_\muhat(z,w)=z^2+4+w^2$.

In fact, for $n=2$ the connection to cubic surfaces goes back to
Fricke and Klein~\cite[\S II.2, p. 285]{fricke-klein}. It boils down
to the following {\it Fricke relation}~\cite[p. 93]{magnus}. Given
three matrices $A_i\in\SL_2(\C)$ for $i=1,2,3$ let $a_i:=\Tr(A_i),
x_i:=\Tr(A_jA_k)$, where $\Tr$ denotes the trace and
$\{i,j,k\}=\{1,2,3\}$.  Then
\begin{equation}
\label{fricke}
0=x_1x_2x_3+\sum_{i=1}^3 \left(x_i^2-\theta_ix_i\right) +\theta_4,
\end{equation}
where, with the same convention on indices,
$$
a_4:=\Tr(A_1A_2A_3),\qquad 
\theta_i:=a_ia_4+a_ja_k,
\qquad
\theta_4:=a_1\cdots a_4 +a_1^2+\cdots+a_4^2-4.
$$
(Viewed as a quadratic equation in $a_4$ the other solution
to~\eqref{fricke} is $\Tr(A_1A_3A_2)$.)

\begin{remark}
\label{affine-quivers}
A similar form of the conjecture occurs when the quiver is the Dynkin
diagram of the affine systems $\tilde E_6, \tilde E_7$ and $ \tilde
E_8$.  The corresponding surfaces are now smooth del Pezzo surfaces of
degree $9-s$, where $s=6,7$ and $8$ respectively, with a nodal
$\proj^1$ removed.  The polynomial corresponding to any positive
imaginary root should then be $\H_\muhat(z,w)= z^2+s +w^2$.
\end{remark}

For $k\geq 5$ the quiver is wild and the main conjecture does not take
a particularly simple form. For future reference we record here the
first few values of $\H_\muhat$ for $\muhat=((n-1,1),\ldots,(n-1,1))$
or, equivalently $\v=(n,1,\ldots,1)$, calculated on a computer. For
completeness we include also the case $n=1$, where
$\v=(1,1,\ldots,1)$; the relevant range so that $\v$ is a root is then
$1\leq n\leq k-1$. We should stress the fact that the computed
$\H_\muhat$ turned out to be polynomials with non-negative
coefficients, as predicted, something that is not clear a
priori.

To simplify, below we write simply $\H_{n,k}$ for $\H_\muhat$ and
display its coefficients as an array. As mentioned above
(Remark~\ref{g=0-remark-1}) $\H_\muhat$ is a function of $z^2,w^2$ so
we record only the even powers. To be sure, for example,
$$
\H_{2,4}=
\begin{array}{cc}
1 &\\
4&  1
\end{array}
$$
corresponds to the polynomial $\H_\muhat(z,w)=z^2+4+w^2$ discussed
above for $k=4$. 
$$
\H_{1,5}=\H_{4,5}=1 \qquad \H_{2,5}=\H_{3,5}=
\begin{array}{ccc}
1 & &\\
5&  1&\\
11&5&1 
\end{array}
$$

$$
\H_{1,6}=\H_{5,6}=1\qquad
\H_{2,6}=\H_{4,6}=
\begin{array}{cccc}
1 & &&\\
6& 1& &\\
16&6&1&\\
26&16&6&1 
\end{array}
\qquad
\H_{3,6}=
\begin{array}{ccccc}
1 & &&&\\
6& 1& &&\\
22&7&1&&\\
51&27&7&1&\\
66&51&22&6&1
\end{array}
$$

$$
\H_{1,7}=\H_{6,7}=1\qquad 
\H_{2,7}=\H_{5,7}=
\begin{array}{ccccc}
1 & &&&\\
7& 1&& &\\
22&7&1&&\\
42&22&7&1&\\
57&42&22&7&1
\end{array}
$$
$$
\H_{3,7}=\H_{4,7}=
\begin{array}{ccccccc}
1 & &&&&&\\
7& 1&&& &&\\
29&8&1&&&&\\
85&36&8&1&&&\\
190&113&37&8&1&&\\
308&246&113&36&8&1\\
302&308&190&85&29&7&1
\end{array}
$$
\begin{remark}
  The observed symmetry $\H_{n,k}=\H_{k-n,n}$ should be a consequence
  of the action of the reflection associated to the central vertex
  (by~\cite{crawley-par} dimension vectors in the same orbit of the
  Weyl group of the quiver yield isomorphic varities); at the level of
  the generating functions, though, this symmetry is far from evident.
\end{remark}
\begin{remark}
  The attentive reader may have noticed that the constant terms of the
  $\H_{n,k}$'s are the first Eulerian numbers. Concretely, the
  polynomials $A_k(t):=\sum_{n=0}^{k-1}\H_{n+1,k+1}(0,0)\,t^n$ are the
  Eulerian polynomials:
$$
A_3(t)= 1+4t+t^2, \quad
A_4(t)=1+11t+11t^2+t^3,\quad
A_5(t)=1+6t+26t^2+6t^3+t^4,\ldots
$$ 
This relation will be the subject of a future publication.
\end{remark}

\subsubsection{Tennis-racquet quiver}
For our next example consider the tennis-racquet quiver, consisting of
one vertex, one loop and one leg of length one. We specialize the
variables as in the case (i): $\x=u_0(1,u_1,0,\ldots)$.
$$
(z^2-1)(1-w^2)\,\Log\left( \sum_{{\lambda}\in {\cal P}}
\calH_\lambda(z,w)
\tilde{H}_\lambda(1,u,0,\ldots;z^2,w^2)
\,u_0^{|\lambda|}\right)=\sum_\v
\H_\v(z,w)\,u^\v,
$$
where we recall
$$
\calH_\lambda(z,w)=\prod \frac{(z^{2a+1}-1)(1-w^{2l+1})}
{(z^{2a+2}-w^{2l})(z^{2a}-w^{2l+2})}.
$$
In the sum $\v$ runs over all non-zero vectors $(n,n_1)$ with $0\leq
n_1\leq n$ and $u^v:=u_0^nu_1^{n_1}$. (We extended the definition of
$\H$ to all such $\v$ as in case (i).) 

With a computer we calculated the first few terms of the right hand
side and obtained the following. We list the coefficients of $u_0^n$
for $n=1,2,3$ divided by $(z-w)^2$ (as pointed out in
Remark~\ref{MH-tensor} below, the divisibility of $\H_\muhat(z,w)$ by
$(z-w)^{2g}$ is predicted by the geometry)
$$
\begin{array}{cl}
1 &\quad 1+u\\
2 & \quad 1+(1+z^2+w^2)u+u^2\\
3 &\quad 1+(1+z^2+w^2+z^4+w^4-2zw+z^2w^2)(1+u)+u^3
\end{array}
$$
So concretely, for example, for $g=k=1$ and $\muhat=((1^2))$ we
have $\H_\muhat(z,w)=(z-w)^2(1+z^2+w^2)$.  Again, note that the
computations confirm that $\H_\v(-z,w)$ is a polynomial with
non-negative coefficients.

\subsubsection{$\GL_3(\C)$-character varieties}

Consider the comet-shaped quiver with $k$ legs of length two and any
$g$. We show how to compute $\H_\muhat(z,w)$ by hand using the tables
of Macdonald polynomials when we take the partition $\mu^i=(1^3)$ at
each leg.  Similar calculations can be performed with other partitions of
$3$. We use freely the notation and definitions of~\S\ref{gen-symm-fctns}.

From Formula~\eqref{V_n-types-1} we find that
\begin{equation}
\H_\muhat(z,w)=\sum_\omega \H_\muhat^\omega(z,w)\label{sum-Hmu}
\end{equation}
with 
$$
\H_\muhat^\omega(z,w):=
(z^2-1)(1-w^2)C_\omega^o\calH_\omega(z,w)\prod_{i=1}^k\left\langle  
  \tilde{H}_\omega(\x_i,z^2,w^2),h_{(1^3)}(\x_i)\right\rangle,
$$
where the sum is over all types of size $3$.  Since
$\muhat=(\mu^1,\dots,\mu^k)$ with $\mu^i=(1^3)$ for all $i$ we have
$$
\H_\muhat^\omega(z,w)=
(z^2-1)(1-w^2)C_\omega^o\calH_\omega(z,w)\left\langle
  \tilde{H}_\omega(\x,z^2,w^2),h_{(1^3)}(\x)\right\rangle^k.
$$ 
There are eight types $\omega=(d_1,\omega^1)\cdots(d_r,\omega^r)$,
with $d_i\in\N$ and $\omega^i\in\calP$, of size $3$: 
$$
(1,3^1), \quad (1,1^3), \quad (1,1^12^1), \quad (3,1), \quad
(1,2^1)(1,1), \quad (1,1^2)(1,1), \quad (2,1)(1,1), \quad (1,1)^3, 
$$ 
where we wrote $(1,1)^3$ for $(1,1)(1,1)(1,1)$.
\begin{proposition}
  \label{conj-gl3}
  We have $\H_\muhat^{(3,1)}(z,w)=\H_\muhat^{(2,1)(1,1)}(z,w)=0$. The
sum (\ref{sum-Hmu}) reduces to
\begin{equation}
\label{conj-gl3-fmla}
\H_\muhat(z,w)= \sum_{i=1}^6 \frac1 {\alpha_i}\beta_i^{2g}\gamma_i^k
\end{equation}
where the $\alpha,\beta$ and $\gamma$'s are the following polynomials
in $z,w$ 
$$
\begin{array}{c|c|c}
\alpha & \beta & \gamma\\
\hline
(z^6-1)(z^4-w^2)(z^4-1)(z^2-w^2) & (z^5-w)(z^3-w)(z-w) & 1+2z^2+2z^4+z^6\\
(z^2-w^4)(w^6-1)(w^4-1)(z^2-w^2) & (z-w^5)(z-w^3)(z-w) & 1+2w^2+2w^4+w^6\\
(z^4-w^2)(z^2-w^4)(z^2-1)(1-w^2) & (z^3-w^3)(z-w)^2 & 1+2z^2+2w^2+z^2w^2\\
-(z^4-1)(z^2-w^2)(z^2-1)(1-w^2) & (z^3-w)(z-w)^2 & 3(z^2+1)\\
-(z^2-w^2)(1-w^4)(z^2-1)(1-w^2) & (z-w^3)(z-w)^2& 3(w^2+1)\\
3(z^2-1)^2(1-w^2)^2 & (z-w)^3 & 6
\end{array}
$$
\end{proposition}   
\begin{proof}We only compute $\H_\muhat^{(1,3^1)}(z,w)$ and
  $\H_\muhat^{(1,1)^3}(z,w)$ as the other cases are similar. We start
  with $$\H_\muhat^{(1,3^1)}(z,w)=(z^2-1)(1-w^2)C_{3^1}^o\calH_{3^1}(z,w)\left\langle\tilde{H}_{3^1}(\x;z^2,w^2),h_{1^3}(\x)\right\rangle^k.$$

  From Formula (\ref{ctau0}) we find that $C_{(1,3^1)}^o=1$.  From the
  Young diagram of the partition $3^1$ we find
  that 
$$
\calH_{3^1}(z,w) =\frac{\left((z^5-1)(z^3-1)(z-w)\right)^{2g}}
{(z^6-1)(z^4-w^2)(z^4-1)(z^2-w^2)(z^2-1)(1-w^2)}. 
$$It
remains to compute
$\langle\tilde{H}_{3^1}(\x;z^2,w^2),h_{1^3}(\x)\rangle$. For any
partition $\lambda$ we have
$$
\tilde{H}_\lambda(\x;q,t)
=\sum_\nu\tilde{K}_{\nu\lambda}(q,t)s_\nu(\x)
$$
where $s_\nu(\x)$ is the Schur symmetric function and where
$\tilde{K}_{\nu\lambda}(q,t):=t^{n(\lambda)}K_{\nu\lambda}(q,t^{-1})$
are the $(q,t)$-Kostka polynomials. From the tables in
\cite[p. 359]{macdonald} we find for $n=3$ the following table for
$\{\tilde{K}_{\nu\lambda}(q,t)\}_{\nu,\lambda}$
\begin{equation*}
\label{table}
\begin{array}{|c|c|c|c|}
\hline
 & 3^1& 1^12^1 & 1^3 \\
\hline
3^1 & 1 & 1& 1 \\
\hline
1^12^1 & q+q^2 & q+t &t+t^2\\
\hline
1^3&  q^3 &  qt & t^3\\
\hline
\end{array}
\end{equation*} 

Hence $$\tilde{H}_{1^3}(\x;z^2,w^2)=s_{3^1}(\x)+(z^2+z^4)s_{1^12^1}(\x)+z^6s_{1^3}(\x).$$Since
the set of monomial symmetric functions $\{m_\lambda(\x)\}_\lambda$ is
the dual basis (with respect to the Hall pairing) of the set of
complete symmetric functions $\{h_\mu(\x)\}$, we need to express the
Schur functions in terms of monomial symmetric functions. Using the
tables \cite[p. 101, p. 111]{macdonald} we find that
$s_{3^1}(\x)=m_{3^1}(\x)+m_{1^12^1}(\x)+m_{1^3}(\x)$,
$s_{1^12^1}(\x)=m_{1^12^1}(\x)+2m_{1^3}(\x)$, and
$s_{1^3}(\x)=m_{1^3}(\x)$. We thus deduce that

$$\left\langle\tilde{H}_{3^1}(\x;z^2,w^2),h_{1^3}(\x)\right\rangle=1+2z^2+2z^4+z^6.$$

Let us now compute the
term 
$$
\H_\muhat^{(1,1)^3}(z,w)
=(z^2-1)(1-w^2)C_{(1,1)^3}^o\calH_{(1,1)^3}(z,w)
\left\langle\tilde{H}_{(1,1)^3}(\x;z^2,w^2),h_{1^3}(\x)\right\rangle^k.
$$
We have $C_{(1,1)^3}^o=\frac{1}{3}$.  By definition of
$\calH_\omega(z,w)$ and $\tilde{H}_\omega(\x;q,t)$ for a type $\omega$
(see~\S \ref{symm-fctns}), we have
$\calH_{(1,1)^3}(z,w)=\calH_1(z,w)^3$ and
$\tilde{H}_{(1,1)^3}(\x;q,t)=\tilde{H}_1(\x;q,t)^3$.

Hence 
$$
\calH_{(1,1)^3}(z,w)= \frac{(z-w)^{6g}}{(z^2-1)^3(1-w^2)^3}
$$
and 
$$
\tilde{H}_{(1,1)^3}(\x;q,t) =m_{(1)}(\x)m_{(1)}(\x)m_{(1)}(\x).
$$
With $\x=\{x_1,x_2,\dots\}$, the monomial symmetric function
$m_{(1)}(\x)$ writes $x_1+x_2+x_3+\cdots$. Hence $m_{(1)}(\x)^3$
decomposes as $m_{3^1}(\x)+3m_{1^12^1}(\x)+6m_{1^3}(\x)$, and
so 
$$
\left\langle\tilde{H}_{(1,1)^3}(\x;z^2,w^2),h_{1^3}(\x)\right\rangle=6.
$$
\end{proof}
\begin{corollary}
\label{gl3-coro}For $\muhat=((1^3),\ldots,(1^3))$ and $g$ arbitrary $\H_\muhat(z,w)$
is a polynomial in $z,w$.
\end{corollary}
\begin{proof}
With the notation of Proposition~\ref{conj-gl3} consider the following
rational function of $z,w,u,v$, where $u,v$ are two new
indeterminates.
\begin{equation}
\label{R-defn}
R:=\sum_{i=1}^6 \frac{\gamma_i}{\alpha_i(1-u\beta_i^2)(1-v\gamma_i)}.
\end{equation}
If we expand $R$ as a power series in $u,v$ then
by~\eqref{conj-gl3-fmla} the coefficient of $u^gv^k$ equals
$\H_\muhat(z,w)$. A calculation using Maple shows that $R=A/B$ with
$A,B\in \Z[z,w,u,v]$ and $B$ a product of polynomials in
$1+u\Z[z,w,u,v]$ or $1+v\Z[z,w,u,v]$. The claim follows.
\end{proof}

Let us write $H(\M_\muhat;x,y,t)=\sum h^{i,j;k}(\M_\muhat)x^iy^jt^k$
for the mixed Hodge polynomial for  ordinary cohomology. Since
$\M_\muhat$ is nonsingular, Poincar\'e duality
gives
$$
H(\M_\muhat;x,y,t)=(xyt^2)^{d_\muhat}H_c(\M_\muhat;x^{-1},y^{-1},t^{-1}).
$$
Hence Conjecture \ref{main} and Proposition \ref{conj-gl3} imply the
following.
\begin{conjecture} 
The polynomial $H(\M_\muhat;x,y,t)$ depends only on
  $xy$ and $t$. Moreover,
$$
H(\M_\muhat;q,t)=(qt^2)^{3k+9g-8}
\H_\muhat\left(-\sqrt{q},\frac{1}{\sqrt{q}t}\right).
$$ 
That  is
\begin{align*}
H(\M_\muhat;q,t)=&\frac{q^{6g-6}t^{12g-12}
\left((q^3t^6)(1+2q+2q^2+q^3)\right)^k
\left((q^3t+1)(q^2t+1)(qt+1)\right)^{2g}}
{(q^3t^2-1)(q^3-1)(q^2t^2-1)(q^2-1)}\\
&+\frac{\left((q^3t^5+1)(q^2t^3+1)(qt+1)\right)^{2g}
\left((qt^2+1)(q^2t^4+qt^2+1)\right)^k}
{(q^3t^6-1)(q^3t^4-1)(q^2t^4-1)(q^2t^2-1)}\\
&+\frac{(qt^2)^{4g-4}\left((q^3t^3+1)(qt+1)^2\right)^{2g}
\left(q^2t^4(2+q+qt^2+2q^2t^2)\right)^k}
{(q^3t^4-1)(q^3t^2-1)(qt^2-1)(q-1)}\\
&-\frac{(qt^2)^{6g-6}\left((q^2t+1)(qt+1)^2\right)^{2g}
\left(3q^3t^6(q+1)\right)^k}{( q^2t^2-1)(q^2-1)(qt^2-1)(q-1)}\\
&-\frac{(qt^2)^{4g-4}\left((q^2t^3+1)(qt+1)^2\right)^{2g}
\left(3q^2t^4(qt^2+1)\right)^k}{(q^2t^4-1)(q^2t^2-1)(qt^2-1)(q-1)}\\
&+\frac{(qt^2)^{6g-6}(qt+1)^{6g}6^k(qt^2)^{3k}}{3(qt^2-1)^2(q-1)^2}.
\end{align*}
\end{conjecture}
Note that by Corollary~\ref{gl3-coro} the predicted $H(\M_\muhat;q,t)$
is indeed a polynomial in $q,t$. Specializing it to $(q,t)\mapsto
(1,t)$ gives a conjectural formula for the Poincar\'e polynomial
$P(\M_\muhat;t)$ of $\M_\muhat$. We have verified that our formula for
$P(\M_\muhat;t)$ agrees with those of~\cite{garciaetal}
(cf. \cite[Remark 11.3]{garciaetal}) for small values of $g$ and $k$
giving support to our main conjecture.

For example, for $g=0$ and $k < 3$ we have $\H_\muhat(z,w)=0$. For
$k=3$ we have
$$
z^2 + (w^2 + 6)
$$ 
and for $k=4$ 
\begin{align*}
\H_\muhat(z,w)&=z^8 + (w^2 + 8)z^6 + (w^4 + 9w^2 + 33)z^4 +\\
& (w^6 +9w^4 + 41w^2 + 93)z^2 + (w^8 + 8w^6 + 33w^4 + 93w^2 + 136).
\end{align*}
Hence
$$
t^2\H_\muhat(-1,1/t)=7t^2+1
$$
and
$$
t^8\H_\muhat(-1,1/t)=271t^8 + 144t^6 + 43t^4 + 9t^2 + 1
$$
respectively, matching the values of $P(\M_\muhat,t)$ calculated
in~\cite[pp. 62--63]{garciaetal}. Note that the case $k=3$ corresponds
to the basic imaginary root of the affine $\tilde E_6$ quiver that we
already encountered (see Remark~\ref{affine-quivers}).

Incidentally, specializing $R$ in~\eqref{R-defn} to $z=w=1$ gives the
rational function
$$
\frac{4v^3(72v^2+57v+2)}{(1-6v)^5}= 8v^3+ 468v^4+ 11448v^5+
192240v^6+\cdots. 
$$
Hence Theorem~\ref{character} implies that the Euler characteristic
$E(\M_\muhat)=E(\M_\muhat,1)$ of $\M_\muhat$ for $g=0$ and
$\muhat=((1^3),\ldots,(1^3))$ is
\begin{equation}
\label{euler-char-example}
E(\M_\muhat)=2^{-5}\cdot3^{-3}\cdot(k-1)(k-2)(9k^2-27k+16)\cdot 6^k,
\end{equation}
see remark~\ref{euler-char-g=0}. (For $g=1$ and
$\muhat=((1^3),\ldots,(1^3))$ a similar calculation yields $E(\tilde
\M_\muhat)= 3^{-1}\cdot 4\cdot6^k$ agreeing with
Theorem~\ref{euler-char-g=1}.)

We have also checked that the result of similar calculations for
$\GL_2$-character varieties matches those of~\cite{boden}.
\subsection{Related work}

The present paper has spawned some recent work on the $A$-polynomial:
\cite{helleloid-villegas} studies $A$-polynomials of general quivers
from a viewpoint motivated by \cite{hausel-villegas} and this paper;
and \cite{hausel-kac} proves a further conjecture of
Kac~\cite[Conjecture 1]{kacconj}, claiming that the constant term of
the $A$-polynomial of a quiver is a certain multiplicity in the
corresponding Kac-Moody algebra, for any loop-free quiver using
Nakajima quiver varieties and techniques closely related to the ones
in this paper.

In \cite{letellier2}, the second author obtained a generalization of the
results of~\S\ref{intromulti} to arbitrary irreducible characters of
$\GL_n(\F_q)$ by computing the Poincar\'e polynomial (for the
intersection cohomology) of quiver varieties associated with the
Zariski closure of $k$ arbitrary adjoint orbits of $\gl_n(\C)$. As in
the semisimple case, it is expected that this Poincar\'e polynomial
coincides with the pure part of the mixed Hodge polynomial (again, for
the intersection cohomology) of character varieties with the Zariski
closure of conjugacy classes at the punctures. This will be
the subject of a future publication

\paragraph{Acknowledgements.} We organized a workshop bearing the
title of this paper, at the American Institute of Mathematics in Palo
Alto in June 2007. We would like to thank the Institute's staff for
their help with the organization and the participants of the
conference for the many talks and discussions, from which we learnt a
lot.  We would like to thank Mathematisches Forschungsinstitut
Oberwolfach for a research in pairs stay where some of the work was
done. TH was supported by NSF grants DMS-0305505,
DMS-0604775 an Alfred Sloan Fellowship and a Royal Society University
Research Fellowship; EL was supported by ANR-09-JCJC-0102-01; FRV was
supported by NSF grant DMS-0200605, an FRA from the University of
Texas at Austin, EPSRC grant EP/G027110/1, Visiting Fellowships at All
Souls and Wadham Colleges in Oxford and a Research Scholarship from
the Clay Mathematical Institute.

\section{Generalities}
\subsection{Character varieties}
\label{subchar}

Fix integers $g\geq 0$, $k,n>0$. We also fix a $k$-tuple of partitions
of $n$ which we denote by $\muhat=(\mu^1,\dots,\mu^k)\in \left(\calP_
  n\right)^k$, i.e. $\mu^i=(\mu^i_1,\mu^i_2,\dots,\mu^i_{r_i})$ such that
$\mu^i_1\geq\mu^i_2\geq \dots $ are non-negative integers and $\sum_j
\mu^i_j =n$.  Let $d$ be the gcd of $\{\mu^i_j\}_{i,j}$ and let $\K$ be an algebraically closed field such
that \beq \label{doesnot} \charx(\K) \nmid d.\eeq

We now construct a variety whose points parametrize representations
of the fundamental group of a $k$-punctured Riemann surface of genus
$g$ into ${\rm GL}_n(\K)$ with prescribed images in semisimple
conjugacy classes $\calC_1,\ldots,\calC_k$ at the punctures.  Assume
that
\beq
\label{prod-det}
\prod_{i=1}^k\det C_i = 1
\eeq
and that $(C_1,C_2,\ldots,C_k)$ has type
$\muhat=(\mu^1,\mu^2,\ldots,\mu^k)$; i.e., $C_i$ has type $\mu^i$ for
each $i=1,2,\ldots,k$, where the {\em type} of a
semisimple conjugacy class $C\subset \GL_n(\K)$ is defined as the
partition $\mu=(\mu_1,\mu_2,\ldots) \in \P_n$ describing the
multiplicities of the eigenvalues of (any matrix in) $C$.

\begin{definition}
\label{genericconjugacy}
The $k$-tuple $(\calC_1,\dots,\calC_k)$ is {\em generic} if the
following holds. If
$V\subseteq \K^n$ is a subspace stable by some $X_i\in C_i$ for each $i$
such that
\beq
\label{prod-condition}
\prod_{i=1}^k\det \left(X_i|_V\right)= 1
\eeq
then either $V=0$ or $V=\K^n$.
\end{definition}

For example, if $k=1$ and $C_1$ is of type $(n)$ i.e.
consists of the diagonal matrix of
eigenvalue $\zeta$ (with $\zeta^n=1$ so that \eqref{prod-det} is
satisfied) then $C$ is generic if and only if $\zeta$ is a {\em
primitive} $n$-th root of $1$.

\begin{lemma}
\label{exists}
There exists a generic
$k$-tuple of semisimple conjugacy classes $(\calC_1,\dots,\calC_k)$ of type
$\muhat$ over $\K$.
\end{lemma}
\begin{proof}Let $r_i$ be the length of the $i$-th coordinate $\mu^i$ of $\muhat$.
Let $A := \Gm^{r_1} \times \cdots \times \Gm^{r_k}$ over $\K$. For any
$\nuhat=(\nu^1,\ldots,\nu^k)=(\nu_j^i) \in \Z^{r_1}\times \cdots
\times \Z^{r_k}$ define the homomorphism
\bes
\phi_\nuhat: A & \longrightarrow  & \Gm \\
(a_j^i) & \mapsto  & \prod_{i,j}(a_j^i)^{\nu_j^i}
\ees
and set $A_\nuhat:=\ker \phi_\nuhat$. By hypothesis $\charx(\K)\nmid d$ and hence $\K$ contains a primitive
$d$-th root of unity $\zeta_d$. Let $A'$ be  defined by
$$
A': \qquad \prod_{i,j}(a_j^i)^{\mu_j^i/d}=\zeta_d.
$$
Observe that $u:=(\mu^i_j/d)_{i,j}$ is a primitive vector in
$\Z^{r_1}\times\dots\times\Z^{r_k}$. Hence we can change coordinates
in this lattice so that $u$ is part of a basis.  In the corresponding
new variables of $A$ the equation defining $A'$ is simply
$a_1=\zeta_d$ and therefore $A^\prime \cong \Gm^{\sum r_i
 -1}$, showing it is irreducible. Thus $A^\prime$ is a connected
component of $A_\muhat$.

Now if $A^\prime\subseteq A_\nuhat$ then $A_\muhat \subseteq A_\nuhat$
as $A^\prime$ generates $A_\muhat$. But $A_\muhat \subseteq A_\nuhat$
implies $l\muhat=\nuhat$ for some $l\in \Z_{\geq 0}$, since
$\charx(\K)$ does not divide $d$. So $A'_\nuhat:=A'\cap A_\nuhat
\subseteq A'$ is a proper Zariski closed subset of the irreducible
space $A'$ for every $\nuhat=(\nu_j^i)$ with $ 0\leq \nu_j^i \leq
\mu_j^i$ different from $\muhat$ and $\bf 0$.  The same is true for
all the subgroups $B$ determined by the equalities
$a_{j_1}^i=a_{j_2}^i$ for $j_1\neq j_2$. Hence the union of all
$A'_\nuhat$ and all $B$'s is not equal to the irreducible $A'$ and the
complement contains a $\K$-point. Given such a $\K$-point $(a^i_j)$
define $C_i$ to be the semisimple conjugacy class with eigenvalues
$a_j^i$ with multiplicities $\mu_j^i$. Then $(C_1,\dots,C_k)$ is
generic of type $\muhat$.
\end{proof}

For a $k$-tuple of conjugacy classes $(\calC_1,\ldots,\calC_k)$ of
type $\muhat$ define $\U_\muhat$ as the subvariety of
$\GL_n(\K)^{2g+k}$ of elements
$(A_1,\ldots,A_g,B_1,\ldots,B_g,X_1,\ldots,X_k)$ which satisfy
\beq
 \label{definingu} 
(A_1,B_1) \cdots (A_g,B_g) X_1\cdots X_k=I_n,
\qquad X_i\in \calC_i.
\eeq

\begin{remark} If $\Sigma_g$ is a compact Riemann surface of genus $g$
 with punctures $S=\{s_1,\ldots,s_k\} \subseteq \Sigma_g$ then
 $\U_\muhat$ can be identified with the set
$$
{\{\rho\in\Hom\left(\pi_1(\Sigma_g\setminus S),\GL_n(\K)\right) \;|
\; \rho(\gamma_i)\in \calC_i}\},
$$
(for some choice of base point, which we omit from the notation). Here
we use the standard presentation
$$
\pi_1(\Sigma_g\setminus S)=\langle
\alpha_1\ldots,\alpha_g;\beta_1\ldots,\beta_g;\gamma_1\ldots,\gamma_k
\;|\;
(\alpha_1,\beta_1)\cdots(\alpha_g,\beta_g)\gamma_1\cdots\gamma_k=1\rangle
$$
($\gamma_i$ is the class  of a simple loop around  $s_i$ with
orientation compatible with that of $\Sigma_g$).
\end{remark}

We have $\GL_n$ acting on $\GL_n^{2g+k}$ by conjugation. As the center
acts trivially this induces an action of $\PGL_n$. The action also
leaves \eqref{definingu}, the defining equations of $\U_\muhat$
invariant, thus induces an action of $\PGL_n$ on $\U_\muhat$.  We call
the affine GIT quotient
$$
\M_\muhat:=\U_\muhat/\!/\PGL_n=\rm{Spec}(\K[\U_\muhat]^{\PGL_n})
$$
{\em a generic character variety} of type $\muhat$. We denote by
$\pi_\muhat$ the quotient morphism
$$
\pi_\muhat:\M_\muhat\to \U_\muhat.
$$

\begin{proposition}
\label{propfree}
If $(C_1,\dots, C_k)$ is  generic of type $\muhat$ then the group
$\PGL_n(\K)$ acts set-theoretically freely on $\U_\muhat$ and
every point of $\U_\muhat$ corresponds to an irreducible
representation of $\pi_1(\Sigma_g\setminus S)$.
\end{proposition}
\begin{proof}
Let $A_1,B_1,\dots,A_g,B_g\in \GL_n(\K)$ and $X_i\in C_i$
satisfy
\beq
\label{aeqn}
(A_1,B_1)  \cdots
(A_g,B_g) X_1\cdots X_k=I_n.
\eeq
Assume that all the matrices $A_i,B_i$
and $X_j$ preserve a subspace $V\subseteq \K^n$.  Let
$A^\prime_i=A_i|_V$, $B^\prime_i=B_i|_V$ and $X^\prime_i=X_i|_V$.
Then
\beq
\label{eqn}
(A^\prime_1,B^\prime_1) \cdots
(A^\prime_g,B^\prime_g) X^\prime_1\cdots X^\prime_k=I_V.
\eeq
Taking determinants of both sides we see that the product of the
eigenvalues of the matrices $X^\prime_i$ equals $1$. Hence, by
the genericity assumption, either $V=0$ or $V=\K^n$ and the corresponding
representation of  $\pi_1(\Sigma_g\setminus S)$ is irreducible.

Now suppose $g\in \GL_n(\K)$ commutes with all the matrices $A_i,B_i$
and $X_j$. By the irreducibility of the action we just proved it follows from Schur's lemma that $g\in
\GL_n(\K)$ is a scalar.  Hence $\PGL_n(\K)$ acts set-theoretically freely on
$\U_\muhat(\K)$.
\end{proof}


Recall (\ref{dimension}) that
$d_\muhat=(2g+k-2)n^2-\sum_{i,j}\left(\mu^i_j\right)^2+2$.

\begin{theorem}
  \label{smoothcharacter}
  If $(C_1,\dots, C_k)$ is a generic $k$-tuple of semisimple conjugacy
  classes in $\GL_n(\K)$ of type $\muhat$ then the quotient
  $\pi_\mu:\U_\muhat\to \M_\muhat$ is a \emph{geometric quotient} and
  a principal $\PGL_n$-bundle.  Consequently, when non-empty, the
  variety $\M_\muhat$ is non-singular of pure dimension $d_{\muhat}$,
  i.e., it is the disjoint union of its irreducible components all
  non-singular of same dimension $d_\muhat$.
\end{theorem}

\begin{proof} If $k=1$ and $C_1$ is a central matrix, this is
 \cite[Theorem 2.2.5]{hausel-villegas}, if $g=0$ and $\K=\C$ then
 this is \cite[Proposition 5.2.8]{oblomkov-etal}. Our proof will
 combine the proofs of these two results.

Let
$$
\rho:\GL_n(\K)^{2g}\times \calC_1\times\cdots\times
\calC_k\rightarrow \SL_n(\K)
$$
be given by $$(A_1,B_1,A_2,B_2,\ldots,A_g,B_g,X_1,\ldots,X_k)\mapsto
(A_1,B_1)\cdots(A_k,B_k)X_1\cdots X_k.$$ We have
$\U_\muhat=\rho^{-1}(I_n)$.  Combining the calculations in
\cite[Theorem 2.2.5]{hausel-villegas} and \cite[Proposition
5.2.8]{oblomkov-etal} it is straightforward, albeit lengthy, to
calculate the differential $d_s\rho$; we leave it to the
reader. Exactly as in [loc. cit] we can then argue that $d_s\rho$ is
surjective for all $s\in \U_\muhat$ and so the affine variety
$\U_\muhat$ is non-singular of dimension
$$
\dim\left(\GL_n(\K)^{2g}\times C_1\times\cdots\times
C_k\right)-\dim
\SL_n(\K)=2gn^2+kn^2-n^2+1-\sum_{i,j}\left(\mu_j^i\right)^2.
$$

Exactly as in \cite[Corollaries 2.2.7, 2.2.8]{hausel-villegas} we
can argue that this is a geometric quotient as well as a $\PGL_n$
principal bundle, proving that $\M_\muhat$ is non-singular of
dimension $d_\muhat$ given by \eqref{dimension}.
\end{proof}

\subsection{Quiver varieties}
\label{subquiver}

\noindent As in \S\ref{subchar}   we fix $g, k,n,$ $
\muhat$. But in this section we take an algebraically closed field
$\K$, which satisfies
\beq
\label{notdivideq}
\charx(\K)\nmid D!
\eeq
where $ D=\min_i \max_j \mu^i_j$.  For
$i=1,\dots,k$ let $\calO_i\subset \gl_n$ be a semisimple adjoint orbit
satisfying \beq \label{sum-tr} \sum_{i=1}^k\Tr \calO_i= 0.  \eeq Let
$a_1^i,\ldots,a_{r_i}^i$ be the distinct eigenvalues of $\calO_i$,
and let $\mu^i_j$ be the multiplicity of $a^i_j$. We assume that
$\mu^i_1\geq\cdots\geq\mu^i_{r_i}$. As in the previous section, we
assume that the multiplicities $\{\mu^i_j\}_j$ determine our fixed
partitions $\mu^i$ of $n$ which is called the type of $\calO_i$, and
$\muhat:=(\mu^1,\ldots,\mu^k)$ is called the type of
$(\calO_1,\ldots,\calO_k)$.  \vspace{.05cm}

\begin{definition}
\label{genericadjoint}
The $k$-tuple $(\calO_1,\dots,\calO_k)$ of semisimple adjoint orbits is
\emph{generic} if the
following holds. If
$V\subseteq \K^n$ is a subspace stable by some $X_i\in \calO_i$ for each $i$
such that
\beq
\label{sum-condition}
\sum_{i=1}^k\Tr \left(X_i|_V\right)= 0
\eeq
then either $V=0$ or $V=\K^n$.
\end{definition}

Let $d:=\gcd\{\mu^i_j\}$. We have the following

\begin{lemma}\label{adjointexists} Assume \eqref{notdivideq}. If $d>1$ generic
$k$-tuples of adjoint orbits of type $\muhat$ do not exist. If
$d=1$, in which case we say that $\muhat$ is \emph{indivisible}, they
do. \end{lemma}
\begin{proof} In terms of eigenvalues \eqref{sum-tr} is equivalent to
 $\sum_{i,j} a^i_j \mu^i_j=0$. If $d>1$ then it is easy to construct
 for a fixed basis in $\K^n$ diagonal matrices $X_i\in \calO_i$ and
 $V\subset \K^n$ of dimension $n/d$ such that
$$
\sum_{i} \Tr(X_i|_V)= \sum_{i,j} a^i_j \frac{\mu^i_j}{d}=0.
$$
This shows the first part of our Lemma.

Phrased in terms of the eigenvalues of a matrix in $\calO_i$, in the
indivisible case we are looking for a point in the complement of a
hyperplane arrangement in $\K^{\sum r_i-1}$. (The hyperplanes do not
degenerate due to the assumption \eqref{notdivideq}.)  As $\K^{\sum
 r_i-1}$ is irreducible such a point exists.  (In the present,
additive, case we do not have the crutch of a $d$-th torsion point as
we did in Lemma~\ref{exists}.)\end{proof}

For a $k$-tuple of semisimple adjoint orbits
$(\calO_1,\dots,\calO_k)$ of type $\muhat$ define $\calV_\muhat$ as
the subvariety of $\gl_n(\K)^{2g+k}$ of matrices
$(A_1,\ldots,A_g,B_1,\ldots,B_g,X_1,\ldots,X_k)$ which satisfy
\beq\label{definingv} [A_1,B_1]+ \dots + [A_g,B_g] + X_1 + \dots +
X_k=0, \qquad X_i\in \calO_i, \eeq where $[\cdot,\cdot]$ is the Lie
bracket in $\gl_n(\K)$. As explained in Remark~\ref{quiverscheme} one
can define $\calV_\muhat$ by equations showing that it is indeed an
affine variety.

\begin{proposition} If $(\calO_1,\ldots,\calO_k)$ is generic then $\PGL_n(\K)$
acts set-theoretically freely on $\calV_\muhat$ and for any element
$(A_1,B_1,\ldots,A_g,B_g,X_1,\ldots,X_k)\in \calV_\mu$ there is no
non-zero proper subspace of $\K^n$ stable by
$A_1,B_1,\ldots,A_g,B_g,X_1,\ldots,X_k$.
\end{proposition}

\begin{proof} Similar to that of Proposition
\ref{propfree}.\end{proof} 

$\GL_n$ acts on $\calV_\muhat$ by simultaneously conjugating the
matrices in the defining equation \eqref{definingv} of
$\calV_\muhat$. We can thus construct an affine {\em quiver variety}
of type $\muhat$ as the affine GIT quotient
$$\calQ_\muhat:=\calV_\muhat/\!/\PGL_n=\rm{Spec}(\K[\calV_\muhat]^{\PGL_n}).$$
In Theorem~\ref{quiveriso} below we will prove  that
$\calQ_\muhat$ is isomorphic to a quiver variety associated to a certain
comet-shaped quiver, hence its name.

\begin{theorem} 
\label{smoothquiver}
If $(\calO_1,\dots,\calO_k)$ is generic then the variety
$\calQ_\muhat$ is non-singular of dimension $d_\muhat$. Moreover,
$\calV_\muhat/\!/\PGL_n(\K)$ is a geometric quotient and the quotient
map $\calV_\muhat\rightarrow\calQ_\muhat$ is a principal $\PGL_n$-bundle.
\end{theorem}

\begin{proof} The proof is similar to that of Theorem
\ref{smoothcharacter}.\end{proof} 

We now review the connection between $\calQ_\muhat$ and quiver
representations due to Crawley-Boevey~\cite{crawley-mat}. Let
$\mathbf{s}=(s_1,\ldots,s_k)\in\N^k$. Put
$I=\{0\}\cup\{[i,j]\}_{1\leq i\leq k,1\leq j\leq s_i}$ and let
$\Gamma$ be the quiver with $g$ loops on the central vertex
represented as below\footnote{The picture is from \cite{daisuke}.}:

\vspace{10pt}

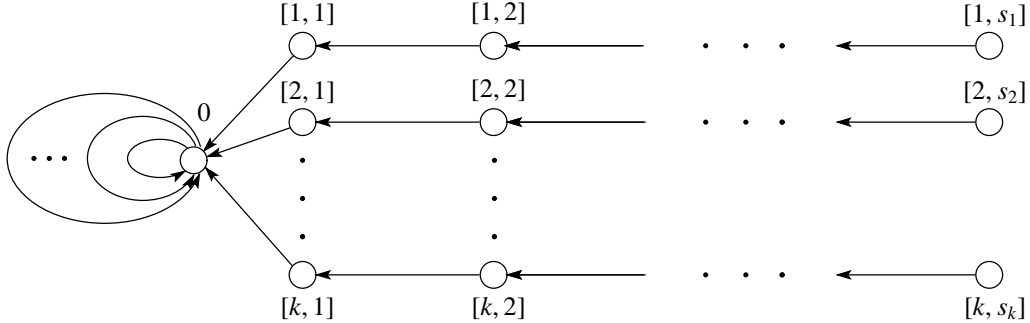
\begin{center}
\unitlength 0.1in
\begin{picture}( 52.1000, 15.4500)(  4.0000,-17.0000)
%
\special{pn 8}%
\special{ar 1376 1010 70 70  0.0000000 6.2831853}%
%
\special{pn 8}%
\special{ar 1946 410 70 70  0.0000000 6.2831853}%
%
\special{pn 8}%
\special{ar 2946 410 70 70  0.0000000 6.2831853}%
%
\special{pn 8}%
\special{ar 5540 410 70 70  0.0000000 6.2831853}%
%
\special{pn 8}%
\special{ar 1946 810 70 70  0.0000000 6.2831853}%
%
\special{pn 8}%
\special{ar 2946 810 70 70  0.0000000 6.2831853}%
%
\special{pn 8}%
\special{ar 5540 810 70 70  0.0000000 6.2831853}%
%
\special{pn 8}%
\special{ar 1946 1610 70 70  0.0000000 6.2831853}%
%
\special{pn 8}%
\special{ar 2946 1610 70 70  0.0000000 6.2831853}%
%
\special{pn 8}%
\special{ar 5540 1610 70 70  0.0000000 6.2831853}%
%
\special{pn 8}%
\special{pa 1890 1560}%
\special{pa 1440 1050}%
\special{fp}%
\special{sh 1}%
\special{pa 1440 1050}%
\special{pa 1470 1114}%
\special{pa 1476 1090}%
\special{pa 1500 1088}%
\special{pa 1440 1050}%
\special{fp}%
%
\special{pn 8}%
\special{pa 2870 410}%
\special{pa 2020 410}%
\special{fp}%
\special{sh 1}%
\special{pa 2020 410}%
\special{pa 2088 430}%
\special{pa 2074 410}%
\special{pa 2088 390}%
\special{pa 2020 410}%
\special{fp}%
%
\special{pn 8}%
\special{pa 3720 410}%
\special{pa 3010 410}%
\special{fp}%
\special{sh 1}%
\special{pa 3010 410}%
\special{pa 3078 430}%
\special{pa 3064 410}%
\special{pa 3078 390}%
\special{pa 3010 410}%
\special{fp}%
\special{pa 3730 410}%
\special{pa 3010 410}%
\special{fp}%
\special{sh 1}%
\special{pa 3010 410}%
\special{pa 3078 430}%
\special{pa 3064 410}%
\special{pa 3078 390}%
\special{pa 3010 410}%
\special{fp}%
%
\special{pn 8}%
\special{pa 2870 810}%
\special{pa 2020 810}%
\special{fp}%
\special{sh 1}%
\special{pa 2020 810}%
\special{pa 2088 830}%
\special{pa 2074 810}%
\special{pa 2088 790}%
\special{pa 2020 810}%
\special{fp}%
%
\special{pn 8}%
\special{pa 2870 1610}%
\special{pa 2020 1610}%
\special{fp}%
\special{sh 1}%
\special{pa 2020 1610}%
\special{pa 2088 1630}%
\special{pa 2074 1610}%
\special{pa 2088 1590}%
\special{pa 2020 1610}%
\special{fp}%
%
\special{pn 8}%
\special{pa 3730 810}%
\special{pa 3020 810}%
\special{fp}%
\special{sh 1}%
\special{pa 3020 810}%
\special{pa 3088 830}%
\special{pa 3074 810}%
\special{pa 3088 790}%
\special{pa 3020 810}%
\special{fp}%
\special{pa 3740 810}%
\special{pa 3020 810}%
\special{fp}%
\special{sh 1}%
\special{pa 3020 810}%
\special{pa 3088 830}%
\special{pa 3074 810}%
\special{pa 3088 790}%
\special{pa 3020 810}%
\special{fp}%
%
\special{pn 8}%
\special{pa 3730 1610}%
\special{pa 3020 1610}%
\special{fp}%
\special{sh 1}%
\special{pa 3020 1610}%
\special{pa 3088 1630}%
\special{pa 3074 1610}%
\special{pa 3088 1590}%
\special{pa 3020 1610}%
\special{fp}%
\special{pa 3740 1610}%
\special{pa 3020 1610}%
\special{fp}%
\special{sh 1}%
\special{pa 3020 1610}%
\special{pa 3088 1630}%
\special{pa 3074 1610}%
\special{pa 3088 1590}%
\special{pa 3020 1610}%
\special{fp}%
%
\special{pn 8}%
\special{pa 5466 410}%
\special{pa 4746 410}%
\special{fp}%
\special{sh 1}%
\special{pa 4746 410}%
\special{pa 4812 430}%
\special{pa 4798 410}%
\special{pa 4812 390}%
\special{pa 4746 410}%
\special{fp}%
%
\special{pn 8}%
\special{pa 5466 810}%
\special{pa 4746 810}%
\special{fp}%
\special{sh 1}%
\special{pa 4746 810}%
\special{pa 4812 830}%
\special{pa 4798 810}%
\special{pa 4812 790}%
\special{pa 4746 810}%
\special{fp}%
%
\special{pn 8}%
\special{pa 5466 1610}%
\special{pa 4746 1610}%
\special{fp}%
\special{sh 1}%
\special{pa 4746 1610}%
\special{pa 4812 1630}%
\special{pa 4798 1610}%
\special{pa 4812 1590}%
\special{pa 4746 1610}%
\special{fp}%
%
\special{pn 8}%
\special{pa 1880 840}%
\special{pa 1450 990}%
\special{fp}%
\special{sh 1}%
\special{pa 1450 990}%
\special{pa 1520 988}%
\special{pa 1500 972}%
\special{pa 1506 950}%
\special{pa 1450 990}%
\special{fp}%
%
\special{pn 8}%
\special{pa 1900 460}%
\special{pa 1430 960}%
\special{fp}%
\special{sh 1}%
\special{pa 1430 960}%
\special{pa 1490 926}%
\special{pa 1468 922}%
\special{pa 1462 898}%
\special{pa 1430 960}%
\special{fp}%
%
\special{pn 8}%
\special{sh 1}%
\special{ar 1946 1010 10 10 0  6.28318530717959E+0000}%
\special{sh 1}%
\special{ar 1946 1210 10 10 0  6.28318530717959E+0000}%
\special{sh 1}%
\special{ar 1946 1410 10 10 0  6.28318530717959E+0000}%
\special{sh 1}%
\special{ar 1946 1410 10 10 0  6.28318530717959E+0000}%
%
\special{pn 8}%
\special{sh 1}%
\special{ar 4056 410 10 10 0  6.28318530717959E+0000}%
\special{sh 1}%
\special{ar 4266 410 10 10 0  6.28318530717959E+0000}%
\special{sh 1}%
\special{ar 4456 410 10 10 0  6.28318530717959E+0000}%
\special{sh 1}%
\special{ar 4456 410 10 10 0  6.28318530717959E+0000}%
%
\special{pn 8}%
\special{sh 1}%
\special{ar 4056 810 10 10 0  6.28318530717959E+0000}%
\special{sh 1}%
\special{ar 4266 810 10 10 0  6.28318530717959E+0000}%
\special{sh 1}%
\special{ar 4456 810 10 10 0  6.28318530717959E+0000}%
\special{sh 1}%
\special{ar 4456 810 10 10 0  6.28318530717959E+0000}%
%
\special{pn 8}%
\special{sh 1}%
\special{ar 4056 1610 10 10 0  6.28318530717959E+0000}%
\special{sh 1}%
\special{ar 4266 1610 10 10 0  6.28318530717959E+0000}%
\special{sh 1}%
\special{ar 4456 1610 10 10 0  6.28318530717959E+0000}%
\special{sh 1}%
\special{ar 4456 1610 10 10 0  6.28318530717959E+0000}%
\put(19.7000,-2.4500){\makebox(0,0){$[1,1]$}}%
\put(29.7000,-2.4000){\makebox(0,0){$[1,2]$}}%
\put(55.7000,-2.5000){\makebox(0,0){$[1,s_1]$}}%
\put(19.7000,-6.5500){\makebox(0,0){$[2,1]$}}%
\put(29.7000,-6.4500){\makebox(0,0){$[2,2]$}}%
\put(55.7000,-6.5500){\makebox(0,0){$[2,s_2]$}}%
\put(19.7000,-17.8500){\makebox(0,0){$[k,1]$}}%
\put(29.7000,-17.8500){\makebox(0,0){$[k,2]$}}%
\put(55.7000,-17.8500){\makebox(0,0){$[k,s_k]$}}%
\put(14.3000,-7.6000){\makebox(0,0){$0$}}%
\special{pn 8}%
\special{sh 1}%
\special{ar 2950 1010 10 10 0  6.28318530717959E+0000}%
\special{sh 1}%
\special{ar 2950 1210 10 10 0  6.28318530717959E+0000}%
\special{sh 1}%
\special{ar 2950 1410 10 10 0  6.28318530717959E+0000}%
\special{sh 1}%
\special{ar 2950 1410 10 10 0  6.28318530717959E+0000}%
\special{pn 8}%
\special{ar 1110 1000 290 220  0.4187469 5.9693013}%
\special{pn 8}%
\special{pa 1368 1102}%
\special{pa 1376 1090}%
\special{fp}%
\special{sh 1}%
\special{pa 1376 1090}%
\special{pa 1324 1138}%
\special{pa 1348 1136}%
\special{pa 1360 1158}%
\special{pa 1376 1090}%
\special{fp}%
\special{pn 8}%
\special{ar 910 1000 510 340  0.2464396 6.0978374}%
\special{pn 8}%
\special{pa 1400 1096}%
\special{pa 1406 1084}%
\special{fp}%
\special{sh 1}%
\special{pa 1406 1084}%
\special{pa 1362 1138}%
\special{pa 1384 1132}%
\special{pa 1398 1152}%
\special{pa 1406 1084}%
\special{fp}%
\special{pn 8}%
\special{sh 1}%
\special{ar 540 1000 10 10 0  6.28318530717959E+0000}%
\special{sh 1}%
\special{ar 620 1000 10 10 0  6.28318530717959E+0000}%
\special{sh 1}%
\special{ar 700 1000 10 10 0  6.28318530717959E+0000}%
\special{pn 8}%
\special{ar 1200 1000 170 100  0.7298997 5.6860086}%
\special{pn 8}%
\special{pa 1314 1076}%
\special{pa 1328 1068}%
\special{fp}%
\special{sh 1}%
\special{pa 1328 1068}%
\special{pa 1260 1084}%
\special{pa 1282 1094}%
\special{pa 1280 1118}%
\special{pa 1328 1068}%
\special{fp}%
\end{picture}%
\end{center}

\vspace{10pt}

\noindent A \emph{dimension vector} for $\Gamma$ is a collection of
non-negative integers $\mathbf{v}=\{v_i\}_{i\in I}\in \N^I$ and a
representation of $\Gamma$ of dimension $\mathbf{v}$ over $\K$ is a
collection of $\K$-linear maps $\phi_{i,j}:\K^{v_i}\rightarrow
\K^{v_j}$ for each arrow $i\rightarrow j$ of $\Gamma$ that we
identify with matrices (using the canonical basis of $\K^r$). Let
$\Omega$ be a set indexing the edges of $\Gamma$. For
$\gamma\in\Omega$, let $h(\gamma),t(\gamma)\in I$ denote
respectively the head and the tail of $\gamma$. The algebraic group
$\prod_{i\in I}\GL_{v_i}(\K)$ acts on the space $${\rm
Rep}_{\K}(\Gamma,\mathbf{v}):=\bigoplus_{\gamma\in\Omega}\text{Mat}_{v_{h(\gamma)},v_{t(\gamma)}}(\K)$$
of representations of dimension $\mathbf{v}$ in the obvious way. As
the diagonal center $(\lambda I_{v_i})_{i\in I}\in \left(\prod_{i\in
I}\GL_{v_i}(\K)\right)$ acts trivially the action reduces to an
action of
$$\G_\mathbf{v}(\K):=\left.\left(\prod_{i\in I}\GL_{v_i}(\K)\right)\right/\K^\times.$$
Clearly two elements of ${\rm Rep}_{\K}(\Gamma,\mathbf{v})$ are
isomorphic if and only if they are $\G_\mathbf{v}(\K)$-conjugate.

Let $\overline{\Gamma}$ be the \emph{double quiver} of $\Gamma$ i.e.
$\overline{\Gamma}$ has the same vertices as $\Gamma$ but the edges
are given by
$\overline{\Omega}:=\{\gamma,\gamma^*|\hspace{.05cm}\gamma\in\Omega\}$
where $h(\gamma^*)=t(\gamma)$ and $t(\gamma^*)=h(\gamma)$. Then via
the trace pairing we may identify  ${\rm
Rep}_{\K}\left(\overline{\Gamma},\mathbf{v}\right)$ with the cotangent
bundle ${\rm T^* Rep}_{\K}(\Gamma,\mathbf{v})$. Define the
\emph{moment map} \beq \label{defmoment} \mu_{\mathbf{v}}:{\rm
Rep}_{\K}\left(\overline{\Gamma},\mathbf{v}\right)\rightarrow
M(\mathbf{v},\K)^0 \\
(x_{\gamma})_{\gamma\in\overline{\Omega}}\mapsto\sum_{\gamma\in\Omega}[x_{\gamma},x_{\gamma^*}],
\eeq where $$M(\mathbf{v},\K)^0:=\left.\left\{(f_i)_{i\in
I}\in\bigoplus_{i\in I}\gl_{v_i}(\K)\,\right|\,\sum_{i\in I}{\rm
Tr}(f_i)=0\right\}$$ is identified with the dual of the Lie algebra
of $\G_\mathbf{v}(\K)$. It is a $\G_\mathbf{v}(\K)$-equivariant map.
We define a bilinear form on $\K^I$ by
$\mathbf{a}\centerdot\mathbf{b}=\sum_ia_ib_i$. For
$\xihat=(\xi_i)_i\in \K^I$ such that $\xihat\centerdot
\mathbf{v}=0$, the element $$(\xi_i.{\rm Id})_i\in\bigoplus_i
\gl_{v_i}(\K)$$ is in fact in $M(\mathbf{v},\K)^0$. For such a
$\xihat\in\K^I$, the affine variety $\mu_{\mathbf{v}}^{-1}(\xihat)$
is endowed with a $\G_\mathbf{v}(\K)$- action. We call the affine
GIT quotient
$$
\mathfrak{M}_{\xi}(\mathbf{v})
:=\mu_{\mathbf{v}}^{-1}(\xihat)/\!/\G_\mathbf{v}(\K)
$$
the affine \emph{quiver variety}. These and related quiver varieties
were considered by many authors including Kronheimer, Lusztig,
Nakajima and Crawley-Boevey
\cite{kronheimer-nakajima,lusztigquiver,nakajima-quiver2,crawley-quiver}.

\noindent Following \cite{crawley-mat}, we now identify our
$\calQ_\muhat$, constructed from a generic $k$-tuple $(
\calO_1,\dots,\calO_k)$ of type $\muhat$, with a certain quiver
variety.  We define $\mathbf{s}$ as $s_i=l(\mu^i)-1$ where
$l(\lambda)$ denotes the length of a partition $\lambda$. Then we
define $\v\in\N^I$ as $v_0=n$ and $v_{[i,j]}=n-\sum_{r=1}^j
\mu^i_r$ for $[i,j]\in I$. Clearly $n\geq v_{[i,1]}\geq\ldots\geq
v_{[i,s_i]}$. We define $\xihat\in\K^I$ as $\xi_0=-\sum_{i=1}^ka_1^i$
and $\xi_{[i,j]}=a^i_j-a^i_{j+1}$. Observe that $\xihat\centerdot\v=0$.

For convenience, the symbol $[i,0]$, with $i\in\{1,\dots,k\}$, will also denote the vertex $0$. For a representation $\varphi\in{\rm
Rep}_{\K}\left(\overline{\Gamma},\mathbf{v}\right)$, and an arrow
$[i,j]\rightarrow [i,j-1]\in\Omega$ with $1\leq j\leq s_i$, denote
by $\varphi_{[i,j]}$ (resp. $\varphi_{[i,j]}^*$) the corresponding linear map
$\K^{v_{[i,j]}}\rightarrow\K^{v_{[i,j-1]}}$ (resp. $\K^{v_{[i,j-1]}}\rightarrow\K^{v_{[i,j]}}$), and if $\gamma_1,\dots,\gamma_g$ are the loops in $\Omega$ we denote by $\varphi_i:\K^{v_0}\rightarrow\K^{v_0}$ the linear map corresponding to $\gamma_i$ and by $\varphi_i^*$ the one corresponding to $\gamma_i^*$. Following \cite[\S3]{crawley-mat}, we construct a surjective
algebraic morphism $\omega:\mu_{\mathbf{v}}^{-1}(\xihat)\rightarrow
\calV$ which is constant on $\prod_{i\in I-\{0\}}\GL_{v_i}(\K)$
orbits. Let $\varphi\in \mu_{\mathbf{v}}^{-1}(\xihat)$. For each
$i\in\{1,\ldots,k\}$, define
$$
X_i=\varphi_{[i,1]}\varphi_{[i,1]}^*+a_1^i\text{Id}\in {\rm
 Mat}_{v_0}(\K).
$$
For $j\in\{1,\ldots,g\}$, put $A_j=\varphi_j$ and
$B_j=\varphi_j^*$. We will set \beq \label{defomega}
\omega(\phi):=(A_1,B_1,\ldots,A_g,B_g,X_1,\ldots,X_k).\eeq To show
that $\omega(\phi)\in \calV$ recall that $\mu$ at the vertex $0$ is
given by
$$
\sum_{j=1}^g[\varphi_j,\varphi_j^*]+
\sum_{i=1}^k\varphi_{[i,1]}\varphi_{[i,1]}^*=\xi_0\text{Id}
$$
which gives $$\sum_{j=1}^g[A_i,B_i]+\sum_{i=1}^kX_i=0.$$ It is
straightforward to see \cite[\S3]{crawley-mat} that we have
$X_i\in \calO_i$ for all $i\in\{1,\ldots,k\}$ from which we deduce
that indeed $$(A_1,B_1,\ldots,A_g,B_g,X_1,\ldots,X_k)\in \calV.$$

The map $\omega$ induces a bijection between isomorphic classes of
simple representations in $\mu_{\mathbf{v}}^{-1}(\xihat)$ and the
$\GL_n(\K)$-conjugacy classes of the set of tuples
$(A_1,B_1,\ldots,A_g,B_g,X_1,\ldots,X_k)\in
\calV$ thus we have~\cite{crawley-mat}:

\begin{theorem}
\label{quiveriso}If $\K=\C$, the bijective morphism $\mathfrak{M}_{\xi}(\mathbf{v})\rightarrow\calQ_\muhat$ induced by 
the map $\omega$ in \eqref{defomega} is an isomorphism.
\end{theorem}

We use this theorem in the proof of the following proposition.

\begin{proposition} 
\label{puremhs} 
Let $\K=\C$. If $(\calO_1,\dots,\calO_k)$ is generic the mixed Hodge
structure of the cohomology $H^*(\calQ_\muhat)$ of the quiver variety
$\calQ_\muhat$ is pure.
\end{proposition}
\begin{proof} We will construct a non-singular variety $\mathfrak{M}$
  with a smooth map $f:\mathfrak{M}\to\C$ such that for $0\neq
  \lambda\in \C$ the preimage
  $f^{-1}(\lambda)\simeq\mathfrak{M}_{\xi}(\mathbf{v})\simeq\calQ_\muhat$.
  Moreover we will define an action of $\C^\times$ on $\mathfrak{M}$
  covering the standard action on $\C$ such that
  $\mathfrak{M}^{\C^\times}$ is projective and the limit point
  $\lim_{\lambda\to 0} \lambda x$ exists for all $x\in
  \mathfrak{M}$. Then by Proposition~\ref{ehresmann} in Appendix B
  $H^*(\calQ_\muhat)$ has pure mixed Hodge structure.

 Similarly to \eqref{defmoment} we define \bes \mu:{\rm
   Rep}_{\K}\left(\overline{\Gamma},\mathbf{v}\right)\times \C &
 \rightarrow & M(\mathbf{v},\K)^0 \\ \left(
   (x_{\gamma})_{\gamma\in\overline{\Omega}},z\right) & \mapsto &
 \sum_{\gamma\in\Omega}[x_{\gamma},x_{\gamma}^*]-\sum_{i\in I}z\xi_i
 \rm{Id}. \ees Now for $\n=(n_i)_{i\in I}\in \Z^I$ satisfying
 $\sum_{i\in I} n_iv_i=0$ we have a character $\chi_\n$ of $\G_\v$
 given by
$$\chi_\n((g_i)_{i\in I})=\prod_{i\in I} \det(g_i)^{n_i}.$$ We call \beq\label{genericweight}\mbox{$\n=(n_i)_{i\in I}\in \Z^I$ {\em generic} if $\n\cdot\v=0$ and for $\v^\prime\in \N^I$, $0<\v^\prime< \v$ implies
 that $\n\cdot \v^\prime\neq 0$}.\eeq Because $\muhat$ is indivisible
we can take a generic $\n$. Now the character $\chi_\n$ will give a
linearization of the action of $\G_\v$ on $\mu^{-1}(0)\times \C$ and
so we can consider the GIT quotient $$\mathfrak{M}:=
\mu^{-1}(0)/\!/_{\chi_\n} \G_\v.$$ We note that $\C^\times$ acts on
$\mu^{-1}(0)$ by \beq\label{action}\lambda\left(
 (x_{\gamma})_{\gamma\in\overline{\Omega}},z\right) =\left( (\lambda
 x_{\gamma})_{\gamma\in\overline{\Omega}},\lambda^2 z\right) \eeq
commuting with the $\G_\v$ action thus descending to an action of
$\C^\times$ on $\mathfrak{M}$. Finally, we also have the map
$f:\mathfrak{M}\to \C$ given by
$f\left((x_{\gamma})_{\gamma\in\overline{\Omega}},z\right)=z$. We have
\begin{theorem}
 For a generic $\n$ the variety $\mathfrak{M}$ is non-singular, $f$
 is a smooth map (in other words a submersion),
 $\mathfrak{M}^{\C^\times}$ is complete and $\lim_{\lambda \to 0}
 \lambda x$ exists for all $x\in \mathfrak{M}$.
   \end{theorem}
   \begin{proof} $\mathfrak{M}$ is non-singular because by the
     Hilbert-Mumford criterion for (semi)-stability \cite{king},
     every semi-stable point on $\mu^{-1}(0)$ will be stable due to
     \eqref{genericweight}.

   The map $f$ is a submersion because the derivative
   $\partial_z\mu=-\sum_{i\in I}\xi_i \rm Id$ is non-zero.

   Construct the affine GIT
   quotient $$\mathfrak{M}_0:=\mu^{-1}(0)/\!/_{\chi_{\bf 0}} \G_\v$$
   using the non-generic ${\bf 0}\in \Z^I$ weight. Then the natural
   map $\mathfrak{M}\to \mathfrak{M}_0 $ is proper and the
   $\C^\times$-action \eqref{action} on $\mathfrak{M}_0$ has one
   fixed point coming from the origin in ${\rm
     Rep}_{\K}\left(\overline{\Gamma},\mathbf{v}\right)\times \C$ and
   all $\C^\times$ orbits on $\mathfrak{M}_0$ will have this origin
   in its closure. The remaining statements of the Theorem
   follow.
\end{proof}
To conclude the proof of Proposition~\ref{puremhs} it is enough to
note that by the GIT construction we have the natural map
$f^{-1}(1)\to \mathfrak{M}_\xi$, which - as a resolution of
singularities and $\mathfrak{M}_\xi$ being non-singular - is an
isomorphism. Therefore Proposition~\ref{ehresmann} implies the result.
   \end{proof}

\subsection{Symmetric functions}
\label{gen-symm-fctns}
\subsubsection{Partitions and types}
\label{partitions-types}
We denote by $\calP$ the set of all partitions including the unique
partition $0$ of $0$, by $\calP^\times$ the set of non-zero partitions and
by $\calP_n$ be the set of partitions of $n$. Partitions $\lambda$ are
denoted by $\lambda=(\lambda_1,\lambda_2,\ldots)$, where
$\lambda_1\geq \lambda_2\geq\cdots\geq 0$. We will also sometimes
write a partition as $(1^{m_1},2^{m_2},\ldots,n^{m_n})$ where $m_i$
denotes the multiplicity of $i$ in $\lambda$. The {\it size} of
$\lambda$ is $|\lambda|:=\sum_i\lambda_i$; the {\it length}
$l(\lambda)$ of $\lambda$ is the maximum $i$ with $\lambda_i>0$. 

For two partitions $\lambda$ and $\mu$, we define
$\langle\lambda,\mu\rangle$ as $\sum_i\lambda_i'\mu_i'$ where
$\lambda'$ denotes the dual partition of $\lambda$. We put
$n(\lambda):=\sum_{i>0}(i-1)\lambda_i$. Then
$\langle\lambda,\lambda\rangle=2n(\lambda)+|\lambda|$.  For two
partitions $\lambda=(1^{n_1},2^{n_2},\dots)$ and
$\mu=(1^{m_1},2^{m_2},\dots)$, we denote by $\lambda\cup\mu$ the
partition $(1^{n_1+m_1},2^{n_2+m_2},\dots)$. For a non-negative
integer $d$ and a partition $\lambda$, we denote by $d\cdot\lambda$
the partition $(d\lambda_1,d\lambda_2,\dots)$. The {\it dominance
  ordering} for partitions is defined as follows: $\mu \unlhd \lambda$
if and only if $\mu_1+\cdots+\mu_j\leq \lambda_1+\cdots+\lambda_j$ for
all $j\geq 1$.

For a partition $\lambda$, let $t_{\lambda}$ in the symmetric group of
permutations of $|\lambda|$ letters $\mathcal{S}_{|\lambda|}$, be an
element in the conjugacy class of type $\lambda$. We denote by
$z_\lambda$ the cardinality of the centralizer of $t_{\lambda}$ in
$\mathcal{S}_{|\lambda|}$. For two partitions $\lambda,\mu$ such that
$|\lambda|=|\mu|$, we denote by $\chi^{\lambda}_{\mu}$ the value at
$t_{\mu}$ of the irreducible character $\chi^{\lambda}$ of
$\mathcal{S}_{|\lambda|}$.

We choose once for all a total order $\geq$ on the set of pairs
$(d,\lambda)$ where $d\in\Z_{>0}$ and $\lambda\in\calP^\times$ such
that if $d>d'$ then $(d,\lambda)>(d',\mu)$, if $|\lambda|>|\mu|$ then
$(d,\lambda)>(d,\mu)$, and if $|\lambda|=|\mu|$, then $(d,\lambda)\geq
(d,\mu)$ if $\lambda$ is larger than $\mu$ with respect to the
lexicographic order. We denote by $\mathbf{T}$ the set of
non-increasing sequences $\omega=(d_1,\omega^1)\geq (d_2,\omega^2)\geq
\cdots \geq (d_r,\omega^r)$, which we will call a {\it type}. To
alleviate the notation we will then omitt the symbol $\geq$ and write
simply $\omega=(d_1,\omega^1)(d_2,\omega^2)\cdots (d_r,\omega^r)$. The
{\it size} of a type $\omega$ is $|\omega|:=\sum_id_i|\lambda^i|$. We
denote by $\mathbf{T}_n$ the set of types of size $n$.  For a type
$\omega=(d_1,\omega^1)(d_2,\omega^2)\cdots (d_r,\omega^r)$, we put
$n(\omega):=\sum_id_i n(\omega^i)$ and
$[\omega]:=\cup_id_i\cdot\omega^i$, a partition of size $|\omega|$

As with partitions it is sometimes convenient to consider a type in
terms of multiplicities.  Given a type $\omega$ let
$m_{d,\lambda}(\omega)$ the multiplicity of $(d,\lambda)$ in $\omega$;
i.e., how many times the pair $(d,\lambda)$ appears in $\omega$. The
integers $m_{d,\lambda}\geq 0 $ indexed by pairs $(d,\lambda)\in
\Nstar\times \; \calP^\times$ determine $\omega$ uniquely.

A partition $\lambda=(n_1,\dots,n_r)$ of $n$ can be seen as the type
$\lambda_*:=(1,1^{n_1})\cdots(1,1^{n_r})\in\bold{T}_n$ which is the
type of a semisimple conjugacy class in the sense
of~\S\ref{types}. Similarly when a multi-partition $\lambdahat$ is
considered as a multi-type it is denoted by $\lambdahat_*$.

\subsubsection{Symmetric functions}
\label{symm-fctns}
Let $\Lambda(\x_1,\ldots,\x_k):=
\Lambda(\x_1)\otimes_\Z\cdots\otimes_\Z\Lambda(\x_k)$ be the ring of
functions separately symmetric in each set $\x_1,\x_2,\ldots,\x_k$ of
infinitely many variables. We will consider
elements in $\Lambda(\x_1,\ldots,\x_k)\otimes_\Z\Q(q,t)$ where $q$ and
$t$ are two indeterminates or similarly
$\Lambda(\x_1,\ldots,\x_k)\otimes_\Z\Q(z,w)$ depending on the situation. To ease the notation we
will simply write $\Lambda$ for the various rings
$\Lambda(\x),\Lambda(\x_1,\ldots,\x_k),
\Lambda(\x_1,\ldots,\x_k)\otimes_\Z\Q(q,t),
\Lambda(\x_1,\ldots,\x_k)\otimes_\Z\Q(z,w) $, etc. as long as the
context is clear. When considering elements $a_\muhat\in \Lambda$
indexed by multi-partitions $\muhat=(\mu^1,\ldots,\mu^r)\in \calP^k$,
we will always assume that they are homogeneous of degree
$(|\mu^1|,\ldots,|\mu^k|)$. Given any family of symmetric functions
indexed by partitions $\mu\in \P$ and a multi-partition $\muhat \in
\P^k$ as above define
$$
a_\muhat:=a_{\mu^1}(\x_1)\cdots a_{\mu^k}(\x_k).
$$
We will deal with elements of the ring $\Lambda(\x) \otimes_\Z
\Q(z,w)$ and their images under two specializations: their {\it pure
  part}, $z=0,w=\sqrt q$ and their {\it Euler specialization},
$z=\sqrt q,w=1/\sqrt q$.

Let  $\langle\cdot ,\cdot \rangle$ be the Hall pairing on
$\Lambda(\x),$ extend its definition to $\Lambda(\x_1,\ldots,\x_k)$
by setting
\beq \label{extendedhall}
\langle a_1(\x_1)\cdots a_k(\x_k), b_1(\x_1)\cdots b_k(\x_k) \rangle
= \langle a_1, b_1 \rangle \cdots \langle a_k, b_k \rangle,
\eeq
for any $a_1,\ldots,a_k;b_1,\ldots,b_k\in \Lambda(\x)$ and to formal
series by linearity.

Given any family of symmetric functions
$A_\lambda(\x_1,\ldots,\x_k;q,t) \in \Lambda$ indexed by partitions
with $A_0=1$, we extend its definition to types
$\omega=(d_1,\omega^1)(d_2,\omega^2)\ldots(d_r,\omega^r)\in\mathbf{T}$
by setting
$$
A_\omega(\x_1,\ldots,\x_k;q,t):= \prod_j
A_{\omega^j}(\x_1^{d_j},\ldots,\x_k^{d_j};q^{d_j},t^{d_j}).
$$
Here $\x^d$ stands for all the variables $x_1,x_2,\ldots$ in $\x$
replaced by $x_1^d,x_2^d,\ldots$ (technically we are applying the
Adams operation $\psi_d$ to $A_{\omega^j}$ in the $\lambda$-ring
$\Lambda$).

We will need the following lemma; $p_\lambda\in \Lambda(\x)$ are the
{\it power sums} symmetric functions.
\begin{lemma}
\label{induc}
Let $\lambda\in\mathcal{P}_n$ and let $d$ be a positive
integer such that $d \mid n$.  Then
$$
\langle p_{(d^{n/d})},h_{\lambda}\rangle=
\begin{cases}
\frac{(n/d)!}{\prod_i\rho_i!} & \text{ if } \lambda=d\cdot\rho
\;\text{ for some }\rho=(\rho_1,\rho_2,\ldots)\in\mathcal{P}_{n/d}\\
0 & \text{otherwise.}
\end{cases}
$$
\end{lemma}
\begin{proof}
For a finite group $G$ let $\langle\cdot,\cdot\rangle_G$ denote  the
standard inner product on class   functions of $G$. Using the
Frobenius characteristic  map \cite[I,7]{macdonald} we have, for any
two partitions $\lambda=(\lambda_1,\lambda_2,\ldots,\lambda_r)$  and
$\mu=(\mu_1,\mu_2,\ldots,\mu_s)$ of size $n$,
$$
\langle p_\mu,h_\lambda  \rangle = z_\mu \left\langle
\delta_\mu,\Ind_{\mathcal{S}_\lambda}^{\mathcal{S}_n}(1)\right\rangle_{\mathcal{S}_n},
$$
 $\delta_\mu(\sigma)=1$ if $\sigma\in \mathcal{S}_n$ has
cycle type $\mu$ and $\delta_\mu(\sigma)=0$ otherwise and 
$\mathcal{S}_\lambda:=\mathcal{S}_{\lambda_1}\times
\mathcal{S}_{\lambda_2}\times \cdots\times
\mathcal{S}_{\lambda_r}\subseteq \mathcal{S}_n$.

Hence, by Frobenius reciprocity
$$
\langle  p_\mu,h_\lambda  \rangle = z_\mu\left\langle
\Res_{\mathcal{S}_\lambda}^{\mathcal{S}_n}
\delta_\mu,1\right\rangle_{\mathcal{S}_\lambda}.
$$
The only non-zero terms contributing to the sum implicit in the
right hand side are those  elements of  $\mathcal{S}_\lambda$ with
cycle type $(\mu^1,\ldots,\mu^r)$ with $|\mu^i|=\lambda_i$ and
$\cup_i \mu^i=\mu$. If $\mu=(d^{n/d})$ this forces $d\mid \lambda_i$
and $\mu^i=(d^{\rho_i})$, where $\rho_i:=\lambda_i/d$ and the claim
follows.
\end{proof}

\subsubsection{$\Exp$ and $\Log$}
\label{generating}
We will use the maps $\Exp$ and $\Log$ of \cite{hausel-villegas}
extended to $\Lambda$. The general context is that of $\lambda$-rings
\cite{getzler} but the following discussion will suffice for us.  For
$V\in T\Lambda[[T]]$  let
\begin{eqnarray}
\label{Exp}
\Exp: T\Lambda[[T]] & \longrightarrow & 1+T\Lambda[[T]] \\
V & \mapsto & \exp
\left(\sum_{d\geq 1}\tfrac 1
  dV(\x_1^d,\ldots,\x_k^d,q^d,t^d,T^d)\right).
\end{eqnarray}

The map $\Exp$ is related to the Cauchy kernel
\begin{equation}
\label{cauchy-defn}
C(\x):=\prod_i(1-x_i)^{-1}
\end{equation}
by
$$
\Exp(X)=C(\x), \qquad \qquad X:=x_1+x_2+\cdots=m_{(1)}(\x),
$$
($m_\lambda(\x)\in\Lambda(\x)$ is the monomial symmetric function).
It has an inverse $\Log$ defined as follows.  Given $F\in
1+T\Lambda[[T]]$ let $U_n\in \Lambda$ be the coefficients in the
expansion
$$
\log(F)=:\sum_{n\geq 1} U_n(\x_1,\ldots,\x_k;q,t)\frac{T^n}n.
$$
Define
\begin{equation}
\label{U-V-general}
V_n(\x_1,\ldots,\x_k;q,t):=\frac 1 n \sum_{d\mid n} \mu(d)\,
U_{n/d}(\x_1^d, \ldots,\x_k^d;q^d,t^d),
\end{equation}
where $\mu$ is the ordinary M\"obius function, then
$$
\Log(F):=\sum_{n\geq 1} V_n(\x_1,\ldots,\x_k;q,t)\,T^n.
$$

To simplify the discussion we now restrict to the case of $k=1$ but
everything extends easily to the general case.  Suppose
$A_\lambda(\x;q,t) \in \Lambda$ is a sequence of symmetric functions
indexed by partitions with $A_0=1$. We want an expression for $V_n\in
\Lambda$ in
$$
\sum_{n\geq 1} V_nT^n :=\Log\left(\sum_{\lambda \in \p}A_\lambda
T^{|\lambda|}\right).
$$
We first compute
$$
\sum_{n\geq 1} U_n \frac{T^n}n:=\log\left(\sum_{\lambda
\in\p}A_\lambda T^{|\lambda|}\right),
$$
where $U_n$ and $V_n$ are related by \eqref{U-V-general}. By the
multinomial theorem we have
\begin{equation}
\frac{U_n} n= \sum_{m_\lambda} (-1)^{m-1}(m-1)! \prod_\lambda
\frac{A_\lambda^{m_\lambda}}{
m_\lambda!},\label{U-equ}\end{equation} where $m:= \sum_\lambda
m_\lambda$ and the sum is over all sequences $\{m_\lambda\}_{\lambda\in\calP^\times}$ of
non-negative integers such that
$$\sum_\lambda m_\lambda|\lambda|=n.$$
We find then
$$
V_n=\sum \frac{\mu(d)} d (-1)^{m_d-1} (m_d-1)! \prod_\lambda
\frac{A_\lambda(\x_1^d,\ldots,\x_k^d;q^d,t^d)^{m_{d,\lambda}}} {
m_{d,\lambda}!},
$$
where the sum is over all sequences of non-negative integers
$m_{d,\lambda}$ indexed by pairs $(d,\lambda)\in \Nstar\times
\; \calP^\times$ satisfying
$$
\sum_\lambda  m_{d,\lambda}d|\lambda| =n, \qquad \qquad
m_d:= \sum_\lambda m_{d,\lambda}.
$$

Alternatively, we may consider  not collecting equal terms  when
expanding the logarithm to obtain
\begin{equation}
\label{log-expansion}
V_n=\sum\frac{\mu(d)}d \frac{(-1)^{r-1}}r
A_{\lambda^1}(q^d)\cdots A_{\lambda^r}(q^d),
\end{equation}
where the sum is over $\lambda^1,\lambda^2,\ldots\in\calP^\times$ and $d
\in \Nstar$ such that
$$
n=d\sum_j|\lambda^j|.
$$

Finally, we may also rewrite the expression for $V_n$ as a sum
over types $\omega$:
\begin{equation}
\label{V_n-types}
V_n= \sum_{|\omega|=n} C_\omega^0 A_\omega,
\end{equation}
so that
\begin{equation}
\label{V_n-types-1}
\Log\left(\sum_{\lambda \in \p}A_\lambda
T^{|\lambda|}\right) = \sum_{\omega} C_\omega^0 A_\omega \;T^{|\omega|},
\end{equation}
where $C_\omega^0=0$ unless $\omega$ is concentrated in some degree $d$;
i.e., $\omega=(d,\omega^1)(d,\omega^2)\cdots(d,\omega^r)$, in which case,
\beq 
\label{ctau0}
C_\omega^0 = \frac{\mu(d)} d (-1)^{r-1}
\frac{(r-1)!} {\prod_\lambda m_{d,\lambda}(\omega)!}.
\eeq

\begin{remark}
The formal power series $\sum_{n\geq 0}a_nT^n$ with $a_n\in \Lambda$
that we will consider in what follows will all have $a_n$ homogeneous
of degree $n$. Hence we will typically  scale the variables of
$\Lambda$ by $1/T$ and eliminate $T$ altogether.
\end{remark}
\begin{remark}
Note also the following useful fact. If we write
$$
\log\left(\sum_{\lambda \in \p}A_\lambda(\x)
T^{|\lambda|}\right) =\sum_\mu U_\mu(q,t)\,m_\mu(\x),
\qquad \qquad
\Log\left(\sum_{\lambda \in \p}A_\lambda(\x)
T^{|\lambda|}\right) =\sum_\mu V_\mu(q,t)\,m_\mu(\x),
$$
where $m_\mu(\x)$ are the monomial symmetric functions then it is easy
to check that
\begin{equation}
\label{V-U-partitions}
V_\mu(q,t):=\frac 1 n \sum_{d\mid \mu} \mu(d)\,
U_{\mu/d}(q^d,t^d),
\end{equation}
where $d\mid \mu$ means that $d$ divides every part $\mu_i$ of $\mu$
and $\mu/d :=(\mu_1/d,\mu_2/d,\ldots)$. In particular, if $\mu$ is
indivisible the sum on the right hand side consist of only the $d=1$
term and $U_\mu=V_\mu/n$ (this is particularly useful
for computations).
\end{remark}

\subsubsection{Macdonald and Hall-Littlewood symmetric functions.
  Green polynomials} 
\label{macdonald-hall-littlewood}
For a partition $\lambda$ let $\tilde{H}_\lambda(\x;q,t) \in
\Lambda(\x) \otimes_\Z \Q(q,t)$ be the {\it Macdonald symmetric
  function} defined in \cite[I.11]{garsia-haiman}.  We collect in this
section some basic properties of these functions that we will need.

We have the duality
\begin{equation}
\tilde{H}_\lambda(\x;q,t)=\tilde{H}_{\lambda^\prime}(\x;t,q)
\label{Hduality}\end{equation}
see \cite[Corollary 3.2]{garsia-haiman}. We define the (transformed)
{\it Hall-Littlewood symmetric function} as
\begin{equation}
\label{HL-defn}
\tilde{H}_\lambda(\x;q):=\tilde{H}_\lambda(\x;0,q).
\end{equation}
In the notation just introduced then $\tilde{H}_\lambda(\x;q)$ is the
pure part of $\tilde{H}_\lambda(\x;z^2,w^2)$.

Define the {\it $(q,t)$-Kotska polynomials}
$\tilde{K}_{\nu\lambda}(q,t)$ by
\begin{equation}
\label{K-defn}
\tilde{H}_{\lambda}(\x;q,t)=
\sum_{\nu}\tilde{K}_{\nu\lambda}(q,t)s_{\nu}(\x),
\end{equation}
where $s_\nu$ are the Schur symmetric functions.  These are $(q,t)$
generalizations of the $\tilde{K}_{\nu\lambda}(q)$ Kostka-Foulkes 
polynomial \cite[III, (7.11)]{macdonald}, which are obtained as
$q^{n(\lambda)}K_{\nu\lambda}(q^{-1})=
\tilde{K}_{\nu\lambda}(q)=\tilde{K}_{\nu\lambda}(0,q)$, i.e., by 
taking their pure part. In particular,
\begin{equation}
\label{Hall-Littlewood}
\tilde{H}_{\lambda}(\x;q)=\sum_{\nu}\tilde{K}_{\nu\lambda}(q)s_{\nu}(\x).
\end{equation}
For partitions $\lambda,\tau$ we define the {\it Green polynomial}
\begin{equation}
  Q_{\lambda}^{\tau}(q)=\sum_{\nu}\chi_{\lambda}^{\nu}\tilde{K}_{\nu\tau}(q),
\label{greenKostka} 
\end{equation}
where $\tilde{K}_{\nu\tau}(q)$ is the Kostka-Foulkes
polynomial~\eqref{K-defn}.

For two partitions $\nu,\lambda\in\mathcal{P}_n$, we have \cite[Page
363]{macdonald} the Euler specialization
\begin{equation}
\label{K-Euler}
\tilde{K}_{\nu\lambda}(q,q^{-1})=q^{-n(\lambda)}K_{\nu\lambda}(q,q)=q^{-n(\lambda)}
H_{\lambda}(q)\sum_{\rho}\frac{\chi^{\nu}_{\rho}\chi^{\lambda}_{\rho}}
{z_{\rho}\prod_i(1-q^{\rho_i})}
\end{equation}where $H_{\lambda}(q):= \prod_{s\in \lambda}(1-q^{h(s)})$
is the \emph{hook polynomial} \cite[I, 3, example 2]{macdonald}.

If $\y=\{y_1,y_2,\ldots\},\x=\{x_1,x_2,\ldots\}$ are two sets of
infinitely many variables, we denote by $\x\y$ the set of variables
${\{x_iy_j}\}_{i,j}$.
\begin{lemma}
\label{H-Euler}
Under the Euler specialization 
$$
\tilde{H}_{\lambda}(\x;q,q^{-1})=
q^{-n(\lambda)}H_{\lambda}(q)s_{\lambda}(\x\y),
$$
where $y_i=q^{i-1}$.
\end{lemma}
\begin{proof}
  With the specialization $y_i=q^{i-1}$ we get
  $p_{\rho}(\y)=\prod_i(1-q^{\rho_i})^{-1}$. Hence by \eqref{K-Euler}
\begin{align*}
\tilde{H}_{\lambda}(\x;q,q^{-1})
&=q^{-n(\lambda)}H_{\lambda}(q)\sum_{\rho,\nu}
\frac{\chi_{\rho}^{\nu}\chi_{\rho}^{\lambda}}{z_{\rho}}p_{\rho}(\y)s_{\nu}(\x)\\
&=q^{-n(\lambda)}H_{\lambda}(q)
\sum_{\rho}z_{\rho}^{-1}\chi_{\rho}^{\lambda}p_{\rho}(\y)\sum_{\nu}\chi_{\rho}^{\nu}s_{\nu}(\x)\\ 
&=q^{-n(\lambda)}
H_{\lambda}(q)\sum_{\rho}z_{\rho}^{-1}\chi_{\rho}^{\lambda}p_{\rho}(\y)p_{\rho}(\x)\\&=q^{-n(\lambda)}
H_{\lambda}(q)\sum_{\rho}z_{\rho}^{-1}\chi_{\rho}^{\lambda}p_{\rho}(\x\y)\\
&=q^{-n(\lambda)}H_{\lambda}(q)s_{\lambda}(\x\y).
\end{align*}
\end{proof}

For two types $\omega=(d_1,\omega^1)\cdots(d_r,\omega^r)$ and $\tau=(\delta_1,\tau^1)\cdots(\delta_s,\tau^s)$, write
$\omega\sim\tau$ if $r=s$ and for each $i=1,2,\dots,r$, $d_i=\delta_i$ and $|\omega^i|=|\tau^i|$.

\noindent For two types $\omega$ and $\tau$, put

$\chi_{\tau}^{\omega}:=\prod_i\chi_{\tau^i}^{\omega^i}$ if
$\omega\sim\tau$, and $\chi_{\tau}^{\omega}=0$ otherwise,

$Q_{\tau}^{\omega}(q):=\prod_iQ_{\tau^i}^{\omega^i}(q^{d_i})$ if
$\omega\sim\tau$, and $Q_{\tau}^{\omega}(q)=0$ otherwise,

$\tilde{K}_{\tau\omega}(q):=\prod_i\tilde{K}_{\tau^i\omega^i}(q^{d_i})$
if $\omega\sim\tau$, and $\tilde{K}_{\tau\omega}(q)=0$ otherwise.
Note that formulas~\eqref{greenKostka} and~\eqref{Hall-Littlewood}
extend to types, namely
$Q_{\tau}^{\omega}(q)=\sum_{\nu}\chi_{\tau}^{\nu}\tilde{K}_{\nu\omega}(q)$
and 
$\tilde{H}_{\omega}(\x;q)=\sum_{\tau}\tilde{K}_{\tau\omega}(q)s_{\tau}(\x)$,
where $\tau,\omega,\nu\in\bold{T}$.

\begin{lemma} For $\alpha,\beta\in\bold{T}$, put
$$
A(\alpha,\beta):=\sum_{\tau}
\frac{z_{[\tau]}\chi_{\tau}^{\alpha}} {z_{\tau}}\sum_{\{\nu|\, 
  [\nu]=[\tau]\}}\frac{Q_{\nu}^{\beta}(q)}{z_{\nu}},
$$
where the sums are over types.
Then 
$$
A(\alpha,\beta)=\langle
s_{\alpha}(\x),\tilde{H}_{\beta}(\x;q)\rangle,
$$
where for a partition $\lambda$, $s_{\lambda}(\x)\in \Lambda(\x)$ is
the Schur symmetric function and $\tilde{H}_{\lambda}(\x;q)$
the transformed Hall-Littlewood symmetric function~\eqref{Hall-Littlewood}.
\label{intermediate1}
\end{lemma}

\begin{proof}
For $\omega\in\bold{T}$, define
$$
a_\omega(\x):=\sum_{\tau}\chi_{\tau}^{\omega}\frac{p_{\tau}(\x)}{z_{\tau}},
\text{ and
} \,
b_\omega(\x):=\sum_{\nu}Q_{\nu}^{\omega}(q)\frac{p_{\nu}(\x)}{z_{\nu}},
$$
where ${\{p_{\lambda}(\x)}\}_{\lambda\in\mathcal{P}}$ is the family of
power symmetric functions which satisfies for two partitions
$\lambda,\tau\in\mathcal{P}$,
$$\langle
p_{\lambda}(\x),p_{\tau}(\x)\rangle=\delta_{\lambda,\tau}z_{\tau}.$$For
a type $\omega\in\bold{T}$, we have
$p_{\omega}(\x):=\prod_ip_{\omega^i}(\x^{d_i})=p_{[\omega]}(\x)$.
Therefore for $\alpha,\beta\in\bold{T}$, we have $\langle
p_{\alpha}(\x),p_{\beta}(\x)\rangle=\delta_{[\alpha],[\beta]}z_{[\alpha]}$.
Hence \begin{align*}\langle
a_{\alpha}(\x),b_{\beta}(\x)\rangle&=\sum_{\tau}\sum_{\nu}\chi_{\tau}^{\alpha}Q_{\nu}^{\beta}(q)\frac{\langle
p_{\tau}(\x),p_{\nu}(\x)\rangle}{z_{\tau}z_{\nu}}\\&=\sum_{\tau}\sum_{\nu}\chi_{\tau}^{\alpha}Q_{\nu}^{\beta}(q)\delta_{[\tau],[\nu]}\frac{z_{[\tau]}}{z_{\tau}z_{\nu}}\\
&=A(\alpha,\beta).\end{align*} Recall that for a partition
$\lambda\in\mathcal{P}$, we have
$$s_{\lambda}(\x)=\sum_{\tau}\chi^{\lambda}_{\tau}\frac{p_{\tau}(\x)}{z_{\tau}}.$$Hence for a
type $\omega\in\bold{T}$, we
have$$s_{\omega}(\x)=\sum_{\tau}\chi^{\omega}_{\tau}\frac{p_{\tau}(\x)}{z_{\tau}}=a_{\omega}(\x).
$$
Hence we may write
\begin{align*}b_{\omega}(\x)&=\sum_{\nu}\sum_{\tau}\chi_{\nu}^{\tau}\tilde{K}_{\tau\omega}(q)\frac{p_{\nu}(\x)}{z_{\nu}}\\&=\sum_{\tau
}\sum_{\nu}\chi_{\nu}^{\tau}\tilde{K}_{\tau\omega}(q)\frac{p_{\nu}(\x)}{z_{\nu}}\\
&=\sum_{\tau}\tilde{K}_{\tau\omega}(q)s_{\tau}(\x)=\tilde{H}_{\omega}(\x;q).\end{align*}
\end{proof} 

\begin{lemma} Let $\lambda\in\mathcal{P}$. With the specialization
$y_i=q^{i-1}$, we have
\begin{equation}
h_{\lambda}(\x\y)=(-1)^{|\lambda|}q^{n(\lambda_*)}
\calH_{\lambda_*}^0(0,\sqrt{q})\tilde{H}_{\lambda_*}(\x;q),
\end{equation}
where
$\calH_{\lambda}^0(z,w)$ is the genus $0$ hook function.
\label{intermediate3}\end{lemma}

\begin{proof} We need to prove that for $m\in\Nstar$,
$$
h_m(\x\y)=(-1)^mq^{n(1^m)}
\calH_{(1^m)}^0(0,\sqrt{q})\tilde{H}_{(1^m)}(\x;q).
$$
In the language of plethystic substitution (we use the notation of
\cite{garsia-haiman}), the transformed Hall-Littlewood
\eqref{HL-defn} $\tilde{H}_{\mu}(\x;q)$ equals
$$
\tilde{H}_{\mu}(\x;q)
=q^{n(\mu)}b_{\mu}(q^{-1})P_{\mu}\left[\frac{X}{1-q^{-1}};q^{-1}\right]
$$ 
where $P_{\mu}(\x;q)$ is the Hall-Littlewood symmetric
function defined in \cite{macdonald}. Since
$\calH_{\mu}^0(0,\sqrt{q})=q^{-\langle\mu,\mu\rangle}b_{\mu}(q^{-1})^{-1}$,
we have
\begin{align}(-1)^{|\mu|}q^{n(\mu)}\calH_{\mu}^0(0,\sqrt{q})\tilde{H}_{\mu}(\x;q)&=(-1)^{|\mu|}q^{-|\mu|}P_{\mu}\left[\frac{X}{1-q^{-1}};q^{-1}\right]\\
&=(-q^{-1})^{|\mu|}P_{\mu}\left[-\frac{qX}{1-q};q^{-1}\right]\label{2}\end{align}
On the other hand from \cite[VI, (4.8)]{macdonald} we have
$$(-q^{-1})^mP_{(1^m)}\left[-qX;q^{-1}\right]=(-q^{-1})^me_m\left[-qX\right]=(-q^{-1})^ms_{(1^m)}\left[-qX\right]=s_{(m^1)}(\x)=h_m(\x).$$
Since for any symmetric function $u$, we have
$u(\x\y)=u\left[\frac{X}{1-q}\right]$, we deduce that
$$h_m(\x\y)=(-q^{-1})^mP_{(1^m)}\left[-\frac{qX}{1-q};q^{-1}\right].$$The
lemma follows thus from Formula (\ref{2}).\end{proof} 

\subsubsection{Genus $g$ hook function}
\label{genus-g-hook-fctn}
Given a partition $\lambda\in\calP_n$ we
define the genus $g$ {\it hook function}
$\calH_{\lambda}(z,w)$ by
$$
\calH_{\lambda}(z,w):=
\prod_{s\in \lambda}\frac{(z^{2a(s)+1}-w^{2l(s)+1})^{2g}}
{(z^{2a(s)+2}-w^{2l(s)})(z^{2a(s)}-w^{2l(s)+2})},
$$
where the product is over all cells $s$ of $\lambda$ with $a(s)$ and
$l(s)$ its arm and leg length, respectively. For details on the hook
function we refer the reader to \cite{hausel-villegas}.

\begin{remark}
\label{g=0-remark}
  Note that $\calH_\lambda(z,w)$ is rational function of $z^2$ and
  $w^2$ when $g=0$.
\end{remark}
We have 
\begin{equation}
\calH_\lambda(z,w)=\calH_{\lambda'}(w,z)\,\,\,{\rm and
}\,\,\,\calH_\lambda(-z,-w)=\calH_\lambda(z,w).\label{hook-duality}
\end{equation}
The pure part of $\calH_\lambda$ is
\noindent
\begin{align*}
\calH_{\lambda}(0,\sqrt{q})&
=\prod_{a=0}\frac{q^{g(2l+1)}}{q^l(q^{l+1}-1)}\prod_{a\neq
0}q^{(g-1)(2l+1)}\\
&=\frac{q^{(g-1)(2n(\lambda)+|\lambda|)}}{\prod_{i\geq
1}(1-1/q)(1-1/q^2)\ldots(1-1/q^{m_i})},
\end{align*}
where $m_i$ is the multiplicity of $i$ in $\lambda$. Hence
\begin{equation}
\calH_{\lambda}(0,\sqrt{q})
=\frac{q^{g\langle\lambda,\lambda\rangle}}{a_{\lambda}(q)},  
\label{central}\end{equation}
where $a_{\lambda}(q)$ is the cardinality of the centralizer of a
unipotent element of $\GL_n(\F_q)$ with Jordan form of type $\lambda$
\cite[IV, (2.7)]{macdonald}. In particular, when $g=0$, we have
$\calH_{(1^n)}(0,\sqrt{q})=1/|\GL_n(\F_q)|$. 

 It is also not difficult to verify that the Euler specialization
of $\calH_\lambda$ is
\begin{equation}
\label{H-specializ}
\calH_{\lambda}(\sqrt{q},1/\sqrt{q})
=\left(q^{-\frac{1}{2}\langle\lambda,
\lambda\rangle}H_{\lambda}(q)\right)^{2g-2}.
\end{equation}


\subsubsection{Cauchy functions}
\label{cauchy-fctns}
As in the introduction let
$$
\Omega(z,w):=\sum_{\lambda\in \calP} \calH_{\lambda}(z,w)
\prod_{i=1}^k\tilde{H_\lambda}(\x_i;z^2,w^2).
$$
By (\ref{Hduality}) and (\ref{hook-duality}) we have
\begin{equation}
\label{Oduality}
\Omega(z,w)=\Omega(w,z) \,\,\,{\rm and
}\,\,\,\Omega(-z,-w)=\Omega(z,w).
\end{equation}

\begin{lemma}
\label{specializ}
With the specialization
$y_i=q^{i-1}$, we have
$$
\Omega\left(\sqrt{q},\frac{1}{\sqrt{q}}\right)
=\sum_{\lambda\in\mathcal{P}}q^{(1-g)|\lambda|}\left(q^{-n(\lambda)} 
H_{\lambda}(q)\right)^{2g+k-2}
\prod_{i=1}^ks_{\lambda}(\x_i\y).
$$
\label{propchar}
\end{lemma} 
\begin{proof}
Follows from Lemma~\ref{H-Euler} and \eqref{H-specializ}.
\end{proof}

For $\muhat=(\mu^1,\cdots,\mu^k)\in\P^k$, we let
\begin{equation}
\H_\muhat(z,w):=(z^2-1)(1-w^2)\left\langle\Log\,\Omega(z,w),h_\muhat\right\rangle.
 \label{H}
\end{equation}

By (\ref{Oduality}) we have

\begin{equation}\H_\muhat(z,w)=\H_\muhat(w,z) \,\,\,{\rm and}\,\,\,\H_\muhat(-z,-w)=\H_\muhat(z,w). \label{HHduality}\end{equation}

We may recover $\Omega(z,w)$ from the $\H_\muhat(z,w)$'s by the formula:

\beq
\Omega(z,w)=\Exp\left(\sum_{\muhat\in\calP^k}\frac{\H_\muhat(z,w)}{(z^2-1)(1-w^2)}m_\muhat\right).
\label{exp}\eeq

If we want to work with partitions of length at most $l_1,\dots,l_k$,
we can specialize the variables $\x_i=(x_{i,1},x_{i,2},\dots)$ in
Formula (\ref{exp}) to say
$(u_{i,1},u_{i,2},\dots,u_{i,l_i},0,0,\dots)$ for some new independent
variables $u_{i,j}$.  Indeed, this specialization takes any $m_\muhat$
with $l(\mu^i)>l_i$ for some $i$ to zero.

For instance, if $k=1$  the specialization
$\x=(x_1,x_2,\dots)$ to $(T,0,0,\dots)$ in Formula (\ref{exp}) gives
\beq 
\label{exp1} 
\sum_\lambda\calH_\lambda(z,w) T^{|\lambda|}=\Exp\left(\sum_{n\geq
    1}\frac{\H_{(n)}(z,w)}{(z^2-1)(1-w^2)}T^n\right),
\eeq
since for a partition $\mu$ of $n$, we have 
$$
m_\mu(T,0,0,\dots)=\begin{cases}T^n \,\,\,{\rm if}\,\,\,
  \mu=(n),\\0\,\,\,{\rm
    otherwise.}
\end{cases}
$$
and 
$$
\tilde{H}_\mu(T,0,0,\dots;q,t)=\tilde{K}_{(n)\mu}(q,t)T^n=T^n.
$$
The identity $\tilde{K}_{(n)\mu}(q,t)=1$ follows from \cite[Formula
(16)]{garsia-haiman}.  Comparing with the left hand side of (3.5.8)
in~\cite{hausel-villegas} we see that, in the notation of that paper
$\H_{(n)}=\bar H_n$.

\subsection{Mixed Hodge polynomials and polynomial count varieties}

We refer the reader to \cite{hausel-villegas} for details on this
section.  For a complex quasi-projective algebraic variety $X$ we
let $H(X;x,y,z)$ and $H_c(X;x,y,z)$ be its mixed Hodge polynomial and
compactly supported mixed Hodge polynomial, respectively. They satisfy
the following properties. The specialization $H(X;1,1,z)$ is the
Poincar\'e polynomial $P(X;z):=\sum_k\dim H^k(X,\C)\,z^k$ and similarly
with $H_c$ and $P_c$.  The $E$-polynomial of $X$ is
$E(X;x,y)=H_c(X;x,y,-1)=\sum_{i,j,k}(-1)^kh_c^{i,j;k}(X)\,x^iy^j$.
The value $E(X;1,1)$ is the compact Euler characteristic
$\sum_i(-1)^i\dim H_c^i(X,\C)$, which is equal to the ordinary Euler
characteristic by \cite{laumon}. We denote it by $E(X)$.

If $X$ is non-singular of pure dimension $d$, i.e., if $X$ is the
disjoint union of its irreducible components all non-singular of
same dimension $d$, then Poincar\'e duality implies that
$$
h_c^{d-i,d-j;2d-k}(X)=h^{i,j;k}(X), \qquad \text{all } i,j,k,
$$
or, equivalently,
\begin{equation}
\label{poincare}
H_c(X;x,y,t)=(xyt^2)^dH(X;x^{-1},y^{-1},t^{-1}).
\end{equation}

We recall the result of Katz given in the appendix to \cite{hausel-villegas}.

\begin{theorem} Assume that $X/\C$ is polynomial-count with
counting polynomial $P_{X}\in\Z[t]$. Then
$$
\label{katz}
E(X;x,y)=P_{X}(xy).
$$
\end{theorem}

If $X$ is polynomial-count, we put
$E(X;q):=E\left(X;\sqrt{q},\sqrt{q}\right)$ and just call it  the
$E$-polynomial of $X$ to simplify. Note that in this case
$\sum_k(-1)^kh_c^{i,j;k}(X)=0$ if $i\neq j$.

\begin{proposition}
\label{epoly2} Assume that $X$ is polynomial-count and that the
mixed Hodge structure on the compactly supported cohomology
$H^*_c(X)$ is pure. Then
$$
E(X;q)=P_c(X;\sqrt{q}).
$$
\end{proposition}
\begin{proof} By the above remark we have
 $\sum_k(-1)^kh_c^{i,j;k}(X)=0$ if $i\neq j$. Since the only
 non-zero term of this sum is when $k=i+j$, by the purity assumption,
 we get that $(-1)^{i+j}h_c^{i,j;i+j}(X)=0$ if $i\neq j$. Hence
 the non-zero mixed Hodge numbers are all of the form
 $h_c^{i,i;2i}(X)$ and
 $E(X;q)=\sum_ih_c^{i,i;2i}(X)q^i$.\end{proof}

\subsection{Complex characters of $\GL_n(\F_q)$ and
$\gl_n(\F_q)$}\label{characters}

Here we recall how to construct the irreducible characters of
$\GL_n(\F_q)$ and $\gl_n(\F_q)$ using the Deligne-Lusztig theory. We
choose a prime $\ell$ which is invertible in the finite field
$\F_q$. Since Deligne-Lusztig theory uses $\ell$-adic cohomology it
will be more convenient to work with
$\overline{\mathbb{Q}}_{\ell}$-characters instead of complex
characters. Note that there is a non-canonical isomorphism over $\Q$
between the two fields $\C$ and $\overline{\mathbb{Q}}_{\ell}$. The
counting formulas \eqref{forchar} and \eqref{forquiver}, which involve
character values, do not depend on the choice of such an isomorphism.

For a finite group $H$, we denote by ${\rm Irr}\,(H)$ the set of irreducible complex characters of $H$.

\begin{nothing}\textbf{Generalities} 
Let $n\in\Nstar$, we put $\GL_n=\GL_n(\overline{\mathbb{F}}_q)$,
and $\gl_n=\gl_n(\overline{\F}_q)$. Unless specified, here the
letter $G$ will always denote a Levi subgroup of a parabolic
subgroup of $\GL_n$, i.e., a subgroup of $\GL_n$ which is
$\GL_n$-conjugate to some $H=\prod_{i=1}^r \GL_{n_i}$ where
$\sum_{i=1}^rn_i=n$. For short we will say that $G$ is a \emph{Levi
subgroup} of $\GL_n$. If $n_i=1$ for all $i$, then $G$ is a maximal
torus of $\GL_n$. The Lie algebra of $G$ is isomorphic to the Lie
algebra $\calH=\bigoplus_i\gl_{n_i}$ of $H$. Let
$\text{Ad}:G\rightarrow\GL(\mathfrak{g})$ be the adjoint
representation: we have $\text{Ad}(g)x=gxg^{-1}$ for $g\in G$ and
$x\in\mathfrak{g}$. For $g\in G$, we denote by $g_s$ the semisimple
part of $g$ and by $g_u$ the unipotent part of $g$, we have
$g=g_sg_u=g_ug_s$. If $x\in\mathfrak{g}$, we denote respectively by
$x_s$ and $x_n$ the semi-simple part of $x$ and the nilpotent part
of $x$. We then have $x=x_s+x_n$ with $[x_s,x_n]=0$. Let
$x\in\mathfrak{g}$ and let $K$ be a subgroup of $G$, we denote by
$C_K(x)$ the centralizer of $x$ in $K$ with respect to the adjoint
action. If $\mathfrak{k}$ is a Lie subalgebra of $\mathfrak{g}$, we
denote by $C_{\mathfrak{k}}(x)$ the centralizer of $x$ in
$\mathfrak{k}$, i.e.,
$C_{\mathfrak{k}}(x)={\{y\in\mathfrak{k}|\hspace{.05cm}[x,y]=0}\}$.
We denote respectively by $Z_G$ the center of $G$ and by
$z(\mathfrak{g})$ the center of $\mathfrak{g}$. If $L$ is a Levi
subgroup of $G$ (i.e., a Levi subgroup of $\GL_n$ which is contained
in $G$), then we denote by $W_G(L)$ the finite group $N_G(L)/L$
where $N_G(L)$ denotes the normalizer of $L$ in $G$.

Finally, we denote by $G_{\rm uni}$, resp. $\mathfrak{g}_{\rm nil}$,
the subvariety of unipotent elements of $G$, resp. the subvariety of
nilpotent elements of $\mathfrak{g}$.

\end{nothing}

\begin{nothing}\textbf{Frobenius endomorphisms:} We denote by $F:\GL_n\rightarrow \GL_n$,
and $F:\gl_n\rightarrow \gl_n$  the \emph{standard} Frobenius
endomorphisms $(a_{ij})\mapsto (a_{ij}^q)$. Assume that $G$ is
$F$-stable. Then $\mathfrak{g}\subset \gl_n$ is $F$-stable and the
restrictions  $F:G\rightarrow G$,
$F:\mathfrak{g}\rightarrow\mathfrak{g}$ are Frobenius endomorphisms
on $G$ and $\mathfrak{g}$. We also have
$F(\text{Ad}(g)x)=\text{Ad}(F(g))F(x)$, therefore, $\text{Ad}$
induces an action of the finite group $G^F$ on the finite Lie
algebras $\mathfrak{g}^F$. Since $G$ is conjugate to $H$, the
Frobenius endomorphism $F:G\rightarrow G$ corresponds to some
$F':H\rightarrow H$ which we write $(G,F)\simeq (H,F')$. We then
have $G^F\simeq H^{F'}$. The Frobenius endomorphism $F'$ is of the
form $wF:H\rightarrow H$, $h\mapsto wF(h)w^{-1}$ for some $w\in
N_{\GL_n}(H)$. We say that an $F$-stable maximal torus $T\subset G$
of rank $n$ is \emph{split} if there exists an isomorphism $T\simeq
\left(\overline{\F}_q^{\times}\right)^n$ defined over $\F_q$. The
$\F_q$-\emph{rank} of an $F$-stable maximal torus of $G$ is defined
to be the rank of its maximal split subtori. An $F$-stable maximal
torus of $G$ is said to be $G$-\emph{split} if it is maximally split
in $G$. The $G$-split $F$-stable maximal tori of $G$ are those which
are contained in some $F$-stable Borel subgroup of $G$.
\end{nothing}

\begin{nothing} \textbf{$F$-conjugacy classes:}  Let $T$ be an $F$-stable
maximal torus of $G$. The Frobenius $F$ acts on the finite group
$W_G(T)$ and we say that two elements $w,v\in W_G(T)$ are
$F$-conjugate if there exists $h\in W_G(T)$ such that
$w=hv(F(h))^{-1}$. Then we can parametrize the $G^F$-conjugacy
classes of the $F$-stable maximal tori of $G$ by the $F$-conjugacy
classes of $W_G(T)$ as follows. Let $T'$ be an $F$-stable maximal
torus of $G$. Then there exists $g\in G$ such that $T'=gTg^{-1}$,
i.e.,  $g^{-1}F(g)\in N_G(T)$. There is a well-defined  map which
sends the $G^F$-conjugacy class of $T'$ to the $F$-conjugacy class
of the image $w$ of $g^{-1}F(g)$ in $W_G(T)$, moreover this map is
bijective. This parametrization depends only on the $G^F$-conjugacy
class of $T$. If $w\in W_G(T)$, then we will denote by $T_w$ an
arbitrary $F$-stable maximal torus of $G$ which is in the
$G^F$-conjugacy class corresponding to the $F$-conjugacy class of
$w$ in $W_G(T)$, and we will denote by $\mathfrak{t}_w$ its Lie
algebra. Under the isomorphism $T\rightarrow T'$, $h\mapsto
ghg^{-1}$, the Frobenius $F:T'\rightarrow T'$ corresponds to
$F'=wF:T\rightarrow T, h\mapsto \dot{w}F(h)\dot{w}^{-1}$ where $w$
is the image in $W_G(T)$ of $\dot{w}:=g^{-1}F(g)\in N_G(T)$.

Unless specified, we will always consider parameterizations with
respect to $G$-split $F$-stable maximal tori of $G$, in which case
we will write $W_G$ instead of $W_G(T)$
\label{nothing1.1.2'}\end{nothing}

\noindent\textbf{Example:} Let $n=2$, let
$x\in\mathbb{F}_{q^2}-\mathbb{F}_q$, and let
$$T=\left\{\left.\left(\begin{array}{ll}a&0\\0&b\end{array}\right)\,\right|\,a,b\in\overline{\mathbb{F}}_q^{\times}\right\},
\hspace{.5cm}
T'=\left\{\frac{1}{x^q-x}\left.\left(\begin{array}{ll}ax^q-bx&-a+b\\(a-b)xx^q&-ax+bx^q\end{array}\right)\,\right|\,
a,b\in \overline{\mathbb{F}}_q^{\times}\right\}.$$Then $T'$ is
$F$-stable, $T'=gTg^{-1}$ where
$g=\left(\begin{array}{ll}1&1\\x&x^q\end{array}\right),\text{ and }
g^{-1}F(g)=\sigma:=\left(\begin{array}{ll}0&1\\1&0\end{array}\right)$.
Therefore, $(T',F)\simeq (T,\sigma F)$, and we have $T^F\simeq
\mathbb{F}_q^{\times}\times\mathbb{F}_q^{\times}$ and $T'^F\simeq
T^{\sigma F}\simeq \mathbb{F}_{q^2}^{\times}$. Since
$|W_{\GL_2}(T)|=2$, any $F$-stable maximal torus of $\GL_2$ is
either $\GL_2^F$-conjugate to $T$ or $T'$.

\begin{nothing}\textbf{Lusztig induction:}\label{deligne-lusztig} Let
  $\ell\nmid q$ be a prime. Let $L$ be an 
$F$-stable Levi subgroup of a (possibly non $F$-stable) parabolic
subgroup $P$ of $G$. Following \cite{DLu}\cite{Lufinite} we
construct a virtual $\overline{\Q}_{\ell}[G^F]$-module $R_L^G(M)$
for any $\overline{\Q}_{\ell}[L^F]$-module $M$ as follows. Let $U_P$
be the unipotent radical of $P$ and let $\mathcal{L}_G:G\rightarrow
G, g\mapsto g^{-1}F(g)$ be the Lang map. The variety
$\mathcal{L}_G^{-1}(U_P)$ is endowed with a left action of $G^F$ by
left multiplication and with a right action of $L^F$ by right
multiplication. These action induce actions on the $\ell$-adic
cohomology
$H_c^i\left(\mathcal{L}_G^{-1}(U_P),\overline{\mathbb{Q}}_{\ell}\right)$.
The virtual $\overline{\mathbb{Q}}_{\ell}$-vector space
$H_c^*\left(\mathcal{L}_G^{-1}(U_P)\right):=\sum_i(-1)^iH_c^i\left(\mathcal{L}_G^{-1}(U_P),\overline{\mathbb{Q}}_{\ell}\right)$
is thus a virtual
$\overline{\Q}_{\ell}[G^F]$-module-$\overline{\Q}_{\ell}[L^F]$. We
put
$R_L^G(M):=H_c^*\left(\mathcal{L}_G^{-1}(U_P)\right)\otimes_{\overline{\Q}_{\ell}[L^F]}M$.

Let $C(G^F)$ be the $\overline{\mathbb{Q}}_{\ell}$-vector space of
all functions $G^F\rightarrow\overline{\mathbb{Q}}_{\ell}$ with are
constant on conjugacy classes of $G^F$. If $C$ is a conjugacy class
of $G^F$ and $x\in C$, we denote either by $1_C$ or $1_x^G$ the
characteristic function of $C$ that takes the value $1$ on $C$ and
$0$ elsewhere. 

The Lusztig functor $R_L^G$ defines a $\Z$-linear map
$\Z\left(\text{Irr}(L^F)\right)\rightarrow \Z\left(\text{Irr}(G^F)\right)$
which by linearity extension leads to the Deligne-Lusztig induction
$R_L^G:C(L^F)\rightarrow C(G^F)$.

For an $F$-stable maximal torus $T$ of $G$, let $Q_T^G:G_{\rm
  uni}^F\rightarrow \overline{\mathbb{Q}}_{\ell}$ be the restriction
to $G_{\rm uni}^F$ of the function $R_T^G(\rm Id)$. The function
$Q_T^G$ is called a \emph{Green function} and its values are products
of the Green polynomials defined in \cite[III (7.8)]{macdonald} (see
\ref{greenKostka}). The following formula \cite[Theorem 4.2]{DLu}
reduces the computation of the values of $R_T^G(\theta)$ to the
computation of Green polynomials.

\begin{equation}R_T^G(\theta)(g)=|C_G(g_s)^F|^{-1}\sum_{\{h\in
G^F|\hspace{.05cm}g_s\in
hTh^{-1}\}}Q_{hTh^{-1}}^{C_G(g_s)}(g_u)\theta(h^{-1}g_sh)\label{charformula2}\end{equation}where
$\theta\in C(T^F)$, $g\in G^F$. \end{nothing} 

\begin{nothing}\textbf{Characters of $\GL_n(\F_q)$:} The
character table of $\GL_n(\F_q)$ was first computed by Green
\cite{green}. Here we recall how to construct it from the point of
view of Deligne-Lusztig theory \cite{LSr}. 

Here we assume that $G=\GL_n$. Let $L$ be an $F$-stable Levi
subgroup of $G$ and let $\varphi$ be an $F$-stable irreducible
character of $W_L$. Then there is an extension $\tilde{\varphi}$ of
$\varphi$ to the semi-direct product $W_L\rtimes \langle F\rangle$
such that the function
$\mathcal{X}_{\varphi}^L:L^F\rightarrow\overline{\mathbb{Q}}_{\ell}$
defined by $$\mathcal{X}_{\varphi}^L=|W_L|^{-1}\sum_{w\in
W_L}\tilde{\varphi}(wF)R_{T_w}^L({\rm Id}_{T_w})$$is an irreducible
character of $L^F$. The characters $\mathcal{X}_{\varphi}^L$ are
called the \emph{unipotent characters} of $L^F$. 

For $g\in G^F$ and $\theta\in \text{Irr}(L^F)$, let
${^g}\theta\in\text{Irr}(gL^Fg^{-1})$ be defined by
${^g}\theta(glg^{-1})=\theta(l)$. We say that a linear character
$\theta:L^F\rightarrow \overline{\mathbb{Q}}_{\ell}^{\times}$ is
\emph{regular} if for $n\in N_{G^F}(L)$, we have ${^n}\theta=\theta$
only if $n\in L^F$. We denote by $\text{Irr}_{\rm reg}(L^F)$ the set
of regular linear characters of $L^F$. Put
$\epsilon_L=(-1)^{\F_q-rank(L)}$. Then for
$\theta^L\in\text{Irr}_{\rm reg}(L^F)$, the virtual character
\begin{equation}\mathcal{X}:=\epsilon_G\epsilon_LR_L^G\left(\theta^L\cdot\mathcal{X}_{\varphi}^L\right)=
\epsilon_G\epsilon_L|W_L|^{-1}\sum_{w\in
W_L}\tilde{\varphi}(wF)R_{T_w}^G(\theta^{T_w})\label{charform1}\end{equation}
where $\theta^{T_w}:=\theta^L|_{T_w}$, is an irreducible true
character of $G^F$ and any irreducible character of $G^F$ is
obtained in this way \cite{LSr}. An irreducible character of $G^F$
is thus completely determined by the $G^F$-conjugacy class of a
datum $(L,\theta^L,\varphi)$ with $L$ an $F$-stable Levi subgroup of
$G$, $\theta^L\in\text{Irr}_{\rm reg}(L^F)$ and $\varphi\in
\text{Irr}\left(W_L\right)^F$. The irreducible characters corresponding
to the data $(L,\theta^L,1)$ are called \emph{semisimple} characters
of $G^F$. This process of decomposing the irreducible characters is
sometimes called \emph{Lusztig-Jordan decomposition}. By analogy
with Jordan decomposition of conjugacy classes, the \emph{semisimple
part} of $\mathcal{X}$ would be $\theta^L$ and the unipotent part
would be $\mathcal{X}_{\varphi}^L$. It is indeed well-known that if
$C$ is a conjugacy class of $G^F$, $x\in C$, $L=C_G(x_s)$, then
$R_L^G(1_{x_s}^L*1_{x_u}^L)=1_C$ where $*$ is the usual convolution
product on $C(G^F)$ defined by $(f*h)(g)=\sum_{y\in
G^F}f(y)h(gy^{-1})$. 
\end{nothing}

\begin{nothing}\textbf{Characters of $\gl_n(\F_q)$:} The
characters of $\gl_n(\F_q)$ were first studied by Springer
\cite{Sp1}.\label{finiteLiechar}\end{nothing}

We denote by $\text{Fun}(\mathfrak{g}^F)$ the
$\overline{\mathbb{Q}}_{\ell}$-vector space of all functions
$\mathfrak{g}^F\rightarrow\overline{\mathbb{Q}}_{\ell}$ and by
$\mathcal{C}(\mathfrak{g}^F)$ the subspace of all functions
$f:\mathfrak{g}^F\rightarrow\overline{\mathbb{Q}}_{\ell}$ which are
$G^F$-invariant, i.e., for any $h\in G^F$ and any
$x\in\mathfrak{g}^F$, $f(\text{Ad}(h)x)=f(x)$. If $\mathcal{O}$ is a
$G^F$-orbit of $\mathfrak{g}^F$ and $\sigma\in\mathcal{O}$, then we
denote either by $1_{\mathcal{O}}$ or
$1_{\sigma}^G\in\mathcal{C}(\mathfrak{g}^F)$ the characteristic
function of $\mathcal{O}$, i.e., $1_{\sigma}^G(x)=1$ if
$x\in\mathcal{O}$ and $1_{\sigma}^G(x)=0$ otherwise. We are
interested in the characters (non-necessarily irreducible) of the
abelian group $(\mathfrak{g}^F,+)$ which are $G^F$-invariant, i.e.,
which are in $\mathcal{C}(\mathfrak{g}^F)$. We call them the
\emph{invariant characters} of $\mathfrak{g}^F$. We say that an
invariant character of $\mathfrak{g}^F$ is \emph{irreducible} if it
can not be written as a sum of two invariant characters. We denote
by $\text{Irr}_{G^F}(\mathfrak{g}^F)$ the set of irreducible
invariant characters of $\mathfrak{g}^F$. We now describe them in
terms of Fourier transforms. 

We fix once for all a non-trivial additive character
$\Psi:\mathbb{F}_q\rightarrow\overline{\mathbb{Q}}_{\ell}^{\times}$
and we denote by $\mu:\mathfrak{g}\times\mathfrak{g}\rightarrow
\overline{\mathbb{F}}_q$ the trace map $(a,b)\mapsto \text{Trace}
(ab)$. It is a non-degenerate $G$-invariant symmetric bilinear form
defined over $\mathbb{F}_q$. We define the Fourier transform
$\mathcal{F}^{\mathfrak{g}}:
\text{Fun}(\mathfrak{g}^F)\rightarrow\text{Fun}(\mathfrak{g}^F)$ with
respect to $(\Psi,\mu)$ by
\begin{equation}
\label{Lie-alg-Fourier-transf}
\mathcal{F}^{\mathfrak{g}}(f)(x)=
\sum_{y\in\mathfrak{g}^F}\Psi\left(\mu(x,y)\right)\,f(y).
\end{equation}
Note that for $\sigma,x\in\mathfrak{g}^F$,
$$
\mathcal{F}^{\mathfrak{g}}(1_{\sigma}^G)(x)=
\sum_{y\in\mathcal{O}_{\sigma}^{G^F}}\Psi(\mu(x,y)).
$$
For a fixed $y\in\mathfrak{g}^F$, the map
$\mathfrak{g}^F\rightarrow\overline{\mathbb{Q}}_{\ell}$, $x\mapsto
\Psi\left(\langle x,y\rangle\right)$ is an irreducible characters of
the abelian finite group $(\mathfrak{g}^F,+)$. Therefore
$\mathcal{F}^{\mathfrak{g}}(1_{\sigma}^G)$, being a sum of
characters of $(\mathfrak{g}^F,+)$, is a character of
$(\mathfrak{g}^F,+)$. Since the sum is over a single adjoint orbit
it is clearly an irreducible invariant character, i.e.,
$\mathcal{F}^{\mathfrak{g}}(1_{\sigma}^G)\in\text{Irr}_{G^F}(\mathfrak{g}^F)$.

Let $L$ be an $F$-stable Levi subgroup of $G$ and let $\mathfrak{l}$
be its Lie algebra. We also have a Deligne-Lusztig induction
$\mathcal{C}(\mathfrak{l}^F)\rightarrow\mathcal{C}(\mathfrak{g}^F)$
defined in \cite{letellier1}. Let $\omega:\mathfrak{g}_{\rm
 nil}\rightarrow G_{\rm uni}$ be the $G$-equivariant isomorphism
given by $v\mapsto v+1$. For an $F$-stable maximal torus $T$ of $G$
with $\mathfrak{t}:=\text{Lie}(T)$, the Deligne-Lusztig induction
$\mathcal{R}_{\mathfrak{t}}^{\mathfrak{g}}$ is defined by the
following character
formula:\begin{equation}\mathcal{R}_{\mathfrak{t}}^{\mathfrak{g}}(\theta)(x)=
 |C_G(x_s)^F|^{-1}\sum_{\{h\in G^F| \hspace{.05cm}
   x_s\in\text{Ad}(h)\mathfrak{t}\}}Q_{hTh^{-1}}^{C_G(x_s)}(\omega(x_n))
 \,\theta(\text{Ad}(h^{-1})x_s)\label{charformula1}\end{equation}
where $\theta\in \mathcal{C}(\mathfrak{t}^F)$ and
$x\in\mathfrak{g}^F$. Note that $C_G(x_s)$ is a Levi subgroup of
$G$. For any semisimple element $\sigma\in\mathfrak{g}^F$, we have the
following character formula
\cite[7.3.3]{letellier}:
\begin{equation}
\label{charform2}
\mathcal{F}^{\mathfrak{g}}(1_{\sigma}^G)=\epsilon_G\epsilon_L|W_L|^{-1}
\sum_{w\in
 W_L}q^{d_L/2}\mathcal{R}_{\mathfrak{t}_w}^{\mathfrak{g}}
\left(\mathcal{F}^{\mathfrak{t}_w}(1_{\sigma}^{T_w})\right)
\end{equation}
where $L=C_G(\sigma)$,
$d_L=\text{dim}\hspace{.05cm}G-\text{dim}\hspace{.05cm}L$.

Note that if $\mathcal{X}$ is a semisimple character of
$\GL_n(\F_q)$, then it is given by Formula (\ref{charform1}) with
$\tilde{\varphi}=1$. Hence the character formulas for semisimple
characters of $\GL_n(\F_q)$ and $\gl_n(\F_q)$ are similar and can be
computed in the same way. 

\section{Counting with Fourier transforms}\label{FourieronG}

Let $K$ be an algebraically closed field isomorphic to $\C$. Fixing
such an isomorphism gives us an involution $K\rightarrow K$, $x\mapsto \overline{x}$ such
that $\overline{\zeta}=\zeta^{-1}$ for any root of unity $\zeta$ in
$K$.

\subsection{Group Fourier transform} \label{group-fourier}

Let $G$ be a finite group. We construct an analogue of the Fourier
transform for class functions of $G$. For convenience we introduce
the following notation. Let $G_\bullet$ be the measure space
consisting of $G$ with its Haar measure, i.e., such that the measure
of $\{g\}$ for $g\in G$ is $1/|G|$. Clearly, the total mass of
$G_{\bullet}$ is $1$. Let $C(G_{\bullet})$ be the $K$-vector space
of class functions on $G$ (a {\it class function} on $G$ is a
function which is constant on conjugacy classes).

Similarly, let $G^{\bullet}$ be the measure space on the set of
irreducible characters of $G$ with its Plancherel measure, i.e.,
such that the measure of the set $\{\chi\}$ for an irreducible
character $\chi$ of $G$ is $\chi(1)^2/|G|$. Again, the total mass of
$G^{\bullet}$ is $1$. Let $C(G^{\bullet})$ be the $K$-vector space
of functions on $G^{\bullet}$.

We now define maps $\calF_{\bullet}$ and $\calF^\bullet$ which are
analogues of the Fourier transform for $G$. We describe some of
their formal properties leaving their proofs  to the reader.

Define $\calF_\bullet: C(G_\bullet) \rightarrow C(G^\bullet)$ by
$$
\calF_\bullet(f)(\chi):=|G|\int_{G_\bullet}f(g)\,
\frac{\chi(g)}{\chi(1)}\,dg= \sum_g f(g)\,\frac{\chi(g)}{\chi(1)}
$$
and $\calF^\bullet: C(G^\bullet) \rightarrow C(G_\bullet)$ by
$$
\calF^\bullet(F)(g):=|G|\int_{G^\bullet}F(\chi)\,
\overline{\left(\frac{\chi(g)}{\chi(1)}\right)}\,d\chi= \sum_\chi
F(\chi)\,\chi(1)\overline{\chi}(g).
$$

 Up to a factor of $|G|$ these maps are mutual inverses
of each other. More precisely,
\begin{equation}
\label{fourier-inversion} \calF^\bullet\circ\calF_\bullet=|G|\cdot
1_{G_\bullet}, \qquad \qquad
\calF_\bullet\circ\calF^\bullet=|G|\cdot 1_{G^\bullet}.
\end{equation}

Consider the algebra structures on  $C(G_\bullet)$ and
$C(G^\bullet)$ defined by  convolution and pointwise multiplication,
respectively. I.e.,
$$
(f_1*f_2)(g):=\sum_{g_1g_2=g}f_1(g_1)f_2(g_2), \qquad f_1,f_2 \in
C(G_\bullet)
$$
and
$$
(F_1\cdot F_2)(\chi):=F_1(\chi)F_2(\chi), \qquad F_1,F_2 \in
C(G^\bullet)
$$
(it is easy to check that $f_1*f_2$ is indeed a class function and
hence belongs to $C(G_\bullet)$).

The maps $\calF_\bullet$ and $\calF^\bullet$ preserve these
operations,
$$
\calF_\bullet(f_1)\cdot\calF_\bullet(f_2)=\calF_\bullet(f_1*f_2),
\qquad f_1,f_2\in C(G_\bullet),
$$
and
$$
\calF^\bullet(F_1)*\calF^\bullet(F_2)=|G|\cdot\calF^\bullet(F_1\cdot
F_2), \qquad F_1,F_2\in C(G^\bullet).
$$

\begin{proposition}
\label{eval-at-1} For $f\in C(G_\bullet)$ we have
$$
f(1)=\int_{G^\bullet}\calF_\bullet(f)(\chi)\,d\chi.
$$
\end{proposition}
\begin{proof}
  This is just a special case of Fourier inversion
  \eqref{fourier-inversion} as both sides equal
$\frac{1}{|G|}\,\calF^\bullet(\calF_\bullet(f))(1)$.
\end{proof}
Given a word $w \in F_r$, where $F_r=\langle X_1,\ldots,X_r\rangle$
is the free group in generators $X_1,\ldots, X_r$, we let $n(w)$ be
the function on $G$ defined by
$$
n(w)(z):=\#\{(x_1,\ldots,x_r) \in G^r \;|\; w(x_1,\ldots,x_r)=z\},
$$
where $w(x_1,\ldots,x_r)$ is a shorthand for $\phi(w)\in G$ with
$\phi:F_r\longrightarrow G$ the homomorphism mapping each $X_i$ to
$x_i$.

Since $w(x_1,\ldots,x_r)=z$ implies
$w(ux_1u^{-1},\ldots,ux_ru^{-1})=uzu^{-1}$ for any $u\in G$ it is
clear that $n(w)$ is a class function. For convenience we define
$$
N(w):=\calF_\bullet(n(w)) \in C(G^\bullet).
$$
The following lemma is straightforward and we omit its proof.
\begin{lemma}
1) For a word $w\in F_r$
$$
n(w)(1)=\int_{G^\bullet}N(w)(\chi)\, d\chi
$$
2) If $w_1,w_2$ are two words in separate sets of variables then
$$
n(w_1w_2)=n(w_1)*n(w_2).
$$
3) Let $C_1,\ldots,C_k$ be conjugacy classes in $G$ and $w\in F_r$.
For $z\in G$ the number of solutions  to
$$
 w(x_1,\ldots,x_r)y_1\cdots y_k=z, \qquad x_i\in G, y_j\in C_j
$$
is given by
$$
n(w)*1_{C_1}*\cdots*1_{C_k}(z),
$$
where for any conjugacy class $C$ we denote by $1_C\in C(G_\bullet)$
its characteristic function.
\end{lemma}

A proof of the following result can be found in
\cite[3.2]{hausel-villegas}.
\begin{lemma}
\label{commutator-sum} For $w=X_1X_2X_1^{-1}X_2^{-1}\in F_2$ we have
$$
N(w)(\chi)=\left(\frac{|G|}{\chi(1)}\right)^2.
$$
\end{lemma}
Finally, putting all the pieces together we see we have the
following result.
\begin{proposition}
\label{group-count} Let $C_1,\ldots,C_k$ be conjugacy classes in
$G$. The number of solutions to
$$
[x_1,y_1]\cdots[x_g,y_g]z_1\cdots z_k=1, \qquad x_i,y_i\in G, z_j\in
C_j,
$$
equals
\begin{equation}
\int_{G^\bullet}\Lam(\chi)^g\,f_\chi(C_1)\cdots f_\chi(C_k)\,d\chi,
\label{form-int}\end{equation}
where
$$
\Lam(\chi):=\left(\frac{|G|}{\chi(1)}\right)^2
$$
and for any conjugacy class $C$
$$
f_\chi(C):=\calF_\bullet(1_C)(\chi)=\frac{|C|\,\chi(C)}{\chi(1)}.
$$
\end{proposition}
\begin{remark}
  The proposition, as well as the introduction of the functions
  $f_\chi$, is due to Frobenius. Proofs can be found in many places in
  the literature since then. The purpose of reproving here is to draw
  as close a parallel as possible with the additive version of the
  next section.
\end{remark}

\subsection{Equivariant Fourier Transform} If the group $G$ of the
previous section is abelian what we have is the usual Fourier
transform. Here we consider the situation of an abelian group, which
will now denote by $A$, together with an action of another group
$G$. We will describe a Fourier transform on $A$, which is
equivariant with respect to the action of $G$ and parallels the one
in \S\ref{group-fourier}. We will apply this to our main example:
$A=\gl_n$ and $G=\GL_n$ acting via the adjoint action.

Let $A_\bullet$ be as in \S\ref{group-fourier} and let
$X:=\Hom(A,K^\times)$. We have a natural action of $G$ on $X$ as
follows.
$$
(g\cdot\phi)(a):=\phi(g^{-1}\cdot a).
$$
Given a  $G$-orbit $\calX$ in $X$ we let
$$
\chi:=\sum_{\phi\in\calX}\phi.
$$
It is a $G$-invariant character of the group $A$. We let $A^\bullet$
be the measure space on the set of such $\chi$'s where the measure
of $\{\chi\}$ is $\chi(0)/|A|$. The total measure of $A^\bullet$ is
$1$ as $\chi(0)=\#\calX$.

In analogy with \S\ref{group-fourier} we let $C(A_\bullet)$ be the
$K$-vector space of functions on $A$ which are $G$-invariant and let
$C(A^\bullet)$ be the $K$-vector space of functions on $A^\bullet$.

Define $\calF_\bullet: C(A_\bullet) \rightarrow C(A^\bullet)$ by
$$
\calF_\bullet(f)(\chi):=|A|\int_{A_\bullet}f(a)\,
\frac{\chi(a)}{\chi(0)}\,da
$$
and $\calF^\bullet: C(A^\bullet) \rightarrow C(A_\bullet)$ by
$$
\calF^\bullet(F)(a):=|A|\int_{A^\bullet}F(\chi)\,
\overline{\left(\frac{\chi(a)}{\chi(0)}\right)}\,d\chi.
$$
Note that if $f:A\rightarrow K$ is constant on $G$-orbits and
$\phi\in X$ then
$$
\sum_{a\in A}f(a)\,\phi(a)
$$
is constant on the $G$-orbit $\calX$ of $\phi$. Hence we can write
this sum as
$$
\sum_{a\in A} f(a)\, \frac{\chi(a)}{\chi(0)},
$$
where $\chi$ corresponds to $\calX$.  In other words,
$\calF_\bullet$ is (up to scaling) just the usual Fourier transform
restricted to $G$-invariant functions on $A$. Similarly,
$\calF^\bullet$ is the usual inverse Fourier transform (up to
scaling) restricted to $G$-stable characters of $A$. It follows that
all the formal properties of the previous section also hold here. In
particular, we have
\begin{proposition}
For $f\in C(A_\bullet)$ we have
$$
f(0)=\int_{A^\bullet}\calF_\bullet(f)(\chi)\,d\chi.
$$
\end{proposition}

Now let $A=\gl_n(\F_q)$ and $G=\GL_n(\F_q)$ acting via the adjoint
action on $A$. We consider the additive analogue of
Proposition~\ref{group-count}. For $x,y\in A$ let $[x,y]:=xy-yx$.  For
fixed $\phi \in X$ and $y\in A$ the map $x\mapsto \phi([x,y])$ is in
$X$. Let $C_A(\phi)$ be the subgroup of $y\in A$ for which this
character is trivial. Its cardinality only depends on the $G$-orbit
$\calX$ of $\phi$ and the order of $A/C_A(\phi)$ is a square since it
carries the non-degenerate pairing induced from
$\phi([\cdot,\cdot])$. Define $c(\chi):=|A/C_A(\phi)|^{\tfrac12}$,
where $\chi=\sum_{\phi\in \calX}\phi\in A^\bullet$ is associated to
$\calX$.

\begin{proposition}
\label{algebra-case}
Let $\calO_1,\ldots,\calO_k$ be $G$-orbits in $A$.  The number of
solutions to
$$
[x_1,y_1]+\cdots+[x_g,y_g]+z_1+\cdots +z_k=0, \qquad x_i,y_i\in A,
z_j\in \calO_j,
$$
equals
$$
\int_{A^\bullet}\Lam(\chi)^g\,f_{\chi}(\calO_1)\cdots
f_{\chi}(\calO_k)\,d\chi,
$$
where
$$
\Lam(\chi):=\left(\frac{|A|}{c(\chi)}\right)^2
$$
and  
$$
f_{\chi}(\calO) :=\calF_\bullet(1_{\calO})(\chi)
=\frac{|\calO|\,\chi(\calO)}{\chi(0)}.
$$  
\end{proposition}

\begin{proof}
  We may proceed exactly as with the proof of
  Proposition~\ref{group-count} thanks to the formal properties of the
  Fourier transform. The analogue of Lemma~\ref{commutator-sum} is the
  following calculation. Let $n\in C(A_\bullet)$ be the function whose
  value at $a\in A$ is the number of solutions $x,y\in A$ of
  $[x,y]=a$. Then with our previous notation
$$
\calF_\bullet(n)(\chi)=\sum_{a\in A}n(a)\,\phi(a) =\sum_{x,y\in
A}\phi([x,y]).
$$
The sum $\sum_{x\in A}\phi([x,y])$ vanishes unless $y\in C_A(\phi)$
in which case it equals $|A|$. Hence
$\calF_\bullet(n)(\chi)=|A|\,|C_A(\phi)|$.
\end{proof}

\begin{remark} With the notation of \S\ref{finiteLiechar},  the $\GL_n(\F_q)$-invariant characters of
$\gl_n(\F_q)$ are the functions $\calF^{\mathfrak{g}}(1_{\calO})$
where $\calO$ describes the set of adjoint orbits. Then note that
$$\calF_\bullet(1_{\calO\,'})\left(\calF^{\mathfrak{g}}(1_{\calO})\right)=\calF^{\mathfrak{g}}(1_{\calO\,'})(\calO).$$
and $c\left(\calF^{\mathfrak{g}}(1_{\calO})\right)=\left(|\gl_n(\F_q)|
q^{-\text{dim}\,C_G(x)}\right)^{\frac{1}{2}}$ where $x\in \calO$.
\label{remarkfour}\end{remark}

\section{Sums of character values}\label{Sumsofcharval}

In this section we obtain a formula which will be used, together
with the results of \S\ref{FourieronG}, to compute the number of
$\F_q$-rational points of character and quiver varieties over
$\overline{\mathbb{F}}_q$. Here
$G=\GL_n(\overline{\F}_q)$.

\subsection{Types of conjugacy classes, irreducible characters and Levi
subgroups}\label{types}

Let $C$ be a conjugacy class of $G^F$. The Frobenius
$f:\overline{\mathbb{F}}_q\rightarrow\overline{\mathbb{F}}_q,
x\mapsto x^q$ acts on the set of eigenvalues of $C$, therefore we
may write the set of eigenvalues of $C$ as a union of $\langle
f\rangle$-orbits$$\{\gamma_1,\gamma_1^q,\dots\}\coprod\{\gamma_2,\gamma_2^q,\dots\}\coprod\cdots\coprod\{\gamma_s,\gamma_s^q,\dots\}$$
Put $d_i=\sharp\{\gamma_i,\gamma_i^q,\dots\}$ and let $m_i$ be the
multiplicity of $\gamma_i$. Clearly $\sum_im_id_i=n$. The unipotent
part of an element of $C$ defines a unique partition $\lambda^i$ of
$m_i$ given by the Jordan blocks. Then
$\lambda=(d_1,\lambda^1)\cdots(d_s,\lambda^s)\in\bold{T}_n$ is called
the \emph{type} of $C$. When $q\geq n$, any type $\omega\in\bold{T}_n$
arises as the type of some conjugacy class of $G^F$. The types of
the semisimple conjugacy classes are of the form
$(d_1,1^{n_1})\cdots(d_r,1^{n_r})$ where $n_1,\dots,n_r$ are the
multiplicities of the eigenvalues and $\left(1^{n_i}\right)$ is the trivial
partition $(1,\dots,1)$ of $n_i$. 

\begin{lemma} Let $\omega\in\bold {T}_n$ and let $\sigma\in G^F$ be an element of type $\omega$. Then

$$\calH_\omega(0,\sqrt{q})=\frac{q^{g\,{\rm dim}\, C_G(\sigma)}}{|C_{G^F}(\sigma)|}$$where $\calH_\omega(z,w)=\prod_i\calH_{\omega^i}(z^{d_i},w^{d_i})$ for $\omega=(d_1,\omega^1)\cdots(d_r,\omega^r)$.
\label{Homega}\end{lemma}

\begin{proof} This follows from Formula (\ref{central}) and the identities ${\rm dim}\, C_G(\sigma)=\sum_{i=1}^rd_i\langle\omega^i,\omega^i\rangle$ and $|C_{G^F}(\sigma)|=\prod_{i=1}^ra_{\omega^i}(q^{d_i})$ which are well-known.

\end{proof}

Recall (see \S\ref{characters}) that an irreducible character
$\mathcal{X}$ of $G^F$ arises from a datum $(L,\theta^L,\varphi)$.
There exist positive integers $d_i,n_i$, $i\in\{1,\dots,s\}$ such
that $$L\simeq \prod_{i=1}^s\GL_{n_i}(\overline{\F}_{q})^{d_i}.$$We
choose the indexing such that $d_1\geq d_2\geq\cdots\geq d_s$, and
$n_i\geq n_j$ if $i>j$ and $d_i=d_j$. Let $\mathcal{S}_n$ be the
symmetric group in $n$ letters and let $\nu\in\mathcal{S}_n\simeq
W_G$, where $W_G$ is the Weyl group of $G$ (with respect to some
split $F$-stable maximal torus), be such that the map $z\mapsto \nu
z\nu^{-1}$ acts on each component of
$\prod_{i=1}^s\GL_{n_i}(\overline{\F}_{q})^{d_i}$ by circular
permutation of the $d_i$ blocks of length $n_i$. Then
\begin{equation}(L,F)\simeq\left(\prod_{i=1}^s\GL_{n_i}(\overline{\F}_{q})^{d_i},\nu
F\right)\label{typelevi}\end{equation}and so $L^F$ is isomorphic to
$\prod_{i=1}^s\GL_{n_i}(\F_{q^{d_i}})$. Moreover
$$(W_L,F)\simeq\left(\prod_{i=1}^s\left(\mathcal{S}_{n_i}\right)^{d_i},\nu\right).$$The
$F$-conjugacy classes of $W_L$ are thus parametrized by the
conjugacy classes of $\prod_i\mathcal{S}_{n_i}$, i.e., by the set
$\mathcal{P}_{n_1}\times\cdots\times\mathcal{P}_{n_s}$. The set of
$F$-stable irreducible characters of $W_L$ is in bijection with
${\rm Irr}(\mathcal{S}_{n_1})\times\cdots\times{\rm
Irr}(\mathcal{S}_{n_s})$ which, by the Springer correspondence, is
parametrized by $\mathcal{P}_{n_1}\times\cdots\times\mathcal{P}_{n_s}$
in such a way that the trivial character corresponds to the
multi-partition $((n_1),\dots,(n_s))$. Hence $\varphi\in{\rm
Irr}(W_L)^F$ defines a partition $\lambda^i\in\mathcal{P}_{n_i}$ for
all $i\in\{1,\dots,s\}$. The type
$(d_1,\lambda^1)(d_2,\lambda^2)\cdots(d_s,\lambda^s)\in\bold{T}_n$ is
called the \emph{type} of the irreducible character $\mathcal{X}$ of
$G^F$. Note that when $q\geq n$ then any type in $\bold{T}_n$ arises as the type of some
irreducible character of $G^F$. The type of the semisimple
irreducible characters of $G^F$ are of the form
$(d_1,(n_1))(d_2,(n_2))\cdots(d_s,(n_s))$.

It will be convenient to introduce the set $\hat{\bold{T}}_n$ of
non-increasing sequences $(d_1,n_1)\cdots(d_r,n_r)$ with
$d_i,n_i\in\Nstar$ and $\sum_id_in_i=n$ where $(d,k)>(d',k')$ if
$d>d'$, or $d=d'$ and $k>k'$.

The types of the semisimple conjugacy classes are in bijection with
$\hat{\bold{T}}_n$ by $$(d_1,1^{n_1})\cdots(d_r,1^{n_r})\mapsto
(d_1,n_1)\cdots(d_r,n_r).$$Similarly $\hat{\bold{T}}_n$ parametrizes
the types of the semisimple irreducible characters of $G^F$ by
$$(d_1,(n_1))\cdots(d_r,(n_r))\mapsto(d_1,n_1)\cdots(d_r,n_r).$$
The map which assigns to a semisimple element of $G$ the Levi
subgroup $C_G(\sigma)$ gives a natural bijection between the types
of the semisimple conjugacy classes of $G^F$  and the
$G^F$-conjugacy classes of the $F$-stable Levi subgroups of $G$. We
will use the set $\hat{\bold{T}}_n$ to parametrize the
$G^F$-conjugacy classes of the $F$-stable Levi subgroups of $G$.
Namely if $\lambda=(d_1,n_1)\cdots(d_r,n_r)\in\hat{\bold{T}}_n$, then
a representative $L$ of the corresponding $G^F$-conjugacy class
will satisfy (\ref{typelevi}). In this case we say that $L$ is of
\emph{type} $\lambda$.

\subsection{Generic characters and generic conjugacy classes}

Let $L$ be an $F$-stable Levi subgroup of $G$. We say that a linear
character $\Gamma$ of $Z_L^F$ is \emph{generic} if its restriction
to $Z_G^F$ is trivial and its restriction to $Z_M^F$ is non-trivial
for any $F$-stable proper Levi subgroup $M$ of $G$ such that
$L\subset M$. We put $$\left(Z_L\right)_{\rm reg}:=\{x\in
Z_L|\hspace{.05cm}C_G(x)=L\}.$$We have the following proposition.

\begin{proposition} Assume that $L$ is of type
$\omega=(d_1,n_1)(d_2,n_2)\cdots(d_r,n_r)\in\hat{\bold{T}}_n$ and that
$\Gamma$ is a generic linear character of $Z_L^F$. Then
$$\sum_{z\in (Z_L)_{\rm reg}^F}\Gamma(z)=(q-1)K_{\omega}^o$$with
$$K_{\omega}^o=\begin{cases}(-1)^{r-1}d^{r-1}\mu(d)(r-1)!&\text{ if
for all } i, d_i=d\\0&\text{ otherwise}\end{cases}$$where $\mu$ is
the ordinary M\"obius function. \label{Kw}\end{proposition}

\begin{proof}Let $\nu_{\omega}$ be an element of $\mathcal{S}_n$ such that
the map $z\mapsto \nu_{\omega}z\nu_{\omega}^{-1}$ induces an action
on each component of
$M:=\prod_i\GL_{n_i}(\overline{\mathbb{F}}_q)^{d_i}$ by circular
permutation of the $d_i$ blocks of length $n_i$. Then $(L,F)\simeq
\left(M,F_{\omega}\right)$ where $F_{\omega}$ is the Frobenius on $G$
defined by $F_{\omega}(g)=\nu_{\omega}F(g)\nu_{\omega}^{-1}$. Then
the character $\Gamma$ can be transferred to a generic character
$\Gamma_M$ of $Z_M^{F_{\omega}}$. Its restriction to
$Z_G^{F_{\omega}}$ is also trivial. Then  $\sum_{h\in (Z_M)_{\rm
reg}^{F_{\omega}}}\Gamma_M(h)=\sum_{h\in (Z_L)_{\rm
reg}^F}\Gamma(h)$. We denote by $P(\omega)$ the set of Levi
subgroups $H$ of $G$ such that $M\subset H\subset G$ and
$P(\omega)^{F_{\omega}}$ the elements of $P(\omega)$ fixed by
$F_{\omega}$. We have the following partition $Z_M=\coprod_{H\in
P(\omega)}(Z_H)_{\rm reg}$. Indeed, if $z\in Z_M$, then $C_G(z)$ is
a Levi subgroup $H$ of $G$ and clearly $z\in (Z_H)_{\rm reg}$. If
$H\in P(\omega)$ then $F_{\omega}(H)\in P(\omega)$, and $(Z_H)_{\rm
reg}\cap (Z_{F_{\omega}(H)})_{\rm reg}=\emptyset$ unless $H\in
P(\omega)^{F_{\omega}}$. Therefore $F_{\omega}$ preserves the above
partition, and $Z_M^{F_{\omega}}=\coprod_{H\in P(\omega)}(Z_H)_{\rm
reg}^{F_{\omega}}$. We define a partial order on $P(\omega)$ by,
$H_1\leq H_2$ if $Z_{H_1}\subset Z_{H_2}$ (i.e. if $H_2\subset
H_1$). Then $G$ is the unique minimal element and $M$ is the unique
maximal element. We have a map $\epsilon:
P(\omega)^{F_{\omega}}\rightarrow \overline{\mathbb{Q}}_{\ell}$ that
sends $H\in P(\omega)^{F_{\omega}}$ to $\sum_{z\in
Z_H^{F_{\omega}}}\Gamma_M(z)$ and a map $\epsilon':
P(\omega)^{F_{\omega}}\rightarrow \overline{\mathbb{Q}}_{\ell}$ that
sends $H\in P(\omega)^{F_{\omega}}$ to $\sum_{z\in (Z_H)_{\rm
reg}^{F_{\omega}}}\Gamma_M(z)$. Since
$Z_H^{F_{\omega}}=\coprod_{E\leq H}(Z_E)_{\rm reg}^{F_{\omega}}$ for
all $H\in P(\omega)^{F_{\omega}}$, we have $\epsilon(H)=\sum_{E\leq
H}\epsilon'(E)$ for all $H\in P(\omega)^{F_{\omega}}$. Then by
inclusion-exclusion principle, we have $\epsilon'(H)=\sum_{E\leq
H}\mu_{\omega}(E,H)\epsilon(E)$ for all $H\in
P(\omega)^{F_{\omega}}$ where $\mu_{\omega}$ is the M\"obius
function on the poset $P(\omega)^{F_{\omega}}$. In particular
$$\sum_{z\in (Z_M)_{\rm reg}^{F_{\omega}}}\Gamma_M(z)=\sum_{H\leq
M}\mu_{\omega}(H,M)\sum_{z\in Z_H^{F_{\omega}}}\Gamma_M(z).$$Using
the assumption on $\Gamma$, we deduce that
$$\sum_{z\in
(Z_M)_{\rm reg}^{F_{\omega}}}\Gamma_M(z)=(q-1)\mu_{\omega}(G,M)$$Let
us compute $\mu_{\omega}(G,M)$. An element of $Z_M$ is a diagonal
matrix $A\in\prod_{i=1}^r\GL_{n_i}(\overline{\mathbb{F}}_q)^{d_i}$
such that each component of $A$ in
$\GL_{n_i}(\overline{\mathbb{F}}_q)$ is central. We identify $Z_M$
with $\prod_{i=1}^r\left(\overline{\mathbb{F}}_q^{\times}\right)^{d_i}$
in the obvious way. Then the elements of $(Z_M)_{\rm reg}$
correspond to the elements of the form $(a_{k,s})_{1\leq k\leq
r,1\leq s\leq
d_k}\in\prod_{i=1}^r\left(\overline{\mathbb{F}}_q^{\times}\right)^{d_i}$
where $a_{i,j}\neq a_{k,l}$ if $(i,j)\neq(k,l)$. Let
$I={\{i_{1,1},\dots,i_{1,d_1},i_{2,1},\dots,i_{2,d_2},\dots,i_{r,1},\dots,i_{r,d_r}}\}$
be a set whose elements are indexed by the pairs $(k,s)$ with $1\leq
k\leq r$ and $1\leq s\leq d_k$. Then the partition
$Z_M=\coprod_{H\in P(\omega)}(Z_H)_{\rm reg}$ is indexed by the
partitions of the set $I$. The part $(Z_M)_{\rm reg}$ corresponds to
the unique partition of $I$ which has $|I|$ parts, i.e. to
${\{i_{1,1}}\},{\{i_{1,2}}\},\dots,{\{i_{r,d_r}}\}$, and the part
$Z_G=(Z_G)_{\rm reg}$ which is the set of diagonal matrices with
exactly one eigenvalue corresponds to the unique partition of $I$
which has one part. By abuse of notation we denote by
$\nu_{\omega}\in \mathcal{S}_{|I|}$ the element which acts by
circular permutation on each subset ${\{i_{k,1},\dots,i_{k,d_k}}\}$ of
$I$. Then it induces an action on the set $P(I)$ of partitions of
$I$ which corresponds via the bijection $P(I)\simeq P(\omega)$ to
the action of $F_{\omega}=\nu_{\omega}F$ on $P(\omega)$. We denote
by $O$ the minimal element of $P(I)^{\nu_{\omega}}$ and by $1$ the
unique maximal element of $P(I)^{\nu_{\omega}}$. Then
$\mu_{\omega}(G,M)=\mu'_{\omega}(0,1)$ where $\mu'_{\omega}$ is the
M\"obius function on the poset $P(I)^{\nu_{\omega}}$. Now
$\mu'_{\omega}(0,1)$ was computed by Hanlon \cite{hanlon}, and we
find that $\mu'_{\omega}(0,1)=K_{\omega}^o$.\end{proof}

\begin{definition} Let $\mathcal{X}_1,\dots,\mathcal{X}_k$ be
$k$-irreducible characters of $G^F$. For each $i$, let
$(L_i,\theta_i,\varphi_i)$ be a datum defining $\mathcal{X}_i$. We
say that the tuple $(\mathcal{X}_1,\dots,\mathcal{X}_k)$ is
\emph{generic} if $\prod_{i=1}^k\left({^{g_i}}\theta_i\right)|_{Z_M}$
is a generic character of $Z_M^F$ for any $F$-stable Levi subgroup
$M$ of $G$ which satisfies the following condition: For all
$i\in\{1,\dots,k\}$, there exists $g_i\in G^F$ such that $Z_M\subset
g_iL_ig_i^{-1}$. \label{genericcondi}\end{definition}

Let $C_1,\dots,C_k$ be $k$-conjugacy classes of $G^F$. For each
$i\in\{1,\dots,k\}$, let $s_i$ be the semisimple part of an element of
$C_i$. Let $\tilde{\bold{C}}_i$ be the conjugacy class of $s_i$ in
$G$. We say that the tuple $(C_1,\dots,C_k)$ is generic if
$(\tilde{\bold{C}}_1,\dots,\tilde{\bold{C}}_k)$ is generic in the sense
of Definition \ref{genericconjugacy}. 

The proof of the following proposition is similar to that of Proposition
\ref{Kw}.

\begin{proposition} Let $(C_1,\dots,C_k)$ be a generic tuple of semisimple conjugacy
classes of $G^F$, let $s_i\in C_i$ and put $L_i=C_G(s_i)$. Assume
that $M$ is an $F$-stable Levi subgroup of $G$ of type
$\omega\in\hat{\bold{T}}_n$ which satisfies the following condition:
For all $i\in\{1,\dots,k\}$ there exists $g_i\in G^F$ such that
$Z_M\subset g_iL_ig_i^{-1}$. Then
$$\sum_{\theta\in{\rm Irr}_{\rm
reg}(M^F)}\prod_{i=1}^k\theta(g_is_ig_i^{-1})=(q-1)K_{\omega}^o.
$$
\label{Kw2}
\end{proposition}

Note that $g_is_ig_i^{-1}$ is in the center of $g_iL_ig_i^{-1}$ and
so commutes with the elements of $Z_M$, i.e., $g_is_ig_i^{-1}\in
C_G(Z_M)=M$. Therefore it makes sense to evaluate $\theta$ at
$g_is_ig_i^{-1}$ in the above formula.

\subsection{Calculation of sums of character values}

For a partition $\nu$, put $T_{\nu}:=T_{t_{\nu}}$, where
$t_{\nu}\in\mathcal{S}_{|\nu|}$ is an element in the conjugacy class
of type $\nu$ .  If $u_{\tau}$ is a unipotent element of $G^F$ whose
Jordan form is given by the partition $\tau$, then the Green
polynomial~\eqref{greenKostka} $Q_{\nu}^{\tau}(q)$
is the value $Q_{T_{\nu}}^G(u_{\tau})$ of the Green function
$Q_{T_{\nu}}^G$ of Deligne-Lusztig defined in \S\ref{deligne-lusztig}.

For $\muhat=(\mu_1,\dots,\mu_k)\in\left(\bold{T}_n\right)^k$ and
$\omega\in\bold{T}_n$, define

$$\bold{H}^{\muhat}_{\omega}(q):=
\frac{(q-1)K_{\omega}^o}{|W(\omega)|}
\prod_{i=1}^k(-1)^{n+f(\mu_i)}\sum_{\tau}\frac{z_{[\tau]}\hspace{.05cm}\chi_{\tau}^{\mu_i}}{z_{\tau}}
\sum_{\{\nu|\hspace{.05cm}
[\nu]=[\tau]\}}\frac{Q_{\nu}^{\omega}(q)}{z_{\nu}},$$
$$\hat{\bold{H}}^{\muhat}_{\omega}(q):=\frac{(q-1)K_{\omega}^o}{|W(\omega)|}
\prod_{i=1}^k(-1)^{n+f(\omega)}\sum_{\tau}\frac{z_{[\tau]}\hspace{.05cm}\chi^{\omega}_{\tau}}{z_{\tau}}
\sum_{\{\nu|\hspace{.05cm}
[\nu]=[\tau]\}}\frac{Q^{\mu_i}_{\nu}(q)}{z_{\nu}}$$where
$K_{\omega}^o:=K_{\pi(\omega)}^o$, for any type
$\tau=(d_1,\tau^1)\cdots(d_r,\tau^r)$, $f(\tau):=\sum_j|\tau^j|$, and
if we write $\omega=\{m_{d,\lambda}\}_{(d,\lambda)}$, then
$W(\omega)=\prod_{(d,\lambda)\in\Nstar\times\mathcal{P}^*}\left(\Z/d\Z\right)^{m_{d,\lambda}}\times\mathcal{S}_{m_{d,\lambda}}$.

Let $(C_1,\dots,C_k)$ be a generic tuple of conjugacy classes of $G^F$ of
type $\muhat$ and let $(\mathcal{X}_1,\dots,\mathcal{X}_k)$ be
a generic tuple of irreducible characters of $G^F$ of type $\muhat$.

\begin{theorem} We have

\noindent (1)
$\sum_{\mathcal{X}}\prod_{i=1}^k\mathcal{X}(C_i)=\hat{\bold{H}}^{\muhat}_{\omega}(q)$
where the sum is over the irreducible characters of $G^F$ of type
$\omega$.

\noindent (2)
$\sum_{\mathcal{O}}\prod_{i=1}^k\mathcal{X}_i(\mathcal{O})=\bold{H}^{\muhat}_{\omega}(q)$
where the sum is over the conjugacy classes of $G^F$ of type
$\omega$.\label{sums}
\end{theorem}

\begin{remark} Let us denote by $\calX_\omega(1)$ the degree of an irreducible character of $G^F$ of type $\omega$ (the degree depends only on the type). Hence Formula (\ref{form-int}) in Proposition \ref{group-count} applied to $\GL_n$ reads

\begin{equation}\sum_{\calX\in{\rm Irr}(G^F)}\frac{\calX(1)^2}{|G^F|}\left(\frac{|G^F|}{\calX(1)}\right)^{2g}\prod_{i=1}^k\frac{|C_i|\calX(C_i)}{\calX(1)}=\sum_{\omega\in{\bf T}_n}\frac{\calX_\omega(1)^2}{|G^F|}\left(\frac{|G^F|}{\calX_\omega(1)}\right)^{2g}\left(\prod_{i=1}^k\frac{|C_i|}{\calX_\omega(1)}\right)\sum_\calX\prod_{i=1}^k\calX(C_i)\label{group-count-ex}\end{equation}where the second sum in the right hand side is over the irreducible characters of type $\omega$. Theorem \ref{sums}(1) will be used in \S \ref{epolychar} to obtain an expression of this formula in terms of symmetric functions. 

Theorem \ref{sums}(2) will be used to prove Theorem \ref{multi}.
\label{rem-group-count}\end{remark}

\begin{proof}[Proof of Theorem \ref{sums}]Let $\mathcal{X}$ be an irreducible character of type
  $\alpha\in\bold{T}_n$ and let $\mathcal{O}$ be a conjugacy class of
  $G^F$ of type $\beta\in\bold{T}_n$. We have (see Formula
  (\ref{charform1}))
$$\mathcal{X}=\epsilon_G\epsilon_M|W_M|^{-1}\sum_{w\in
W_M}\tilde{\varphi}(wF)R_{T_w}^G(\theta^{T_w}).$$The
$\mathbb{F}_q$-rank of $M$ is $f(\alpha)$, so
$\epsilon_G\epsilon_M=(-1)^{n+f(\alpha)}$. Let
$\sigma\in\mathcal{O}$, and put $L=C_G(\sigma_s)$. Then for $w\in
W_M$,
$$R_{T_w}^G(\theta^{T_w})(\sigma)=|L^F|^{-1}\sum_{{\{h\in
G^F|\hspace{.05cm}\sigma\in
h^{-1}T_wh}\}}Q_{h^{-1}T_wh}^{L}(\sigma_u)\theta^{T_w}(h\sigma_s
h^{-1}).$$We have ${\{h\in G^F|\hspace{.05cm}\sigma\in
h^{-1}T_wh}\}={\{h\in G^F|\hspace{.05cm}h^{-1}T_wh\subset L}\}$. Put
$A_w:={\{h\in G|\hspace{.05cm}h^{-1}T_wh\subset L}\}$. Note that the
sum over $A_w^F$ depends only on the $F$-conjugacy class of $w$ in
$W_M$. The $F$-conjugacy classes of $W_M$, and so the
$M^F$-conjugacy classes of the $F$-stable maximal tori of $M$, are
parametrized by the set of types
$\{\tau|\hspace{.05cm}\tau\sim\alpha\}$ as in \S\ref{types}. From
its definition, the value $\tilde{\varphi}(wF)$ depends also only on
the $F$-conjugacy class of $w$ in $W_M$. For $\tau\in\bold{T}_n$, we
write $T_{\tau}$, $A_{\tau}$, $\tilde{\varphi}(\tau)$ instead of
$T_w$, $A_w$, $\tilde{\varphi}(wF)$ if the $F$-conjugacy class of
$w$ is of type $\tau$. Let $c(\tau)$ be the cardinality of the
corresponding $F$-conjugacy class in $W_M$. Then
$$\mathcal{X}(\sigma)=(-1)^{n+f(\alpha)}|L^F|^{-1}\sum_{\tau\sim\alpha}\sum_{h\in
A_{\tau}^F}\frac{c(\tau)}{|W_M|}\tilde{\varphi}(\tau)Q_{h^{-1}T_{\tau}h}^{L}(\sigma_u)
\theta^{T_{\tau}}(h\sigma_s h^{-1}).$$We have
$c(\tau)/|W_M|=z_{\tau}^{-1}$ and
$\tilde{\varphi}(\tau)=\chi_{\tau}^{\alpha}$. Hence
$$\mathcal{X}(\sigma)=(-1)^{n+f(\alpha)}|L^F|^{-1}\sum_{\tau}z_{\tau}^{-1}\chi_{\tau}^{\alpha}\sum_{h\in
A_{\tau}^F}Q_{h^{-1}T_{\tau}h}^{L}(\sigma_u)
\theta^{T_{\tau}}(h\sigma_s h^{-1}).$$Since by convention
$\chi_{\tau}^{\alpha}=0$ if $\tau\nsim \alpha$, we omit the
condition $\tau\sim\alpha$ in the above sum. The map $h\mapsto
h^{-1}T_{\tau}h$ is a surjective map from the set $A_{\tau}^F$ onto
the set of $F$-stable maximal tori of $L$ that are in the
$G^F$-conjugacy class (of $F$-stable maximal tori of $G$) of type
$[\tau]\in\mathcal{P}_n$. Therefore it induces a  surjective map
$A_{\tau}^F/L^F\rightarrow
\{\nu|\hspace{.05cm}\nu\sim\beta,[\nu]=[\tau]\}$. Hence

\begin{equation}\mathcal{X}(\sigma)=(-1)^{n+f(\alpha)}\sum_{\tau}z_{\tau}^{-1}\chi_{\tau}^{\alpha}
\sum_{\{\nu|\hspace{.05cm} [\nu]=[\tau]\}}Q_{\nu}^{\beta}(q)
\sum_{l\in\overline{A}_{\nu}}\theta^{T_{\tau}}(l\sigma_s
l^{-1})\label{charvalue}\end{equation}where $\overline{A}_{\nu}$ is
the set of elements $lL^F$ of $A_{\tau}^F/L^F$ such that the
$L^F$-conjugacy class of $l^{-1}T_{\tau}l$ is of type $\nu$.

Let us determine the set $\overline{A}_{\nu}$. The $L^F$-conjugacy
classes of the $F$-stable maximal tori of $L^F$ are parametrized by
the set $\{\nu|\hspace{.05cm}\nu\sim\beta\}$. Let $T_{\nu}$ denote
an $F$-stable maximal torus of $L$ whose $L^F$-conjugacy class is of
type $\nu\in \{\gamma|\hspace{.05cm}\gamma\sim\beta,
[\gamma]=[\tau]\}$. Then the $G^F$-conjugacy class of $T_{\nu}$ is
of type $[\nu]=[\tau]$ and so $T_{\nu}$ is $G^F$-conjugate to
$T_{\tau}$, say $T_{\tau}=gT_{\nu}g^{-1}$ with $g\in G^F$. We put
$B_{\nu}={\{h\in G|\hspace{.05cm}
 h^{-1}T_{\nu}h\subset L}\}$. Then the map $h\mapsto g^{-1}h$ induces
a bijection $\left(A_{\tau}^F/L^F\right)\simeq\left(B_{\nu}^F/L^F\right)$.
Since the maximal tori of $L$ are all $L$-conjugate, the map
$N_G(T_{\nu})\rightarrow\left(B_{\nu}/L\right)$, $n\mapsto nL$ is
surjective and commutes with the Frobenius $F$. This map induces a
bijection
$\left(N_G(T_{\nu})/N_L(T_{\nu})\right)\overset{\sim}{\rightarrow}\left(B_{\nu}/L\right)$
which commutes with $F$. We thus have a bijection
$$\left(W_G(T_{\nu})/W_L(T_{\nu})\right)^F\overset{\sim}{\rightarrow}\left(B_{\nu}/L\right)^F.$$Since
$L$ is connected we get bijections
$$\left(W_G(T_{\nu})/W_L(T_{\nu})\right)^F\overset{\sim}{\rightarrow}\left(B_{\nu}^F/L^F\right)\simeq\left(A_{\tau}^F/L^F\right).$$
Under this bijection, the elements of $\overline{A}_{\nu}$
correspond to the elements
$u\in\left(W_G(T_{\nu})/W_L(T_{\nu})\right)^F$ such that
$(T_{\nu})_{\dot{u}^{-1}F(\dot{u})}$, $\dot{u}\in W_G(T_{\nu})$
being a representative of $u$, and $T_{\nu}$ are $L^F$-conjugate.
Now saying that $(T_{\nu})_{\dot{u}^{-1}F(\dot{u})}$ and $T_{\nu}$
are $L^F$-conjugate is equivalent to saying that
$\dot{u}^{-1}F(\dot{u})$ is in the $F$-conjugacy class of $1$ in
$W_L(T_{\nu})$, i.e., $\dot{u}^{-1}F(\dot{u})=w^{-1}F(w)$ for some
$w\in W_L(T_{\nu})$. We know that $W_G(T_{\nu})/W_L(T_{\nu})\simeq
\mathcal{S}_n/\prod_i\left(\mathcal{S}_{|\beta^i|}\right)^{d_i}$. Under
this bijection, the automorphism $F$ on $W_G(T_{\nu})$ induces an
automorphism on $\mathcal{S}_n$ which stabilizes
$\prod_i(\mathcal{S}_{|\beta^i|})^{d_i}$. Let us determine the
automorphism obtained. Let $v_{\beta}$ be an element of
$\mathcal{S}_n$ such that the automorphism $z\mapsto
v_{\beta}zv_{\beta}^{-1}$ induces an action on each component of
$\prod_i\left(\mathcal{S}_{|\beta^i|}\right)^{d_i}$ by circular
permutation of the $d_i$ blocks of length $|\beta^i|$. Then
$(W_L,F)\simeq (\prod_i(\mathcal{S}_{|\beta^i|})^{d_i},v_{\beta})$.
Now let $w_{\nu}\in \prod_i(\mathcal{S}_{|\beta^i|})^{d_i}$ be in
the $v_{\beta}$-conjugacy class of
$\prod_i(\mathcal{S}_{|\beta^i|})^{d_i}$ corresponding to $\nu$,
then
$\left(W_G(T_{\nu}),F\right)\simeq\left(\mathcal{S}_n,w_{\nu}v_{\beta}\right)$,
where $w_{\nu}v_{\beta}:\mathcal{S}_n\rightarrow\mathcal{S}_n,
z\mapsto w_{\nu}v_{\beta}z(w_{\nu}v_{\beta})^{-1}$. We deduce that
$\overline{A}_{\nu}$ is in bijection with the set
$\overline{\mathcal{W}}_{\nu}$ of elements
$x\left(\prod_i(\mathcal{S}_{|\beta^i|})^{d_i}\right)$ with
$x\in\mathcal{S}_n$ such that
$x^{-1}(w_{\nu}v_{\beta})x=t(w_{\nu}v_{\beta})t^{-1}$ for some
$t\in\prod_i(\mathcal{S}_{|\beta^i|})^{d_i}$. 

\noindent Let us determine  the cardinality of $\overline{A}_{\nu}$
as we will need it later. Put
$H=\prod_i(\mathcal{S}_{|\beta^i|})^{d_i}$. We have a bijective
map
$C_{\mathcal{S}_n}(w_{\nu}v_{\beta})/C_H(w_{\nu}v_{\beta})\rightarrow\overline{\mathcal{W}}_{\nu}$,
$xC_H(w_{\nu}v_{\beta})\mapsto xH$. But
$|C_{\mathcal{S}_n}(w_{\nu}v_{\beta})|=z_{[\tau]}$ and
$|C_H(w_{\nu}v_{\beta})|=z_{\nu}$, therefore
\begin{equation}|\overline{A}_{\nu}|=|\overline{\mathcal{W}}_{\nu}|=
z_{[\tau]}z_{\nu}^{-1}.\label{form3}\end{equation}

Now let us compute
$\sum_{\mathcal{X}}\prod_i\mathcal{X}(\mathcal{C}_i)$ and
$\sum_{\mathcal{O}}\prod_i\mathcal{X}_i(\mathcal{O})$. We first
compute the second sum. Let $(L,C)$ be a pair of type $\omega$ where
$L$ is an $F$-stable Levi subgroup and $C$ and $F$-stable unipotent
conjugacy class of $L$. Let $u\in C$. We have a surjective map
$(Z_L)_{\rm reg}^F\rightarrow {\{G^F-\text{orbits of type
}\omega}\}$ that sends $z$ to $\mathcal{O}_{zu}^{G^F}$. If $s,s'\in
(Z_L)_{\rm reg}^F$, then $s$ and $s'$ have the same image if there
exists $g\in G^F$ such that $g(sC)g^{-1}=s'C$,i.e., $gsg^{-1}=s'$
and $gCg^{-1}=C$. The identity $gsg^{-1}=s'$ implies that $g\in
N_G(L)$. Therefore the fibers of our map can be identified with
$W_G(L,C):={\{g\in G^F|\hspace{.05cm}g\in N_G(L)\cap N_G(C)}\}/L^F$
which is of cardinality $|W(\omega)|$. We thus have
$$\sum_{\mathcal{O}}\prod_{i=1}^k\mathcal{X}_i(\mathcal{O})=\frac{1}{|W(\omega)|}\sum_{z\in
(Z_L)_{\rm reg}^F}\prod_{i=1}^k\mathcal{X}_i(zu).$$Applying the
Formula (\ref{charvalue}) with $(\alpha,\beta)=(\mu_i,\omega)$, we
get 

$$\sum_{\mathcal{O}}\prod_{i=1}^k\mathcal{X}_i(\mathcal{O})=\frac{1}{|W(\omega)|}
\sum_{\tau_1,\dots,\tau_k}\sum_{\{(\nu_1,\dots,\nu_k)|\hspace{.05cm}[\nu_i]=[\tau_i]\}}
\prod_{i=1}^k\chi_{\tau_i}^{\mu_i}z_{\tau_i}^{-1}
Q_{\nu_i}^{\omega}(q)$$
$$
\sum_{(l_1,\dots,l_k)\in\overline{A}_{\nu_1}\times\cdots\times\overline{A}_{\nu_k}}\left(\sum_{z\in
(Z_L)_{\rm
reg}^F}\prod_{i=1}^k\theta_i^{T_{\tau_i}}(l_izl_i^{-1})\right).$$Put
$\theta_i^{l_i^{-1}T_{\tau_i}l_i}(z):=\theta_i^{T_{\tau_i}}(l_izl_i^{-1})$
for all $z\in Z_L^F$. Then $\prod_i\theta_i^{l_i^{-1}T_{\tau_i}l_i}$
is a linear character of $Z_L^F$. By assumption, it is generic and
so by Proposition \ref{Kw}, we have $\sum_{z\in (Z_L)_{\rm
reg}^F}\prod_i\theta^{l_i^{-1}T_{\tau_i}l_i}(z)=(q-1)K_{\omega}^o$,
from which we deduce that:

$\sum_{\mathcal{O}}\prod_{i=1}^k\mathcal{X}_i(\mathcal{O})=$
$$\frac{(q-1)K_{\omega}^o}{|W(\omega)|}\sum_{\tau_1,\dots,\tau_k}\sum_{\{(\nu_1,\dots,\nu_k)|\hspace{.05cm}[\nu_i]=[\tau_i]\}}\prod_{i=1}^k\chi_{\tau_i}^{\mu_i}z_{\tau_i}^{-1}
Q_{\nu_i}^{\omega}(q)|\overline{\mathcal{W}}_{\nu_1}|\cdots|\overline{\mathcal{W}}_{\nu_k}|.$$The
assertion (2) of the theorem follows then from the Formula
(\ref{form3}).

Let us now compute $\sum_{\mathcal{X}}\prod_i\mathcal{X}(C_i)$. Let
$(L,\chi)$ be of type $\omega$ with $L$ an $F$-stable Levi subgroup
of $G$ and $\chi$ an $F$-stable irreducible character of $W_L$. Let
$\mathcal{X}_{\chi}^L$ be the unipotent character of $L^F$
associated to $\chi$. The map $\text{Irr}_{\rm reg}(L^F)\rightarrow
{\{\mathcal{X}\in\text{Irr}(G^F)|\hspace{.05cm}\mathcal{X}\text{ of
type }\omega}\}$ that sends $\theta$ to $\epsilon_G\epsilon_L
R_L^G(\theta\cdot\mathcal{X}_{\chi}^L)$ is surjective and its fibers
are of cardinality $|W(\omega)|$. We thus have
$$\sum_{\mathcal{X}}\prod_{i=1}^k\mathcal{X}(C_i)=\frac{1}{|W(\omega)|}\sum_{\theta\in
\text{Irr}_{{\rm
reg}}(L^F)}\prod_{i=1}^k\epsilon_G\epsilon_LR_L^G(\theta\cdot\mathcal{X}_{\chi}^L)(C_i).$$The
value
$\epsilon_G\epsilon_LR_L^G(\theta\cdot\mathcal{X}_{\chi}^L)(C_i)$ is
of the form $\mathcal{X}(\sigma)$, see Formula (\ref{charvalue}),
with $(\alpha,\beta)=(\omega,\mu_i)$. Hence
$$\sum_{\mathcal{X}}\prod_{i=1}^k\mathcal{X}(C_i)=\frac{1}{|W(\omega)|}
\sum_{\tau_1,\dots,\tau_k}\sum_{\{(\nu_1,\dots,\nu_k)|\hspace{.05cm}[\nu_i]=[\tau_i]\}}\prod_{i=1}^k\chi_{\tau}^{\omega}z_{\tau_i}^{-1}
Q_{\nu_i}^{\mu_i}(q)$$
$$
\sum_{(l_1,\dots,l_k)\in\overline{A}_{\nu_1}\times\cdots\times\overline{A}_{\nu_k}}
\left(\sum_{\theta\in\text{Irr}_{\rm
reg}(L^F)}\prod_{i=1}^k\theta^{T_{\tau_i}}(l_i\sigma_{i,s}l_i^{-1})\right)
$$
where
$\sigma_{i,s}$ is the semisimple part of some fixed element
$\sigma_i\in C_i$. Recall that for $\theta\in\text{Irr}_{\rm
reg}(L^F)$, $\theta^{T_{\tau_i}}$ is the restriction of $\theta$ to
$T_{\tau_i}^F$. The assertion (1) of the theorem follows from
Proposition \ref{Kw2} and Formula (\ref{form3}).\end{proof}

\section{Character varieties}
\label{gen-char-var}
Fix a non-negative integer $g$ and choose a generic tuple
$(\calC_1,\calC_2,\dots,\calC_k)$ of semisimple conjugacy classes of
$\GL_n(\C)$ of type $\muhat=(\mu^1,\dots,\mu^k)$ where
$\mu^i=(\mu^i_1,\dots,\mu^i_{r_i})$ is a partition of $n$. Recall that
the non-negative integers $\mu^i_1,\dots,\mu^i_{r_i}$ are the
multiplicities of the distinct eigenvalues of $\calC_i$.  Let
$\mathcal{M}_{\muhat}$ be the corresponding complex character variety
as defined in \S\ref{subchar}.
\subsection{Independence of the generic eigenvalues}
\label{indep-eigenv}
Though the variety $\M_\muhat$ depends on the choice of generic
eigenvalues our main Conjecture~\ref{main} predicts that the mixed
Hodge polynomial $H_c(\M_\muhat;x,y,t)$ should not. In general, in GIT
problems depending on parameters, it is normal to see change in cohomology as one crosses a wall of a certain chamber structure in the
space of parameters. In hyperk\"ahler situations, however, it has been
observed that no change takes place; see for example~\cite{boden}
and~\cite{garciaetal}.

Generalizing the argument of Corollary~2.2.4 of \cite{hausel-villegas}
here we prove that for a dense subset (in the analytic topology) of
generic eigenvalues of multiplicities $\muhat$ the mixed Hodge
polynomial of $\M_\muhat$ is constant.  In particular, at least on
this dense subset, there is no change of behaviour across walls.
We prove in Corollary~\ref{e-poly-indep} below that the $E$-polynomial
of $\M_\muhat$ is completely independent of the choice of generic
eigenvalues of multiplicities $\muhat$.

\begin{proposition}
  There is a dense subset (in the analytic topology) of generic
  eigenvalues of multiplicities $\muhat$ for which the mixed Hodge
  polynomial $H_c(\M_\muhat;x,y,t)$ is constant.
\end{proposition}
\begin{proof}
  Let $r=r_1+\cdots+r_k$ be the number of distinct eigenvalues of the
  conjugacy classes $\calC_i$. With the notation of the proof of
  Lemma~\ref{exists} pick $a'\in A'\cong \Gm^{r-1}$ corresponding to
  $r-1$ algebraically independent transcendental complex numbers. By a
  general fact on automorphisms of $\C/\Q$ any two such choices can be
  conjugated by an element of $\Aut(\C/\Q)$. By functiorality the two
  corresponding varieties have isomorphic mixed Hodge structures.
  This proves our claim.
\end{proof}

\subsection{$E$-polynomial}
\label{epolychar}
In this section we prove that
$\mathcal{M}_{\muhat}$ is polynomial-count and we give a closed
formula for $E(\mathcal{M}_{\muhat};q)$. This formula will be used to
compute the Euler characteristic ine \S\ref{Euler} and later to prove
the connectedness of $\M_{\muhat}$ (see
\cite{hausel-letellier-villegas2} \cite{hausel-letellier-villegas3}).

\begin{theorem}
\label{E-pol-char-var}
  The variety $\mathcal{M}_{\muhat}$  is
polynomial-count and its $E$-polynomial is given by
$$
E(\M_{\muhat};q)= q^{\frac{1}{2}d_{\muhat}}
\mathbb{H}_{\muhat}\left(\sqrt{q},\frac{1}{\sqrt{q}}\right)
$$
where $\mathbb{H}_{\muhat}(z,w)$ is defined in~\eqref{H} and
$d_{\muhat}=\dim(\M_\muhat)$ (see \eqref{dimension}).
\label{maintheo1}\end{theorem} 
\begin{proof}
  It is clear that $\mathbb{H}_{\muhat}(z,w)\in \Q(z,w)$. Hence
  Theorem~\ref{poly} below implies that there exists
  $Q(x)\in\mathbb{Q}(x)$ such that for all $r$ we have
  $\sharp\{\M_{\muhat}^{\phi}(\F_{q^r})\}=Q(q^r)$. In particular
  $Q(x)$ is an integer for infinitely many integer values of $x$,
  hence $Q(x)\in\mathbb{Q}[x]$. Therefore $\mathcal{M}_{\muhat}$ is
  polynomial-count and so our claim follows from Theorem \ref{katz}
  and Theorem \ref{poly} below.
\end{proof}

The theorem has the following straightforward consequence:

\begin{corollary}
\label{e-poly-indep}
 The $E$-polynomial of
$\M_{\muhat}$ does not depend on the choice of the generic
semisimple conjugacy classes $\calC_1,\dots,\calC_k$ of a given type
$\muhat$.
\end{corollary}
Let $\calU_{\muhat}=\Spec(\calA)$ be the $R$-scheme defined in
Appendix A. Put
$\mathcal{X}_{\muhat}=\Spec\left(\calA^{\PGL_n(R)}\right)$. Then the
$R$-scheme $\mathcal{X}_{\muhat}$ is a \emph{spreading out} of
$\M_{\muhat}$, i.e., $\mathcal{X}_{\muhat}$ gives back $\M_{\muhat}$
after extension of scalars from $R$ to $\C$. If $\phi:R\rightarrow k$
is a ring homomorphism into a field $k$, we denote by
$\M_{\muhat}^{\phi}$ the $k$-scheme obtained from
$\mathcal{X}_{\muhat}$ by extension of scalars.

\begin{theorem} For any ring homomorphism $\phi:R\rightarrow \F_q$,
$$\sharp\{\M_{\muhat}^{\phi}(\F_{q})\}=
q^{\frac{1}{2}d_{\muhat}}\mathbb{H}_{\muhat}\left(\sqrt{q},\frac{1}{\sqrt{q}}\right).$$
\label{poly}\end{theorem} 
\begin{proof} Let $k$ be an algebraic closure of $\F_q$. Since $\PGL_n(k)$ is connected any
$F$-stable $\PGL_n(k)$-orbit of $\mathcal{U}_{\muhat}^{\phi}(k)$
contains an $F$-stable point, i.e., an $\F_q$-rational point. Hence
the natural map
$$\mathcal{U}_{\muhat}^{\phi}(\F_q)/\PGL_n(\F_q)\rightarrow
\left(\mathcal{U}_{\muhat}^{\phi}(k)/\PGL_n(k)\right)^F=\mathcal{M}_{\muhat}^{\phi}(\F_q)$$
is surjective. The $k$-tuple of conjugacy classes
$(\calC_1^{\phi},\dots,\calC_k^{\phi})$ being generic, the group
$\PGL_n(\F_q)$ acts freely on $\mathcal{U}_{\muhat}^{\phi}(\F_q)$,
and so the above map is injective. Hence
$$\sharp\{\M_{\muhat}^{\phi}(\F_{q})\}=\frac{\sharp\{\calU_{\muhat}^{\phi}(\F_q)\}}{|\PGL_n(\F_q)|}.$$
Let $\text{Irr}(G^F)_{\omega}$ denote the set of irreducible
characters of type $\omega$. We denote by $\mathcal{X}_{\omega}(1)$
the degree of the irreducible characters in
$\text{Irr}(G^F)_{\omega}$. For $i\in\{1,\dots,k\}$, let $C_i$ be the
conjugacy class $\calC_i^{\phi}(\F_q)$ of $G^F=\GL_n(\F_q)$. From
Proposition \ref{group-count} and Theorem \ref{sums}(1) (see Remark \ref{rem-group-count}), we have

\begin{align*}\sharp\{\calU_{\muhat}^{\phi}(\F_q)\}
&=|G^F|^{2g-1}\sum_{\mathcal{X}\in\text{Irr}(G^F)}\frac{1}{\mathcal{X}(1)^{2g-2+k}}\prod_{i=1}^k
|C_i|\mathcal{X}(C_i)\\
&=\sum_{\omega\in\bold{T}_n}\frac{|G^F|^{2g-1}
\prod_{i=1}^k|C_i|}{\mathcal{X}_{\omega}(1)^{2g-2+k}}\sum_{\mathcal{X}\in\text{Irr}(G^F)_{\omega}}\prod_{i=1}^k\mathcal{X}(C_i)\\
&=\sum_{\omega\in\bold{T}_n}\frac{|G^F|^{2g-1}
\prod_{i=1}^k|C_i|}{\mathcal{X}_{\omega}(1)^{2g-2+k}}\hat{\bold{H}}^{\muhat_*}_{\omega}(q)\\
&=\sum_{\omega\in\bold{T}_n}\frac{|G^F|^{2g-1}
(q-1)K_{\omega}^o\prod_{i=1}^k|C_i|}{|W(\omega)|\mathcal{X}_{\omega}(1)^{2g-2+k}}(-1)^{kn+kf(\omega)}
\prod_{i=1}^kA(\omega,\mu^i)\end{align*} with $A(\omega,\mu^i_*)$ as
in Lemma \ref{intermediate1} where $\mu^i_*$ is the type in
$\bold{T}_n$ corresponding to the partition $\mu^i$, see beginning
of this section. For a type
$\omega=(d_1,\omega^1)\cdots(d_r,\omega^r)$, recall (see \cite[IV,
6.7]{macdonald})
$$\frac{|G^F|}{\mathcal{X}_{\omega}(1)}=(-1)^{f(\omega)}H_{\omega}(q)q^{\frac{1}{2}n(n-1)-n(\omega)}.$$
By Formula (\ref{central}) we have
$\calH_{\mu^i_*}^0(0,\sqrt{q})=|C_i|/|G^F|$ and note that
$C_{\omega}^0=K_{\omega}^o/|W(\omega)|$, see Formula (\ref{ctau0})
and Proposition \ref{Kw}. Using also Lemma \ref{intermediate1}, we
thus deduce that:

$\sharp\{\calU_{\muhat}^{\phi}(\F_q)\}$
\begin{align*}&=|G^F|(q-1)\sum_{\omega\in\bold{T}_n}\left((-1)^{f(\omega)}H_{\omega}(q)q^{\frac{1}{2}n(n-1)-n(\omega)}\right)^{2g+k-2}
C_{\omega}^0(-1)^{kn+kf(\omega)}\prod_{i=1}^k \left\langle
s_{\omega}(\x_i),\calH_{\mu^i_*}^0(0,\sqrt{q})\tilde{H}_{\mu^i_*}(\x_i;q)\right\rangle\\
&=|G^F|(q-1)(-1)^{kn}q^{\frac{1}{2}n(n-1)(2g+k-2)}\sum_{\omega\in\bold{T}_n}C_{\omega}^0
\left(H_{\omega}(q)q^{-n(\omega)}\right)^{2g+k-2} \prod_{i=1}^k
\left\langle
s_{\omega}(\x_i),\calH_{\mu^i_*}^0(0,\sqrt{q})\tilde{H}_{\mu^i_*}(\x_i;q)\right\rangle\\
&=|G^F|(q-1)(-1)^{kn}q^{\frac{1}{2}(n^2(k+2g-2)-kn)}\\
&\hspace{2cm}\left\langle\sum_{\omega\in\bold{T}}C_{\omega}^0q^{(1-g)|\omega|}\left(H_{\omega}(q)q^{-n(\omega)}\right)^{2g+k-2}
\prod_{i=1}^k
s_{\omega}(\x_i),\prod_{i=1}^k\calH_{\mu^i_*}^0(0,\sqrt{q})\tilde{H}_{\mu^i_*}(\x_i;q)\right\rangle\\
&=|G^F|(q-1)(-1)^{kn}q^{\frac{1}{2}(n^2(k+2g-2)-kn)}\\
&\hspace{2cm}\left\langle\Log\left(\sum_{\lambda\in\mathcal{P}}q^{(1-g)|\lambda|}\left(H_{\lambda}(q)q^{-n(\lambda)}\right)^{2g+k-2}
\prod_{i=1}^k
s_{\lambda}(\x_i)\right),\prod_{i=1}^k\calH_{\mu^i_*}^0(0,\sqrt{q})\tilde{H}_{\mu^i_*}(\x_i;q)\right\rangle
\\&=|G^F|(q-1)q^{\frac{1}{2}(n^2(k+2g-2)-kn)-\sum_in(\mu^i_*)}\left\langle\Log\left(\sum_{\lambda\in\mathcal{P}}q^{(1-g)|\lambda|}\left(H_{\lambda}(q)q^{-n(\lambda)}\right)^{2g+k-2}
\prod_{i=1}^k
s_{\lambda}(\x_i)\right),\prod_{i=1}^kh_{\mu^i}(\x_i\y)\right\rangle.\end{align*}
In the third equality $|\omega|$ is defined as the size of $\omega$,
i.e., $|\omega|=n$ if $\omega\in\bold{T}_n$. The last equality
follows from Lemma \ref{intermediate3}. For any symmetric functions
$u$ and $v$, $\langle u(\x\y),v(\x)\rangle=\langle
u(\x),v(\x\y)\rangle$. This can be checked on the
basis of power symmetric functions. We deduce from Lemma
\ref{propchar} that

$$\sharp\{\calU_{\muhat}^{\phi}(\F_q)\}
=|G^F|(q-1)q^{\frac{1}{2}(n^2(k+2g-2)-kn)-\sum_in(\mu^i_*)}\left\langle\Log\left(\Omega(\sqrt{q},1/\sqrt{q})\right),\prod_{i=1}^kh_{\mu^i}(\x_i)\right\rangle.$$
We thus have
$$\sharp\{\M_{\muhat}^{\phi}(\F_{q})\}=(q-1)^2q^{\frac{1}{2}(n^2(k+2g-2)-kn)-\sum_in(\mu^i_*)}\left\langle\Log\left(\Omega(\sqrt{q},1/\sqrt{q})\right),\prod_{i=1}^kh_{\mu^i}(\x_i)\right\rangle.$$
We have
$\mathbb{H}_{\muhat}(\sqrt{q},1/\sqrt{q})=\frac{(q-1)^2}{q}\left\langle\Log\left(\Omega(\sqrt{q},1/\sqrt{q})\right),\prod_{i=1}^kh_{\mu^i}(\x_i)\right\rangle.$
It remains to check that the remaining power of $q$ is
$d_{\muhat}/2$, but this follows from the observation that
$2n(\mu^i_*)+n=\sum_j(\mu^i_j)^2$.\end{proof}

Here we can prove a consequence of the Curious Poincar\'e duality
Conjecture~\ref{cpd}.  

\begin{corollary}
The $E$-polynomial is palindromic, i.e., it satisfies the "curious"
Poincar\'e
duality:\begin{align*}E(\mathcal{M}_{\muhat};q)&=q^{d_{\muhat}}E(\mathcal{M}_{\muhat};q^{-1})\\ 
&=\sum_i\left(\sum_k(-1)^kh^{i,i;k}(\M_{\muhat})\right)q^i\end{align*}\label{curious}\end{corollary}

\begin{proof} By Theorem \ref{smoothcharacter} the variety
 $\mathcal{M}_{\muhat}$ is non-singular of pure dimension
 $d_{\muhat}$. Hence the second equality is a
 consequence of Formula (\ref{poincare}). From Theorem
 \ref{maintheo1} we have
\begin{align*}
E(\M_{\muhat};q^{-1})&=q^{-d_{\muhat}/2}\mathbb{H}_{\muhat}(1/\sqrt{q},\sqrt{q})\\
&=q^{-d_\muhat/2}\frac{(q-1)^2}{q}\left\langle\prod_ih_{\mu^i}(\x_i),
  \Log\left(\Omega(1/\sqrt{q},\sqrt{q})\right)\right\rangle
\end{align*} 
From (\ref{Oduality}) we conclude that
\begin{align*}
E(\M_{\muhat};q^{-1})&= 
q^{-d_{\muhat}/2}\frac{(q-1)^2}{q}
\left\langle\prod_ih_{\mu^i}(\x_i),
\Log \left(\Omega(\sqrt{q},1/\sqrt{q})\right)\right\rangle
=q^{-d_\muhat}E(\M_{\muhat};q)
\end{align*}
\end{proof}

\subsection{Euler characteristic}
\label{Euler}
The $2g$-dimensional torus $(\C^{\times})^{2g}$ acts on the character
variety $\M_\muhat$ by scalar multiplication on the first $2g$
coordinates. Let $\tilde{\M}_\muhat$ be the affine GIT quotient
$\left(\M_\muhat\right)/\!/(\C^{\times})^{2g}$.  Exactly as in
\cite{hausel-villegas}[Theorem 2.2.12] we can argue that
$$
H^*(\M_\muhat)\cong H^*((\C^{\times})^{2g}) \otimes
H^*(\tilde{\M}_\muhat)
$$
as mixed Hodge structures, which implies that
$$
H_c(\M_\muhat;x,y,t)=H_c(\tilde\M_\muhat;x,y,t)\cdot (1+xyt)^{2g}
$$
and hence also 
$$
\label{e-pol-tilde}
E(\M_\muhat;q)=E(\tilde \M_\muhat;q)\cdot (1-q)^{2g}.
$$
It follows that $E(\M_\muhat)=0$ if $g>0$. Here we compute
$E(\tilde{\M}_\muhat)$ for $g>0$.
\begin{remark}
\label{MH-tensor}
Note, in particular, that Conjecture~\ref{main} (iii) implies that
$(z-w)^{2g}$ should divide $\H_\muhat(z,w)$. This is not readily
visible from its definition~\eqref{introduceH}. 
\end{remark}

\begin{theorem} Assume that $g>1$, then
$$
E(\tilde{\M}_\muhat)=
\begin{cases}
\mu(n)\,n^{2g-3} & \text{ if }
\muhat=((n),\ldots ,(n))\\
0& \text{ otherwise}
\end{cases}
$$
where $\mu$ is the ordinary M\"obius function.
\end{theorem}

\begin{proof}
First note that
\begin{equation}
\label{euler-char-1} E(\tilde{\M}_\muhat)=\left. \frac{\langle
h_\muhat,\Log(\Omega(\sqrt{q},1/\sqrt{q}))\rangle}
{(q-1)^{2g-2}}\right|_{q=1},
\end{equation}
where, as before, $h_\muhat:=\prod_{i=1}^kh_{\mu^i}(\x_i)$.  We have
by \S\ref{cauchy-fctns}
\begin{equation}
\label{A-defn}
\Omega(\sqrt{q},1/\sqrt{q})=\sum_{\lambda\in\calP} A_\lambda, \qquad
A_\lambda: = \left(q^{-\frac{1}{2}\langle\lambda,\lambda\rangle}
H_{\lambda}(q)\right)^{2g-2}
\prod_{i=1}^k\tilde{H}_{\lambda}(\x_i;q,q^{-1}).
\end{equation}
Let
$U_n=U_n(\x_1,\ldots,\x_k;q)$ be defined by
$$
\log\left(\Omega(\sqrt{q},1/\sqrt q)\right)=\sum_{n\geq 1} \frac1n
U_n(\x_1,\ldots \x_k;q)
$$
then as in (\ref{U-equ})
\begin{equation}
\label{U_n-sum} \frac{U_n} n =
\sum_{m_\lambda}(-1)^{m-1}(m-1)!\prod_{\lambda}
\frac{A_\lambda^{m_\lambda}} {m_\lambda!}
\end{equation}
where $m:=\sum_\lambda m_\lambda$ and the sum is over all sequences
$\{m_\lambda\}$ of non-negative integers such that $\sum_\lambda
m_\lambda|\lambda|=n$. Since $(q-1)^{|\lambda|}$ divides
$H_{\lambda}(q)$, $(q-1)^{(2g-2)n}$ divides $U_n$ as it divides
each term in the sum \eqref{U_n-sum}. Let $V_n=V_n(\x_1,\ldots
\x_k;q)$ be defined by
$$
\Log(\Omega(\sqrt q,1/\sqrt q))= \sum_{n\geq 1}V_n(\x_1,\ldots
,\x_k;q),
$$
then by \eqref{U-V-general}
$$
\left\langle h_\muhat,\Log(\Omega(\sqrt{q},1/\sqrt{q}))\right\rangle
=\langle
h_\muhat,V_n \rangle\\
=\frac{1}{n} \sum_{d\mid n}\mu(d)\left\langle h_\muhat,
U_{n/d}(\x_1^d,\ldots ,\x_k^d;q^d)\right\rangle.
$$

Since $(q-1)^{(2g-2)(n/d)}$ divides
$U_{n/d}(\x^d_1,\ldots ,\x^d_k;q^d)$ for all $d$, we have
$$
\left. \frac{\langle h_\muhat, V_n
\rangle}{(q-1)^{2g-2}}\right|_{q=1}= \tfrac 1n\mu(n)\left\langle
h_\muhat, \left. \frac {U_1(\x^n_1,\ldots
,\x^n_k;q^n)}{(q-1)^{2g-2}}\right|_{q=1}\right\rangle
$$
But
$$
U_1(\x^n_1,\ldots ,\x^n_k;q^n)=
q^{n(g-1)}(q^n-1)^{2g-2}\prod_{i=1}^k\tilde{H}_{(1)}(\x^n_i;q^n,q^{-n})
$$
and $\tilde H_{(1)}(\x^n)=p_{(1)}(\x^n)=p_{(n)}(\x)$. Hence
\begin{align*}
\left. \frac{\langle h_\muhat,
V_n\rangle}{(q-1)^{2g-2}}\right|_{q=1} &=\frac 1n
\mu(n)n^{2g-3}\prod_{i=1}^k
\langle h_{\mu^i}(\x_i),p_{(n)}(\x_i)\rangle\\
&=
\begin{cases}
\frac 1n \mu(n)n^{2g-3} & \text{
if }\muhat=((n),\ldots ,(n))\\
0 & \text{otherwise}
\end{cases}.
\end{align*}
The last equality follows from Lemma \ref{induc}
\end{proof}

\begin{theorem} 
For $g=1$
$$
E(\tilde{\M}_\muhat)= \frac 1n\sum_{d\mid \gcd(\mu_j^i)}
\sigma(n/d)\mu(d)\frac{\left((n/d)!\right)^k}{\prod_{i,j}(\mu_j^i/d)!}.
$$
where $\sigma(m)=\sum_{d \mid m}d$.
\end{theorem}

\begin{proof} By \cite[VI, (8.16)]{macdonald}, we have
$K_{\lambda\mu}(1,1)=\chi_{(1^n)}^{\lambda}=n!/h(\lambda)$ where
$h(\lambda)$ is the hook length of $\lambda$ and so for a partition
$\mu$ of size $n$, we have \cite[I, p.66]{macdonald}
$$
\tilde{H}_{\mu}(\x;1,1)= \sum_{\lambda}\frac{n!}{h(\lambda)}
s_{\lambda}(\x)=e_1(\x)^n=h_1(\x)^n.
$$
Hence
\begin{equation}
\label{euler-char-2} \Omega(1,1)=\sum_{\lambda}h_{(1,1,\dots,1)}^{|\lambda|}
=\prod_{m\geq 1} \left(1-h_1^m\right)^{-1}
\end{equation}
by Euler's  formula. As before, let $U_n=U_n(\x_1,\ldots ,\x_k)$ and
$V_n=V_n(\x_1,\ldots ,\x_k)$ be the coefficients of
$\log\left(\Omega^1(1,1)\right)$ and
$\Log\left(\Omega^1(1,1)\right)$ respectively. Then
$U_n=\sigma(n)h_1$ and

\begin{align*}
\left\langle h_\muhat,\Log\left(\Omega(1,1)\right)\right\rangle&=
\left\langle h_\muhat,V_n\right\rangle\\
&=\frac{1}{n}\sum_{d \mid n}\sigma(n/d)\mu(d)\prod_{i=1}^k\langle
h_1(\x_i^d)^{n/d},h_{\mu^i}(\x_i)\rangle\\
&=\frac{1}{n}\sum_{d \mid n}\sigma(n/d)\mu(d)\prod_{i=1}^k\langle
p_{(d^{n/d})}(\x_i),h_{\mu^i}(\x_i)\rangle\\&=\frac{1}{n}
\sum_{d|{\rm
gcd}(\mu_j^i)}\sigma(n/d)\mu(d)\frac{\left((n/d)!\right)^k}{\prod_{i,j}(\mu_j^i/d)!}.
\end{align*}
The last equality follows from Lemma \ref{induc}.
\end{proof}

\begin{remark} 
\label{euler-char-g=0}
The task to evaluate the Euler characteristic when $g=0$ is more
complicated, due to the presence of high order poles in
$\calH_{\lambda}^0(\sqrt{q},1/\sqrt{q})$ at
$q=1$. See~\eqref{euler-char-example} for a computation in a specific
example.
\end{remark}

\section{The pure part of $\mathbb{H}_{\rm \mu}(z,w)$}

In this section we fix once and for all a multi-partition
$\muhat=(\mu^1,\dots,\mu^k)\in\left(\mathcal{P}_n\right)^k$ where
$\mu^i=(\mu^i_1,\dots,\mu^i_{l_i})$. We give both a representation
theoretical and a cohomological interpretation of the pure part
$\H_{\muhat}(0,w)$ of $\H_\muhat(z,w)$.

\subsection{Multiplicities in tensor products}
\label{multiplicities}

In this section $G=\GL_n(\overline{\F}_q)$. For a partition
$\mu=(n_1,\dots,n_r)$ we define $\mu_{\dagger}$ to be the type
$(1,(n_1)^1)\cdots(1,(n_r)^1)\in\bold{T}$. Let
$(\mathcal{X}_1,\dots,\mathcal{X}_k)$ be a generic tuple of
$k$-irreducible characters of type
$\muhat_{\dagger}:=(\mu^1_{\dagger},\dots,\mu^k_{\dagger})\in\bold{T}_n$. The
irreducible characters $\mathcal{X}_1,\dots,\mathcal{X}_k$ are then
semisimple. Put $$R_{\muhat}:=\bigotimes_{i=1}^k\mathcal{X}_i.$$

Let $\Lambda:G^F\rightarrow\overline{\mathbb{Q}}_{\ell}$ be defined
by $x\mapsto q^{g\hspace{.05cm}\text{dim}\hspace{.05cm}C_G(x)}$.
Note that the map $x\mapsto q^{\text{dim}\hspace{.05cm}C_G(x)}$ is
the character of the representation of $G^F$ in the group algebra
$\overline{\mathbb{Q}}_{\ell}[\mathfrak{g}^F]$ where $G^F$ acts on
$\mathfrak{g}^F$ by the adjoint action.

Let
$\langle\cdot,\cdot\rangle_{G^F}$ be the non-degenerate bilinear
form on $C(G^F)$ defined by $$\langle
f,g\rangle_{G^F}=|G^F|^{-1}\sum_{x\in
G^F}f(x)\overline{g(x)}.$$

\begin{theorem} We have $$\left\langle \Lambda\otimes
R_{\muhat},1\right\rangle_{G^F}=\mathbb{H}_{\muhat}(0,\sqrt{q}).$$
\label{multi}\end{theorem} 

\begin{proof} Recall Lemma \ref{Homega} which says that if $C$ is a conjugacy class of
$G^F$ of type $\omega\in\bold{T}_n$, then
$\calH_{\omega}(0,\sqrt{q})=q^{g\text{dim}\hspace{.05cm}C_G(x)}|C|/|G^F|$
where $x\in C$. Hence by Theorem \ref{sums}(2)

\begin{align*}\left\langle \Lambda\otimes
\bigotimes_{i=1}^k\mathcal{X}_i,{\rm
Id}\right\rangle_{G^F}&=\sum_C\frac{|C|}{|G^F|}\Lambda(C)\prod_{i=1}^k\mathcal{X}_i(C)\\
&=\sum_{\omega\in\bold{T}_n}\calH_{\omega}(0,\sqrt{q})\bold{H}_{\omega}^{\muhat_{\dagger}}(q)\\
&=
\sum_{\omega\in\bold{T}_n}\frac{(q-1)K_{\omega}^o}{|W(\omega)|}\calH_{\omega}(0,\sqrt{q})
\prod_{i=1}^k(-1)^{n+f(\mu^i_{\dagger})}\left\langle
s_{\mu^i_{\dagger}}(\x_i),\tilde{H}_{\omega}(\x_i;q)\right\rangle\\
&=(-1)^{kn+\sum_if(\mu^i_{\dagger})}\sum_{\omega\in\bold{T}_n}(q-1)C_{\omega}^0\calH_{\omega}(0,\sqrt{q})
\left\langle\prod_is_{\mu^i_{\dagger}}(\x_i),\prod_i\tilde{H}_{\omega}(\x_i;q)\right\rangle\\
&=(-1)^{kn+\sum_if(\mu^i_{\dagger})}(q-1)
\left\langle\prod_is_{\mu^i_{\dagger}}(\x_i),\sum_{\omega\in\bold{T}_n}C_{\omega}^0\calH_{\omega}(0,\sqrt{q})\prod_i\tilde{H}_{\omega}(\x_i;q)\right\rangle\\
&=(q-1)
\left\langle\prod_ih_{\mu^i}(\x_i),\text{Log}\hspace{.05cm}\left(\Omega\left(0,\sqrt{q}\right)\right)\right\rangle\end{align*}The
last equality follows from the fact that $f(\mu^i_{\dagger})=n$ and
$s_{\mu^i_{\dagger}}(\x)=s_{(\mu^i_1)^1}(\x)\cdots s_{(\mu^i_{l_i})^1}(\x)=h_{\mu^i}(\x)$.\end{proof}

\subsection{Poincar\'e polynomial of quiver varieties}
\label{poincare-quiver}

Here we assume that $\muhat$ is indivisible so that we can choose a
generic tuple $(\calO_1,\dots,\calO_k)$ of semisimple adjoint orbits
of $\gl_n(\C)$ of type $\muhat$. Let $\mathcal{Q}_{\muhat}$ be the
corresponding complex quiver variety as in \S\ref{subquiver}.

The aim of this section is to prove the following theorem.

\begin{theorem} The compactly supported Poincar\'e polynomial of
$\mathcal{Q}_{\muhat}$ is given by
$$
P_c(\mathcal{Q}_{\muhat};t)= 
t^{d_{\muhat}}\mathbb{H}_{\muhat}\left(0,t\right).$$ 
\label{maintheo2}\end{theorem} 

As we did for the character variety in Appendix \ref{appendix}, we
define a spreading out $\mathcal{Y}_{\muhat}/\mathcal{R}$ of $\calQ_{\muhat}$
such that  for any ring
homomorphism $\phi:\mathcal{R}\rightarrow \K$ into an algebraically
closed field $\K$, the adjoint orbits
$\calO_1^{\phi},\dots,\calO_k^{\phi}$ of $\gl_n(\K)$ are generic and
of same type as $\calO_1,\dots,\calO_k$. Let
$\mathcal{Q}_{\muhat}^{\phi}$ denote the corresponding quiver
variety over $\K$.

\begin{theorem} For any ring homomorphism $\phi:\mathcal{R}\rightarrow\F_q$ we have
\begin{equation}\label{num}\sharp\{\mathcal{Q}_{\muhat}^{\phi}(\F_q)\}=q^{\frac{1}{2}d_{\muhat}}\mathbb{H}_{\muhat}\left(0,\sqrt{q}\right).\end{equation}
\label{poly2}\end{theorem} 

Theorem \ref{maintheo2} follows from Proposition \ref{epoly2}, Proposition \ref{puremhs}, and 
Theorem \ref{poly2}. Indeed, Theorem \ref{poly2} implies that
$\mathcal{Q}_{\muhat}^g/\C$ is polynomial-count (see remark just
after Theorem \ref{poly}). 

We now prove Theorem \ref{poly2}. 

\noindent For $i\in\{1,\dots,k\}$, let $O_i$ be the adjoint orbit
$\calO_i^{\phi}(\F_q)$ of $\mathfrak{g}^F=\gl_n(\F_q)$. As in the
character variety case we show that
$$\sharp\{\calQ_{\muhat}^{\phi}(\F_q)\}=\frac{\sharp\{\calV_{\muhat}^{\phi}(\F_q)\}}{|\PGL_n(\F_q)|}.$$ Let
$\Lambda:\mathfrak{g}^F\rightarrow \overline{\mathbb{Q}}_{\ell}$,
$x\mapsto q^{g\,\text{dim}\hspace{.05cm} C_G(x)}$. By Proposition
\ref{algebra-case}  and Remark \ref{remarkfour}, we have

\begin{align*}\sharp\{\calQ_{\muhat}^{\phi}(\F_q)\}&=q^{n^2(g-1)}(q-1)
\sum_{O}\frac{|O|}{|G^F|}\Lambda(O)\prod_{i=1}^k\mathcal{F}^{\mathfrak{g}}(1_{O_i})(O)\\
&=q^{n^2(g-1)}(q-1)
\sum_{\omega\in\bold{T}_n}\calH_{\omega}(0,\sqrt{q})\sum_{O}\prod_{i=1}^k\mathcal{F}^{\mathfrak{g}}(1_{O_i})(O)\end{align*}
where the second sum is over the adjoint orbits $O$ of
$\mathfrak{g}^F$ of type $\omega$. The type of adjoint orbits is
defined exactly as for conjugacy classes, see  \S\ref{types}. We
need the following lemma 

\begin{lemma} Given $\omega\in\bold{T}_n$, we have
$$\sum_O\prod_{i=1}^k\mathcal{F}^{\mathfrak{g}}(1_{O_i})(O)=\frac{q^{1+\sum_id_i/2}}{q-1}\bold{H}_{\omega}^{\muhat_{\dagger}}(q)$$
where the sum is over the adjoint orbits of type $\omega$, where
$\muhat_{\dagger}$ is as in \S\ref{multiplicities} and where
$d_i=n^2-\sum_j(\mu^i_j)^2$.
\end{lemma}

\begin{proof} We first remark that if $C$ is a semisimple adjoint
orbit of $\mathfrak{g}^F$ of type
$(1,1^{n_1})(1,1^{n_2})\cdots(1,1^{n_r})$, then by Formula
(\ref{charform2})
$$\mathcal{F}^{\mathfrak{g}}(1_C)=\epsilon_G\epsilon_L|W_L|^{-1}\sum_{w\in
W_L}q^{d_L/2}\mathcal{R}_{\mathfrak{t}_w}^{\mathfrak{g}}\left(\mathcal{F}^{\mathfrak{t}_w}(1_{\sigma}^{T_w})\right)$$where
$L=\prod_{i=1}^r \GL_{n_i}(\overline{\F}_q)$ and where $\sigma\in
C\cap L$.

If $\mathcal{X}$ is an irreducible character of type
$(1,(n_1)^1)(1,(n_2)^1)\cdots(1,(n_r)^1)$, by Formula (\ref{charform1})
we have
$$\mathcal{X}=\epsilon_G\epsilon_L|W_L|^{-1}\sum_{w\in
W_L}R_{T_w}^G\left(\theta^{T_w}\right)$$where $L=\prod_{i=1}^r
\GL_{n_i}(\overline{\F}_q)$. Hence from the formulae
(\ref{charformula2}) and (\ref{charformula1}) we see that the
calculation of the values of $\mathcal{X}$ and
$\mathcal{F}^{\mathfrak{g}}(1_C)$ is completely similar. We thus may
follow the proof of Theorem \ref{sums}(2) to compute
$\sum_{\mathcal{O}}\prod_{i=1}^k\mathcal{F}^{\mathfrak{g}}(1_C)(O)$.
To do that we need to use the Lie algebra analogue of Proposition
\ref{Kw} which is as follows. Let $M$ be an $F$-stable Levi subgroup
of $G$ of type $\omega\in\hat{\bold{T}}_n$ with Lie algebra
$\mathfrak{m}$. We say that a linear character $\Theta:
z(\mathfrak{m})^F\rightarrow\overline{\mathbb{Q}}_{\ell}$ is
\emph{generic} if its restriction to $z(\mathfrak{g})^F$ is trivial
and if for any proper $F$-stable Levi subgroup $H$ containing $M$,
its restriction to $z(\mathfrak{h})^F$ is non-trivial. Put
$$z(\mathfrak{m})_{\rm reg}:=\{x\in z(\mathfrak{m})|\hspace{.05cm}
C_G(x)=M\}.$$Then $$\sum_{z\in z(\mathfrak{m})_{\rm
reg}^F}\Theta(z)=qK_{\omega}^o$$where $K_{\omega}^o$ is as in
Proposition \ref{Kw}. The proof of this identity is completely
similar to that of Proposition \ref{Kw} except that here we are
working with additive characters of $\F_q$ instead of multiplicative
characters of $\F_q^{\times}$. This explains the coefficient $q$
instead of $q-1$.\end{proof} 

We thus have

\begin{align*}\sharp\{\mathcal{Q}_{\muhat}^{\phi}(\F_q)\}
&=q^{n^2(g-1)}(q-1)
\sum_{\omega\in\bold{T}_n}\calH_{\omega}(0,\sqrt{q})
\frac{q^{1+\sum_id_i/2}}{q-1}\bold{H}_{\omega}^{\muhat_{\dagger}}(q)\\
&=q^{d_{\muhat}/2}
\sum_{\omega\in\bold{T}_n}\calH_{\omega}(0,\sqrt{q})
\bold{H}_{\omega}^{\muhat_{\dagger}}(q).
\end{align*}
We may now proceed as in the proof of Theorem \ref{multi} to
complete the proof of Theorem \ref{poly2}. \vspace{.5cm}

\subsection{Quiver representations, Kac-Moody algebras and the
  character ring of $\GL_n(\F_q)$}

Let $\Gamma$ be the comet-shaped quiver associated to $g$ and $\muhat$
as in \S\ref{subquiver} and let $\v$ be the dimension vector with
 dimension $\sum_{j=1}^l\mu_j^i$ at the $l$-th vertex on the $i$-th
 leg.  Then

\begin{theorem}
\label{non-empty}
For $\muhat$ indivisible the two followings are
equivalent: 

(a) $\left\langle \Lambda\otimes R_{\muhat},1\right\rangle\neq
0$.

(b) The quiver variety $\calQ_{\muhat}$ is non-empty.

\noindent For $g=0$ (a) or (b) hold if and only if $\v$ is a root
of  the Kac-Moody algebra associated to $\Gamma$.
\end{theorem}

\begin{proof} The equivalence between (a) and (b) follows from the
theorems \ref{maintheo2} and \ref{multi}. If $g=0$, then it is
proved by Crawley-Boevey \cite[\S6]{crawley-mat} that
$\calQ_{\muhat}$ is non-empty if and only if $\v$ is a root.\end{proof}

As mentioned in the introduction, the problem of the non-emptiness of
$\calQ_{\muhat}$ in the genus $g=0$ case, which is part of the
Deligne-Simpson problem, was first solved by Kostov
\cite{kostov}\cite{kostov2}. The equivalence of (a) and (b) in
Theorem~\ref{non-empty} is formally similar to the connection between
Horn's problem (which asks for which partitions $\lambda,\mu,\nu$
does $H_{\lambda}+H_{\mu}+H_{\nu}=0$ have solutions in Hermitian
matrices) and the problem of the non-trivial appearance of the trivial
representation in the tensor product $V_{\lambda}\otimes
V_{\mu}\otimes V_{\nu}$ of the irreducible representations
$V_{\lambda},V_{\mu},V_{\nu}$ of $\GL_n(\C)$ \cite{knutson}.

We conclude with a naturally arising question: Can the identity
$A_{\muhat}(q)=\left\langle \Lambda\otimes R_\muhat, 1\right\rangle$
in \S\ref{intromulti} be strengthened by establishing an explicit
bijection between the set of isomorphic classes of absolutely
indecomposable representations of $\Gamma$ and a basis of
$\left(V_{\Lambda}\otimes
 V_1\otimes\cdots\otimes V_k\right)^{\GL_n(\F_q)}$ where
$V_{\Lambda}:=\left(\overline{\mathbb{Q}}_{\ell}[\gl_n(\F_q)]\right)^{\otimes
g}$ and $V_i$ is a representation of $\GL_n(\F_q)$ which affords the
character $\mathcal{X}_i$?

\section{Appendices}
\subsection{Appendix A}
\label{appendix}
Fix integers $g\geq 0$, $k,n>0$.  We now construct a scheme whose
points parametrize representations of the fundamental
group of a $k$-punctured Riemann surface of genus $g$ into ${\rm
GL}_n$ with prescribed images in conjugacy classes
$\calC_1,\ldots,\calC_k$
at the punctures. We give the construction
of this scheme in stages to alleviate the notation.

Fix $\muhat=(\mu^1,\mu^2,\ldots,\mu^k)\in {\P_n}^k$ and let $a_j^i$,
for $i=1,\ldots,k;j=1,\ldots,r_i:=l(\mu^i)$, be indeterminates. We
should think of $a_1^i,\ldots,a_{r_i}^i$ as the distinct eigenvalues
of $\calC_i$ each with multiplicity $\mu_j^i$; it will be in fact
convenient to work with the multiset
$\A_i:=\{a_1^i,\ldots,a_1^i,a_2^i,\ldots,
a_2^i,\ldots,a_{r_i}^i,\ldots,a_{r_i}^i\}$. To simplify we write
$[\A]:=\prod_{a\in \A}a$ for any multiset $\A\subseteq \A_i$.

Let
$$
R_0:=\Z[a_j^i]/\left(1-[\A_1]\cdots[\A_k]\right)
$$
and consider the multiplicative set $S\subseteq R_0$ generated by
(the classes of) $a^i_{j_1}-a^i_{j_2}$ for $j_1\neq j_2$ and
$1-[\A'_1]\cdots[\A'_k]$ for $\A_i'\subseteq \A_i$ of the same
cardinality $n'$ with $0<n'<n$.

Since $R_0$ is reduced and $S$ does not contain $0$ the localization
$$
R:=S^{-1}R_0
$$
is not trivial ($R$ is a ring with $1$). We  refer to it as the
{\it ring of generic eigenvalues of type $\muhat$}.

In the special case where $k=1$ and $\mu=(n)$ we have
$$
R_0=\Z[a]/(1-a^n)
$$
and $S\subseteq R_0$ is the multiplicative set generated by
$1-a^{n'}$ for  $1\leq n'<n$.

\begin{lemma}
For $k=1$ and $\mu=(n)$ the ring $R=S^{-1}R_0$ is isomorphic to
$\Z[\frac 1n,\zeta_n]$, where $\zeta_n$ is a primitive $n$-th root of unity.
\end{lemma}
\begin{proof}

The natural map $\psi: R_0 \rightarrow R=S^{-1}R_0$ has kernel the
ideal generated by $(1-a^n)/(1-a^{n'})$ for  $1\leq n'<n$. This means
that $\psi$ factors through $\Z[\zeta_n]\hookrightarrow R$ with
$\psi(a)=\zeta_n$.
Since
$$
\prod_{i=1}^{n-1}(1-\zeta_n^i)=n
$$
and each factor is in the image of $S$ it follows that $\tfrac 1n\in
R$. Hence $\Z[\tfrac 1n, \zeta_n]\hookrightarrow R$.

By the same token the map $\phi: R_0 \rightarrow \Z[\frac 1n,
\zeta_n]$ sending $a$ to $\zeta_n$ takes $1-a^{n'}$ to a unit. Hence
by the universal property of $R$ there is a unique extension $\phi:R
\rightarrow \Z[\frac 1n,\zeta_n]$. This completes the proof.
\end{proof}

In general, we have a map $\Z[a]/(1-a^d)\hookrightarrow R_0$, where
$d:=\gcd(\mu^i_j)$, defined by sending $a$ to
$\prod_{i,j}(a_j^i)^{\mu_j^i/d}$. By the lemma we get $\Z[\tfrac 1d,
\zeta_d] \hookrightarrow R$.

Recall the definitions from \S\ref{subchar}.  Note that up to a
possible reordering of eigenvalues of equal multiplicity a map $\phi:
R \rightarrow \K$ uniquely determines a $k$-tuple of semisimple
generic conjugacy classes $(C_1^\phi,C_2^\phi,\ldots,C_k^\phi)$ of
type $\muhat$ in $\GL_n(\K)$ satisfying \eqref{prod-det} and
conversely ($\calC_i^\phi$ has eigenvalues $\phi(a_j^i)$ of
multiplicities $\mu_j^i$).

Consider the algebra $\calA_0$ over $R$ of polynomials in $n^2(2g+k)$
variables, corresponding to the entries of $n\times n$ matrices
$A_1,\ldots,A_g;B_1,\ldots, B_g;X_1,\ldots,X_k$, with
$$
\det A_1,\ldots,\det A_k;\quad \det B_1,\ldots,\det B_k; \quad \det
X_1,\ldots,\det X_k
$$
inverted. Let $I_n$ be the identity matrix and for  elements $A,B$ of a group
put $(A,B):=ABA^{-1}B^{-1}$.

Define $\calI_0\subset \calA_0$ to be the radical of the ideal
generated by the entries of
$$
(A_1,B_1)\cdots (A_g,B_g) X_1\cdots X_k-I_n,
\qquad
(X_i-a_1^iI_n)\cdots(X_i-a_{r_i}^iI_n), \quad i=1,\ldots,k
$$
and the coefficients of the polynomial
$$
\det(tI_n-X_i)-\prod_{j=1}^{r_i}(t-a_j^i)^{\mu_j^i}
$$
in an auxiliary variable $t$. Finally, let $\calA:=\calA_0/\calI_0$ and
$\U_\muhat:=\Spec(\calA)$.

Let  $\phi:R \rightarrow K$ be a map  to a field $K$ and let
$\U_\muhat^\phi$ be the corresponding base change of $\U_\muhat$
to $K$. A $K$-point of $\U_\muhat^\phi$ is a
solution in $\GL_n(K)$ to
$$
(A_1,B_1) \cdots (A_g,B_g) X_1\cdots X_k=I_n, \qquad X_i\in \calC_i^\phi,
$$
where, as before, $\calC_i^\phi$  is the semisimple conjugacy class in
$\GL_n(K)$
with eigenvalues $\phi(a^i_1),\ldots,\phi(a_{r_i}^i)$ of multiplicities
$\mu^i_1,\ldots,\mu^i_{r_i}$.

Hence, if $\Sigma_g$ is a compact
Riemann surface of genus $g$ with punctures
$S=\{s_1,\ldots,s_k\} \subseteq \Sigma_g$ then  $\U_\muhat^\phi(K)$
can be identified with the set
$$
{\{\rho\in\Hom\left(\pi_1(\Sigma_g\setminus S),\GL_n(K)\right)
\;| \; \rho(\gamma_i)\in \calC_i^\phi}\},
$$
(for some choice of base point, which we omit from the notation). Here
we use the standard presentation
$$
\pi_1(\Sigma_g\setminus S)=\langle
\alpha_1\ldots,\alpha_g;\beta_1\ldots,\beta_g;\gamma_1\ldots,\gamma_g
\;|\;
(\alpha_1,\beta_1)\cdots(\alpha_g,\beta_g)\gamma_1\cdots\gamma_k=1\rangle
$$
($\gamma_i$ is the class  of a simple loop around  $s_i$ with
orientation compatible with that of $\Sigma_g$).

\begin{remark}\label{quiverscheme} A completely analogous construction
 works for the quiver case in the case that $\muhat$ is indivisible
 yielding an affine scheme $\calV_\muhat$ with similar
 properties. For example, in the definition of $R_0$ and $R$ we
 replace the product of elements in a multiset by their sum to
 guarantee genericity (see \ref{genericadjoint}).  The primes
 $p\in\Z$ that become invertible in $R$ are those that are smaller
 than $\min_i\max_j\mu_i^j$ (compare with \eqref{notdivideq}).
\end{remark}

\subsection{Appendix B} Here we prove a version of the smooth-proper
base change theorem. A closely related result was obtained by Nakajima
\cite[Appendix]{crawley-boevey-etal}.
\begin{theorem}\label{ehresmann} Let $X$ be a non-singular complex
 algebraic variety and $f:X\to \C$ a smooth morphism, i.e. a
 surjective submersion.
 Let $\C^\times$ act on $X$ covering a positive power of the standard action on $\C$ such
 that the fixed point set $X^ {\C^\times}$ is complete and for all
 $x\in X$ the $\lim_{\lambda\to 0} \lambda x$ exists. Then the fibers
 have isomorphic cohomology supporting pure mixed Hodge
 structures.\end{theorem}

\begin{proof} The proof is similar to that of \cite[Lemma 6.1]{HT4},
 we give the details  to be self-contained. 	By base change if necessary
	 we can assume that the $\C^\times$-action on $X$ covers the standard action
	on $\C$. Let $\C^\times$ act on
 $\C^2$ by $\lambda(z,w)=(\lambda z,w)$.  Then $\C^2\to \C$ given by
 $(z,w)\mapsto zw$ is $\C^\times$-equivariant with the standard
 action on $\C$. Let now $X^\prime$ denote the base change of $X$ via
 this map, in other words $X^\prime=\{(x,z,w)\in X\times \C^2 |
 f(x)=zw\}$. $X^\prime$ then inherits the $\C^\times$ action given by
 $\lambda(x,z,w)=(\lambda x, \lambda z,w)$. Then $f$ induces the map
 $f^\prime:X^\prime \to \C$ by $f(x,z,w)=w$ which is equivariant with
 respect to the trivial action on the base. By \cite[Theorem
 11.2]{simpson} the set $U\subset X^\prime$ of points $u\in X^\prime$
 such that $\lim_{\lambda\to \infty} \lambda u$ does not exist is
 open and there exists a geometric quotient
 $\overline{X}:=U/\!/\C^\times$ which is proper over $\C$ via the
 induced map $\overline{f}:\overline{X}\to \C$.  Indeed it is a
 completion of $X$ over $\C$ as $X\subset \overline{X}$ naturally by
 the embedding $x\mapsto \C^\times(x,1,f(x))$.

 We now show that $\overline{f}$ is topologically trivial. It is not
 entirely straightforward, as $\overline{X}$ is only an orbifold,
 because the action of $\C^\times$ on $U$ may not be free, there
 could be points with finite stablizers. However the multiplicative
 group $\R^\times_+$ of positive real numbers acts on $U$ as a
 subgroup of $\C^\times$. Therefore the action of $\R^\times$ on $U$
 is free. It is properly discontinuous because the action of
 $\C^\times$ on $U$ is properly discontinuous as $U\to \overline{X}$
 is a geometric quotient.  The quotient space $U/\R^\times_+$ is
 therefore a smooth manifold and the total space of a principal
 $\T:=\rm{U}(1)$ orbi-bundle over the orbifold $\overline{X}$, which
 is proper over $\C$. Hence the induced map $f_+:U/\R^\times_+\to \C$
 is a proper submersion. Thus by choosing a $\T$-invariant Riemannian
 metric on $U/\R^\times_+$ and flowing perpendicular to the
 projection, we find a $\T$-equivariant trivialization of $f_+$ in
 the analytic topology. Dividing out by the $\T$ action yields a
 trivialization of $\overline{f}$ in the analytic
 topology. Consequently the restriction $H^*(\overline{X})\to
 H^*(\overline{X}_w)$ to the cohomology of any fibre of
 $\overline{f}$ is an isomorphism.

Note that $Z:=\overline{X} \setminus X = \{ \C^\times(x,0,w)|
\lim_{\lambda\to \infty} \lambda x \mbox{ exists } \}$ is trivial
over $\C$, therefore $H^*(Z)\to H^*(Z_w)$ is an isomorphism. Applying
the Five Lemma to the long exact sequences of the pairs
$(\overline{X},Z)$ and $(\overline{X}_w,Z_w)$ we get that
$H^*(\overline{X},Z)\cong H^*(\overline{X}_w,Z_w)\cong
H^*_{cpt}(X_w)$. Thus any two fibres of $f$ have isomorphic
cohomology, in particular $H^*_{cpt}(X_w)\cong H^*_{cpt}(X_0)$ for
all $w\in \C$. As $\overline{X}_0$ is a proper orbifold (in
particular a rational homology manifold) \cite[8.2.4]{Del2} implies
that its cohomology has pure mixed Hodge structure. Finally, by
standard Morse theory arguments $H^*(\overline{X}_0)\to H^*(X_0)$ is
surjective thus $H^*(X_0)$ also has pure mixed Hodge structure. The
proof is complete.
 \end{proof}


\begin{thebibliography}{}

{\small

\bibitem{boden} {\sc Boden, H. and Yokogawa, K.}: Moduli spaces of
  parabolic Higgs bundles and parabolic $K(D)$ pairs over smooth
  curves. I. \newblock{Internat. J. Math.} {\bf 7} (1996), no. 5,
  573–598

\bibitem{crawley-quiver}{\sc Crawley-Boevey, W.}: Geometry of the
  Moment Map for Representations of Quivers {\em Comp. Math.} {\bf
    126} (2001) 257--293.

\bibitem{crawley-mat}{\sc  Crawley-Boevey, W.}: On matrices in
prescribed conjugacy classes with no common invariant subspace and
sum zero.  \newblock{\em Duke Math. J.}  {\bf 118}  (2003),  no. 2,
339--352.


\bibitem{crawley-boevey-etal} {\sc Crawley-Boevey, W. {\rm and} Van
    den Bergh, M.}: \newblock Absolutely indecomposable
  representations and Kac-Moody Lie algebras. With an appendix by
  Hiraku Nakajima. \newblock {\em Invent. Math.} {\bf 155} (2004),
  no. 3, 537--559.

\bibitem{crawley-par}{\sc Crawley-Boevey, W.}: Indecomposable
  parabolic bundles and the existence of matrices in prescribed
  conjugacy class closures with product equal to the identity.
  \newblock{\em Publ. Math. Inst. Hautes \'etudes Sci.} (2004),
  no. 100, 171--207.


\bibitem{Del1}{\sc Deligne, P.}:\newblock Th\'eorie de Hodge II. {\em
Inst. hautes Etudes Sci. Publ. Math.} {\bf  40} (1971), 5--47.

\bibitem{Del2}{\sc Deligne, P.}:\newblock Th\'eorie de Hodge III. {\em
Inst. hautes Etudes Sci. Publ. Math.} {\bf  44} (1974), 5--77.



\bibitem{DLu}{\sc  Deligne, P. {\rm and } Lusztig, G.}: \newblock Representations of
reductive groups over finite fields. {\em Ann. of Math. (2)}{\bf
103} (1976), 103--161.



\bibitem{oblomkov-etal} {\sc  Etingof, P.  Gan, W.L. {\rm and} Oblomkov, A.}:
Generalized double affine Hecke algebras of higher rank. {\em J.
Reine Angew. Math.}  {\bf 600} (2006), 177--201.

\bibitem{etingof-et-al} {\sc Etingof, P.,  Oblomkov, A. {\rm and}
    Rains, E.}  Generalized double
affine Hecke algebras of rank 1 and quantized del Pezzo surfaces.
Adv. Math.  {\bf 212}  (2007),  no. 2, 749--796

\bibitem{fricke-klein} {\sc Fricke, R. and Klein, F.}, Vorlesungen
  ¨uber die Theorie der Automorphen Funktionen, 
Teubner, Leipzig, Vol. I, 1912.

\bibitem{frobenius} {\sc  Frobenius, F.G.}: \"Uber Gruppencharacktere (1896),
in {\em Gesammelte Abhandlungen III}, Springer-Verlag, 1968.

\bibitem{garciaetal} {\sc Garc\'ia-Prada, O. Gothen, P. B. and Mu\~noz, V.} Betti numbers of the moduli space of rank 3 parabolic Higgs bundles.  {\em Mem. Amer. Math. Soc.}  {\bf 187}  (2007),  no. 879, viii+80 pp.


\bibitem{getzler}{\sc Getzler, E.}:  Mixed Hodge Structures of configuration spaces,  preprint, arXiv:alg-geom/9510018.

\bibitem{garsia-haiman} {\sc Garsia, A.M. {\rm and} Haiman, M.}: \newblock  A remarkable q,t-Catalan sequence and q-Lagrange  inversion,
\newblock {\em J. Algebraic Combin.} {\bf 5} (1996) no. 3, 191-244.

\bibitem{green} {\sc Green, J.A.}: \newblock
{The characters of the finite general linear groups}. {\em Trans.
Amer. Math. Soc.} {\bf 80} (1955), 402--447.

\bibitem{hanlon} {\sc Hanlon, P.}: \newblock
{The fixed point partition lattices}. {\em Pacific J. Math..} {\bf
96} (1981), 319--341.

\bibitem{hausel-kac}{\sc  Hausel, T.}: Kac conjecture from Nakajima quiver varieties, arXiv:0811.1569

\bibitem{hausel2} {\sc  Hausel, T.}: Arithmetic harmonic analysis, Macdonald polynomials and
the topology of the Riemann-Hilbert monodromy map (with an Appendix
by E. Letellier) (in preparation)


\bibitem{hausel-letellier-villegas2}
{\sc  Hausel T., Letellier, E. {\rm and} Rodriguez-Villegas, F.}:
{Arithmetic harmonic analysis on character and quiver varieties II},
In preparation.

\bibitem{hausel-letellier-villegas3}
{\sc  Hausel T., Letellier, E. {\rm and} Rodriguez-Villegas, F.}: Topology of character varieties and representations of quivers, {\em C. R. Math. Acad. Sci. Paris}   { \bf 348}, No. 3-4, (2010), 131--135.



\bibitem{HT4} {\sc Hausel, T. {\rm and} Thaddeus, M.}:
\newblock Mirror symmetry, {L}anglands duality and {H}itchin systems.
\newblock {\em Invent. Math.}, {\bf 153}, No. 1, (2003), 197-229. arXiv:  math.AG/0205236

\bibitem{hausel-villegas}
{\sc  Hausel, T. {\rm and}  Rodriguez-Villegas, F.}: {Mixed Hodge
polynomials
of character varieties}, {\em Inv. Math.}   { \bf 174}, No. 3, (2008), 555--624, arXiv:{\tt math.AG/0612668}.

\bibitem{helleloid-villegas} {\sc Helleloid, G.T. {\rm and}  Rodriguez-Villegas, F.}: Counting Quiver Representations over Finite Fields Via Graph Enumeration, arXiv:0810.2127

\bibitem{hua} {\sc Hua, J.}:  Counting representations of quivers over finite fields. {\em J. Algebra} {\bf  226} (2000), no. 2, 1011--1033

\bibitem{kacconj} {\sc  Kac, V.}: Root systems, representations of quivers and invariant theory.  {\em Invariant theory (Montecatini, 1982)}, 74--108, {\em Lecture Notes in Mathematics}, {\bf 996}, Springer Verlag 1983



\bibitem{king} {\sc  King, A.D.}: Moduli of representations of finite-dimensional algebras.
{\em Quart. J. Math. Oxford} Ser. (2) {\bf 45} (1994), no. 180, 515--530.

\bibitem{knutson} {\sc Knuston, A. {\rm and } Tao, T.}: The honeycomb model of ${\rm GL}\sb n(\C)$ tensor products. I. Proof of the saturation conjecture.  (English summary)
{\em J. Amer. Math. Soc.} {\bf 12} (1999), no. 4, 1055--1090.

\bibitem{kostov} {\sc  Kostov, V.P.}: On the Deligne-Simpson problem,
C. R. Acad. Sci. Paris S\'er. I Math. {\bf 329} (1999), 657--662.


\bibitem{kostov2} {\sc  Kostov, V.P.}: The Deligne-Simpson problem---a survey.  {\em J. Algebra}  {\bf 281}  (2004),  no. 1, 83--108.

\bibitem{kronheimer-nakajima}
{\sc  Kronheimer, P.~B. {\rm and} ~Nakajima, H.}:
\newblock Yang-{M}ills instantons on {ALE} gravitational instantons.
\newblock {\em Math. Ann.}, 288(2):263--307, 1990.



\bibitem{laumon}{\sc Laumon, G.}:\newblock Comparaison de
caract\'eristiques d'Euler-Poincar\'e en cohomologie $\ell$-adique.
{\em C.R. Acad. Sci. Paris S\'er. I Math.} {\bf 292} (1981), no. 3,
209--212.

\bibitem{Le1}{\sc  Lehrer, G.}: The space of invariant functions on a finite Lie algebra, {\em Trans. Amer. Math. Soc.}, {\bf 348}, no. 1, 31--50.

\bibitem{letellier1}{\sc  Letellier, E.}: Deligne-Lusztig induction of
invariant functions on finite Lie algebras of Chevalley type. {\em
Tokyo J. Math.}

\bibitem{letellier} {\sc  Letellier, E.}: Fourier Transforms of Invariant Functions on Finite Reductive Lie Algebras. {\em Lecture Notes in Mathematics, Vol. {1859}}, Springer-Verlag, 2005.

\bibitem{letellier2}{\sc Letellier, E.}: Quiver varieties and the character ring of general linear groups over finite fields, arXiv:1103.2759.


\bibitem{Lufinite}{\sc  Lusztig, G.}: On the finiteness of the number of unipotent classes, {\em Invent. Math.}, {\bf 34} (1976), 201--213.

\bibitem{lusztigquiver}{\sc  Lusztig, G.}: On quiver varieties, {\em Adv. in Math.} {\bf 136} (1998), 141--182.

\bibitem{LSr}{\sc  Lusztig, G. {\rm and }  Srinivasan, B.}: The characters of the finite unitary groups, {\em Journal of Algebra}, {\bf 49} (1977), 167--171.

\bibitem{macdonald} {\sc Macdonald, I.G}: Symmetric Functions and Hall
Polynomials, {\em Oxford Mathematical Monographs, second ed., Oxford
Science Publications}. The Clarendon Press Oxford University Press,
New York, 1995.



\bibitem{magnus} {\sc Magnus, W.}
Rings of Fricke characters and automorphism groups of free groups.
Math. Z. {\bf 170} (1980), 91--103

\bibitem{nakajima-quiver2}
{\sc ~Nakajima, H.}:
\newblock Quiver varieties and {K}ac-{M}oody algebras.
\newblock {\em Duke Math. J.}, 91(3):515--560, 1998.

\bibitem{saito}
{\sc ~Saito, M.}:
\newblock Mixed {H}odge {M}odules.
\newblock {\em Publ. Res. Inst. Math. Sci.}, {\bf 26}, 221--333 (1990).

\bibitem{simpson}
{\sc  Simpson, C.T.}:
The Hodge filtration on nonabelian cohomology,
{\em Algebraic geometry---Santa Cruz 1995},
Proc.\ Symp.\ Pure Math.\ 62,
ed.\ J. Koll\'ar, R. Lazarsfeld, and D. Morrison,
American Math.\  Soc., 1997.


\bibitem{Sp1}{\sc  Springer, T.}: Caract\`eres d'alg\`ebres de Lie
finies, S\'eminaire Dubreuil, Alg\`ebre, tome 28, no. 1 (1974-1975),
exp. no. 10, p. 1--6.

\bibitem{daisuke} {\sc Yamakawa, D.}: Title: Geometry of Multiplicative Preprojective Algebra, (arXiv:0710.2649)


}

\end{thebibliography}
\end{document}